\documentclass[12pt,a4paper,reqno]{article}
\usepackage{amssymb,amsfonts,amsthm,amsmath,amssymb}
\usepackage{caption}
\usepackage{float}
\usepackage{setspace}
\usepackage{epic}
\usepackage{graphics}
\usepackage{graphicx}
\usepackage[latin2]{inputenc}
\usepackage[T1]{fontenc}
\usepackage{hyperref}
\usepackage{enumerate}
\usepackage{indentfirst}
\usepackage[margin=2.9cm]{geometry}
\allowdisplaybreaks
\DeclareMathOperator{\rk}{rk}
\newcommand{\KL}{Kazhdan-Lusztig\ }
\newtheorem{thm}{Theorem}[section]
\newtheorem{lem}[thm]{Lemma}
\newtheorem{prop}[thm]{Proposition}

\newtheorem{conj}[thm]{Conjecture}

\numberwithin{equation}{section}
\usepackage{tikz}
\usetikzlibrary{calc}
\usetikzlibrary{positioning}
\usetikzlibrary{arrows}
\usetikzlibrary{calc,positioning,arrows,decorations.pathreplacing}
\definecolor{lattice}{RGB}{182,219,219}
\usepackage{mma}
\begin{document}
\begin{center}
{\large \bf  Kazhdan-Lusztig polynomials of fan matroids, wheel matroids and whirl matroids}
\end{center}

\begin{center}
Linyuan Lu$^1$, Matthew H.Y. Xie$^2$ and Arthur L.B. Yang$^{3}$\\[6pt]

$^{1}$Department of Mathematics\\
University of South Carolina, Columbia, SC 29208, USA\\[6pt]

$^{1,2,3}$Center for Combinatorics, LPMC\\
Nankai University, Tianjin 300071, P. R. China\\[6pt]

Email: $^{1}${\tt lu@math.sc.edu},
       $^{2}${\tt xiehongye@163.com},
       $^{3}${\tt yang@nankai.edu.cn}
\end{center}


\begin{abstract}
The \KL polynomial of a matroid was introduced by Elias, Proudfoot and Wakefield, whose properties need to be further explored. In this paper we prove that the \KL polynomials  of fan matroids coincide with Motzkin polynomials, which was recently conjectured by Gedeon. As a byproduct, we determine
the \KL polynomials  of graphic matroids of squares of paths. We further obtain explicit formulas of the \KL polynomials of wheel matroids and whirl matroids. We prove the real-rootedness of the \KL polynomials of these matroids, which provides positive evidence for a conjecture due to Gedeon, Proudfoot and Young. Based on the results on the \KL polynomials, we also determine the $Z$-polynomials of fan matroids, wheel matroids and whirl matroids, and prove their real-rootedness, which provides further evidence in support of a conjecture of Proudfoot, Xu, and Young.
\end{abstract}

\emph{AMS Classification 2010:} 05A15, 05B35, 26C10


\section{Introduction}

In the study of the Hecke algebra of Coxeter groups, Kazhdan and Lusztig \cite{kazhdan1979representations} associated to each pair of group elements an integral polynomial,
now known as the Kazhdan-Lusztig polynomial. In analogy with the classic Kazhdan-Lusztig polynomials, Elias, Proudfoot and Wakefield  \cite{elias2016kazhdan} associated to every matroid
an integral polynomial, which can also be defined for each pair of comparable elements in the lattice of flats.
As noted by Gedeon, Proudfoot and   Young \cite{gedeon2016survey}, both the \KL polynomials of matroids and the classic \KL polynomials can be considered as special cases of the Kazhdan-Lusztig-Stanley functions, first introduced by Stanley \cite{stanley1992subdivisions} and further studied by Brenti \cite{brenti1999twisted, brenti2003p}. Computer experiments suggest that the \KL polynomials of matroids has many special properties such as real-rootedness, as conjectured by Gedeon, Proudfoot and Young \cite{gedeon2016survey}.
However, there still remains challenge for simple matroids such as braid matroids, even to determine the leading coefficients of their \KL polynomials. The main objective of this paper is to determine the \KL  polynomials of fan matroids, wheel matroids  and whirl matroids, and to prove their real-rootedness in support of Gedeon, Proudfoot and Young's conjecture.

Let us give an overview of some background. We begin with the definition of the Kazhdan-Lusztig polynomial of a matroid, as introduced by Elias, Proudfoot and Wakefield  \cite{elias2016kazhdan}. We follow their notation and terminology to a large extent. Given a matroid $M$, let $L(M)$ denote the lattice of flats of $M$ and
let $\chi_M(t)$ denote its characteristic polynomial.
For instance, the lattice of flats of the graphic matroid of a fan graph with four vertices is given in Figure \ref{fig-lattice}.
For any flat $F \in L(M)$, let $M^F$ be the  contraction of $M$ at $F$ and let $M_F$ be the  localization  of $M$ at $F$.
Figure \ref{fig-contraction} gives an illustration of $M^F$ and
$M_F$. Let $\rk M$ denote the rank of $M$. Elias, Proudfoot and Wakefield \cite{elias2016kazhdan} proved that there is a unique way to associate to each loopless matroid $M$ a polynomial $P_M(t) \in \mathbb{Z}[t]$ satisfying the following properties:
\begin{itemize}
\item If $\rk M=0$, then $P_M(t)=1$.

\item If $\rk M>0$, then $\deg P_M(t) < \frac 12 \rk M$.

\item For every $M$, $t^{\rk M}P_M(t^{-1}) = \displaystyle\sum_{F \in L(M)} \chi_{M_F}(t) P_{M^F}(t)$.
\end{itemize}
\begin{minipage}[b]{0.5\textwidth}
\centering
\begin{tikzpicture}[scale=0.5,line width=0.5pt]
\node[shape=circle,fill=lattice,minimum size=3mm,inner sep=0pt] (0) at (0,12) {};
\node [shape=circle,fill=lattice,minimum size=3mm,inner sep=0pt] (12) at (0,0) {};
\foreach \s in {1,...,6}  \node [shape=circle,fill=lattice,minimum size=3mm,inner sep=0pt]  (\s) at ({5-2 * (\s-1)},8) {};
\foreach \s in {7,8,10,11}  \node[shape=circle,fill=lattice,minimum size=3mm,inner sep=0pt]  (\s) at ({3-1.5* (\s-7)},4) {};
\node[shape=circle,fill=lattice,minimum size=3mm,inner sep=0pt]  (9) at (0,4) {};
\node[draw=none,minimum size=3mm,inner sep=0pt]   at (-0.7,3.8) {$F$};
\foreach \s in {7,8}   \draw[->] (12)--(\s);
\foreach \s in {10,...,11}   \draw[->] (12)--(\s);
\draw[->,line width=1pt] (12)--(9);
\foreach \s in {3,4,6} \draw[->] (11)--(\s);
\foreach \s in {5,6}   \draw[->] (7)--(\s);
\foreach \s in {1,2,6}  \draw[->] (10)--(\s);
\foreach \s in {2,4,5} \draw[->,line width=1pt,color=blue] (9)--(\s);
\foreach \s in {1,3,5} \draw[->] (8)--(\s);
\foreach \s in {1,3,6} \draw[->] (\s)--(0);
\foreach \s in {2,4,5} \draw[->,line width=1pt,color=blue] (\s)--(0);
\end{tikzpicture}
\captionsetup{font=footnotesize,type=figure}
\captionof{figure}{$L(M)$}
\label{fig-lattice}
\end{minipage}
\begin{minipage}[b]{0.3\textwidth}
\centering
\begin{tikzpicture}[scale=0.5,line width=0.5pt]
\node[shape=circle,fill=lattice,minimum size=3mm,inner sep=0pt]  (0) at (0,12) { };
\foreach \s in {2,4,5}  \node[shape=circle,fill=lattice,minimum size=3mm,inner sep=0pt]   (\s) at ({5-2 * (\s-1)},8) { };
\node[shape=circle,fill=lattice,minimum size=3mm,inner sep=0pt]   (f1) at (0,4) {};
\node[draw=none,minimum size=3mm,inner sep=0pt]   at (-0.8,3.5) {$F$};
\foreach \s in {2,4,5}   \draw[->,color=blue,line width=1pt] (f1)--(\s);
\foreach \s in {2,4,5}   \draw[->,color=blue,line width=1pt] (\s)--(0);
\node[draw=none,minimum size=3mm,inner sep=0pt]   at (4.2,3.5) {$F$};
\node[shape=circle,fill=lattice,minimum size=3mm,inner sep=0pt]  (f2) at (5,4) {};
\node[shape=circle,fill=lattice,minimum size=3mm,inner sep=0pt]   (12) at (5,-1) { };
\draw[->,line width=1pt] (12)--(f2);
\end{tikzpicture}
\captionsetup{font=footnotesize,type=figure}

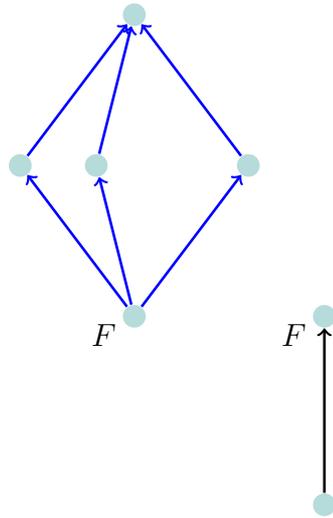
\captionof{figure}{$L(M^F)$ and $L(M_F)$}
\label{fig-contraction}
\end{minipage}


The polynomial $P_M(t)$ is called the \KL polynomial of $M$.
Note that if $M$ and $M'$ are two matriods satisfying $L(M)\cong L(M')$, then $P_M(t)=P_{M'}(t)$. Thus, we may always assume that all matroids appeared in the definition have no parallel elements.

Elias,  Proudfoot and Wakefield \cite{elias2016kazhdan} pointed out that the \KL polynomials for matroids behave very differently from the ordinary \KL polynomials for
Coxeter groups. However, they do possess some similar properties. For example, the coefficients of the classic \KL polynomials  are non-negative,
as proved by  Elias and  Williamson \cite{elias2012hodge}.
Elias,  Proudfoot and Wakefield
\cite{elias2016kazhdan} showed that the coefficients of the \KL polynomials of representable matroids are also non-negative, and they further conjectured that
this is true for any matroids. They also used the defining recursion of $P_M(t)$ to interpret the coefficients of some lower degree terms in terms of
the doubly indexed Whitney numbers of the first and second kinds, introduced by Green and Zaslavsky\cite{greene1983interpretation}.
While, except for the constant term (always equal to $1$) and the linear term, neither the quadratic term nor the cubic term can be proven to be positive by using such explicit
formulas. As noted by Elias, Proudfoot and Wakefield, it is difficult to find a closed formulas for every coefficient by the same method.
Later Wakefield \cite{wakefield2016partial} expressed every coefficient  as an alternating sum of $r$-Whitney  numbers, and Proudfoot, Xu and  Young \cite{proudfoot2017z} gave another combinatorial formula for every coefficient. But neither of these formulas is manifestly positive.

Though the positivity conjecture on the  \KL polynomials for matroids is still open, much work has been focused on determining the  \KL polynomials for specific families of matroids.
Elias,  Proudfoot and Wakefield \cite{elias2016kazhdan}  already showed the \KL polynomial of a finite Boolean matroid is equal to $1$. Furthermore, they proved that the \KL polynomial of a matroid is equal  to $1$  if and only if  its lattice of flats is modular. For uniform matroids, they also obtained a recursive relation among the coefficients of the \KL polynomials, which was transformed into an equivalent functional equation concerning the generating functions of the \KL polynomials. Since uniform matroids are representable matroids, the coefficients of such \KL polynomials should be nonnegative. But the recursive formula given by Elias, Proudfoot and Wakefield seems not helpful to show that the positivity conjecture is valid for general uniform matroids. Based on the recursive formula, Proudfoot, Wakefield and Young \cite{proudfoot2016intersection} gave an explicit formula of the \KL polynomial of the uniform matroid of rank $n-1$ on $n$ elements.
For general uniform matroids, explicit formulas of the \KL polynomials can be deduced from those of the corresponding equivariant \KL polynomials due to Gedeon, Proudfoot and  Young \cite{gedeon2017equivariant}, by which the nonnegativity of the coefficients becomes evident.

Even for graphic matroids, determining the explicit expressions of the \KL polynomials remains a challenge except for few families of graphs.
It is easy to show that the \KL polynomial for any forest is always equal to $1$ since its graphic matroid is isomorphic to a boolean matroid. The \KL polynomial for a cycle graph with $n$ vertices was also known since its graphic matroid is just the uniform matroid of rank $n-1$ on $n$ elements, see Gedeon, Proudfoot and Young \cite{gedeon2016survey}. Gedeon \cite{gedeon2016thagomizer} determined the \KL polynomials
for thagomizer matroids. Gedeon, Proudfoot and Young \cite{gedeon2016survey} further determined the \KL polynomials for complete bipartite graphs with one part having exactly two vertices. The \KL polynomials for braid matroids, which are graphic matroids associated with complete graphs, were first studied by Elias,  Proudfoot and  Wakefield \cite{elias2016kazhdan}. For these polynomials, they explicitly determined the coefficients of some terms of degree less than four in terms of Stirling numbers. However, there are no conjectured explicit formulas for general coefficients of the \KL polynomials for braid matroids, and even their leading coefficients are still mysterious, see \cite{elias2016kazhdan, gedeon2016survey}.

In this paper we obtain explicit expressions of the \KL polynomials for some matroids associated with square of paths, fan graphs and wheel graphs. Recall that a square of a path
with $n$ ($n\geq 1$) vertices, denoted by $S_n$, is the graph formed by joining every pair of vertices of distance two in the path. A fan graph $F_n$  ($n\geq 1$) is a graph with $n+1$ vertices, formed by connecting a single vertex to all vertices of a path of $n$ vertices. A wheel  graph $W_n$  ($n\geq 3$) is a graph with $n+1$ vertices, formed by connecting a single vertex to all vertices of an $n$-edge cycle graph.
Let $P_{S_n}(t),P_{F_n}(t)$ and $P_{W_n}(t)$ denote the \KL polynomials of the associated graphic matroids. Let $P_{W^n}(t)$ denote the \KL polynomial of the whirl matroid $W^n$, which is obtained from the graphic matroid  of $W_n$ by declaring the outer cycle of $W_n$  to be an independent set and leaving the remaining independent sets the same. Gedeon \cite{gedeon2017preparation} conjectured that the \KL polynomials $P_{F_n}$ coincide with Motzkin polynomials. The first main result of this paper is as follows, which in particular provides an affirmative answer to Gedeon's conjecture.

\begin{thm}\label{klcoef}
 The \KL polynomials $P_{F_n}(t), P_{S_n}(t), P_{W_n}(t)$ and $P_{W^n}(t)$ are respectively given by
 \begin{align}
P_{F_n}(t)&=\sum_{k=0}^{\lfloor \frac{n-1}{2}\rfloor} {\frac{1}{k+1}\binom{n-1}{k,k,n-2k-1} t^k},\quad \mbox{for } n\geq 1\label{eq-klpol-fan}\\[5pt]
P_{S_n}(t)&=\sum_{k=0}^{\lfloor \frac{n-1}{2}\rfloor} {\frac{1}{k+1}\binom{n-1}{k,k,n-2k-1} t^k},\quad \mbox{for } n\geq 1\label{eq-klpol-square}\\[5pt]
P_{W_n}(t)&=\sum_{k=0}^{\left\lfloor \frac{n-1}{2}\right\rfloor} {
\left(\frac{k+1}{n-k} +\frac{k}{n-k+1}-\frac{k}{n-k-1} \right)\binom{n}{k,k+1,n-2k-1}t^{k} },\quad \mbox{for } n\geq 3\label{eq-klpol-wheel}\\[5pt]
P_{W^n}(t)&=\sum _{k=0}^{\left\lfloor \frac{n-1}{2}\right\rfloor }{\frac{n}{n-k}\binom{n-1}{k, k, n-2k-1} t^k}, \quad \mbox{for } n\geq 3\label{eq-klpol-whirl}
\end{align}
  where $\lfloor x \rfloor$ stands for the smallest integer less than or equal to $x$.
\end{thm}

The second part of this paper is devoted to the study of the real-rootedness of the \KL polynomials $P_{F_n}(t), P_{S_n}(t), P_{W_n}(t)$ and $P_{W^n}(t)$. This was motivated by the following conjecture due to Gedeon, Proudfoot and  Young \cite{gedeon2016survey}.

\begin{conj}[{\cite[Conjecture 3.2]{gedeon2016survey}}]\label{klconjroot}
The \KL polynomial $P_M(t)$  has only negative zeros for any matroid $M$.
\end{conj}

A weaker conjecture than Conjecture \ref{klconjroot} was proposed by Elias,  Proudfoot, Wakefield \cite[Conjecture 2.5]{elias2016kazhdan}, which states that for any matroid $M$ the
\KL polynomial $P_M(t)$ is a log-concave polynomial with no internal zeros. Recall that a polynomial
$$f(t)=a_0+a_1t+\cdots +a_nt^n$$
with real coefficients is said to be log-concave if ${a_i}^2 \geq  a_{i-1} a_{i+1}$ for any $0<i<n$, and it is said to have no internal zeros if there are not three indices $0 \leq i < j < k \leq n$ such that $a_i,a_k \not  = 0$ and $a_j = 0$. By the well known Newton inequality, if $f(t)$ has only negative zeros,
then it must be a log-concave polynomial without internal zeros.

Gedeon, Proudfoot and  Young \cite{gedeon2016survey} also studied the interlacing property concerning the \KL polynomials for non-degenerate matroids. A matroid $M$ is called  non-degenerate  if $\rk M=0$  or its \KL polynomial $P_{M}(t)$ is of degree $\lfloor  \frac{\rk M-1}{2}\rfloor$. Given two real-rooted polynomials $f(t)$ and $g(t)$ with positive leading coefficients, let $\{u_i\}$ be the set of zeros of $f(t)$ and let $\{v_j\}$ be the set of zeros of $g(t)$. We say that {$g(t)$ is an interleaver of $f(t)$}, denoted $g(t) \preceq  f(t)$, if either of the following two conditions is satisfied:
\begin{itemize}
\item[(1)] $\deg f(t)=\deg g(t)=n$ and
\begin{align}\label{alt-def}
v_n\le u_n\le v_{n-1}\le\cdots\le v_2\le u_2\le v_1\le u_1;
\end{align}

\item[(2)]
$\deg f(t)=\deg g(t)+1=n$ and
\begin{align}
u_{n}\le v_{n-1}\le\cdots\le v_{2}\le u_{2}\le v_{1}\le u_{1}.\label{int-def}
\end{align}
\end{itemize}
For the second case, we usually say that $g(t)$ interlaces $f(t)$, which is different from the definition given in \cite{gedeon2016survey}. Gedeon, Proudfoot and  Young
 \cite{gedeon2016survey} proposed the following conjecture.

\begin{conj}[{\cite[Conjecture 3.4]{gedeon2016survey}}] \label{conjinterlaced}
Given a matroid $M$ and an element $e$ of the ground set of $M$,
let $M/e$ be the contraction of $M$ at $e$. If both $M$ and $M/e$  are non-degenerate, then $P_{M/e}(t)\preceq P_M(t)$.
\end{conj}

Despite strong interest, there were very few results
about Conjectures \ref{klconjroot} and \ref{conjinterlaced}.
For any $n\geq 1$, it was known that Conjectures \ref{klconjroot} and \ref{conjinterlaced} are both valid for the uniform matroid of rank $n-1$ on $n$ elements, see \cite{gedeon2016survey,zhang2016multiplier}.
Based on the theory of multiplier sequences and $n$-sequences (see \cite{craven1977multiplier,craven1983location1}),
we prove that Conjecture \ref{klconjroot} holds for fan matroids, wheel matroids and whirl matroids, and Conjecture \ref{conjinterlaced} holds for fan matroids.

\begin{thm}\label{klroot}
For $n\geq 3$, each of the  \KL  polynomials $P_{F_n}(t), P_{S_n}(t), P_{W_n}(t)$ and $P_{W^n}(t)$ has only negative zeros. In particular, we have
$P_{F_n}(t) \preceq P_{F_{n+1}}(t)$.
\end{thm}


The third part of this paper is concerned with the $Z$-polynomials for fan matroids, wheel matroids and whirl matroids. The notion of the $Z$-polynomial of a matroid
was introduced by Proudfoot, Xu and  Young \cite{proudfoot2017z}.
Given a matroid $M$, its $Z$-polynomial is defined by
$$Z_M(t):= \sum_{F\in L(M)}{t^ {\rk M_F} P_{M^F}(t)},$$
Analogous to the real-rootedness conjecture for  the \KL polynomials of matroids, Proudfoot, Xu and  Young \cite{proudfoot2017z} posed he following conjecture.

\begin{conj}[{\cite[Conjecture 5.1]{proudfoot2017z}}]\label{zconjroot}
The $Z$-polynomial $Z_M(t)$ has only negative zeros for any matroid $M$.
\end{conj}

Proudfoot, Xu and Young obtained explicit formulas of the $Z$-polynomials for
the uniform matroid of rank $n-1$ on $n$ elements and the matroid represented by all vectors of the vector space
$\mathbb{F}_q^n$, as well as the real-rootedness of these $Z$-polynomials. Let $Z_{F_n}(t), Z_{W_n}(t), Z_{W^n}(t)$ denote the $Z$-polynomials corresponding to the fan graph $F_n$, the wheel graph $W_n$ and the whirl matroid $W^n$ respectively.
In this paper we obtain explicit formulas of $Z_{F_n}(t), Z_{W_n}(t)$ and $Z_{W^n}(t)$ as given below.

\begin{thm}\label{zcoef}
We have
 \begin{align}
Z_{F_n}(t)&=\sum _ {k = 0}^{n}\frac{1}{n+1} \binom {n + 1} {k+1}\binom {n + 1} {k}  t^k,\quad \mbox{for } n\geq 1\label{eq-zpol-fan-sect1}\\[5pt]
Z_{W_n}(t)&= \sum _ {k = 0}^n \left (\binom {n} {k}^2 - \frac {2}{n} \binom {n} {k +1}\binom {n} {k - 1}\right)t^k,\quad \mbox{for } n\geq 3\label{eq-zpol-wheel-sect1}\\[5pt]
Z_{W^n}(t)&=\sum _ {k = 0}^n  \binom {n} {k}^2  t^k, \quad \mbox{for } n\geq 3.\label{eq-zpol-whirl-sect1}
\end{align}
\end{thm}

Based on the above formulas, we further show that Conjecture \ref{zconjroot} holds for fan matroids, wheel matroids and whirl matroids. Precisely, we have the following result.

\begin{thm}\label{zroot}
The $Z$-polynomials $Z_{F_n}(t), Z_{W_n}(t)$ and $Z_{W^n}(t)$  are all real-rooted.
\end{thm}

This paper is organized as follows. In Section \ref{sec-2} we will give an alternative description of the \KL  polynomial of a graphic matroid directly in the language of graph theory.
In Section \ref{sec-2.5}, we will recall some known results on the ordinary generating functions. We also give a variation of the composition formula with respect to cycles, which turns out to be new but very useful. In Section \ref{sec-3}, we first determine the \KL polynomials for fan matroids by using the generating function technique. Based on this result, we further obtain explicit formulas of the \KL polynomials for wheel matroids and whirl matroids. By Whitney's $2$-isomorphism theorem, we also determine the \KL polynomials for squares of paths. In Section \ref{sec-4} we will prove the real-rootedness of the \KL polynomials obtained in Section $3$ by using the theory of multiplier sequences and the theory of $n$-sequences. In Section \ref{sec-5} we get closed formulas for the $Z$-polynomials of fan matroids, wheel matroids and whirl matroids, based on which we verified the validity of Conjecture \ref{zconjroot} for these matroids.

\section{Graphic matroids}\label{sec-2}

In this section we aim to give an alternative description of
the Kazhdan-Lusztig polynomial, as well as that of the $Z$-polynomial, for a graphic matroid by directly using the language of graph theory. As will be shown later, such a formulation is convenient for computing the \KL polynomials and $Z$-polynomials for fan matroids, wheel matriods and whirl matroids.

Given a loopless graph $G$ with edge set $E(G)$ and vertex set $V(G)$, let $M(G)$ denote the corresponding graphic matroid.
It is known that the ground set of  $M(G)$ is $E(G)$ and its independent sets are the forests in $G$.
To reformulate the definition of  the \KL  polynomial of ${M(G)}$, we first recall how to describe its rank, flats and characteristic polynomial in graph theory terminology. Recall that the rank of $M(G)$ is equal to the rank of $G$, denoted by $\rk G$, which is defined as the cardinality of $V(G)$ minus the number of connected components of $G$.
A flat of $M(G)$ can be identified with a partition of $V(G)$ into vertex sets of connected induced subgraphs of $G$. Such a partition was called a composition of $G$ in \cite{knopfmacher2001graph}, and we use $\mathcal{C}(G)$ to denote the set of all compositions of $G$.
Given a flat $F$ of $M(G)$, the localization $M(G)_F$ and the contraction $M(G)^F$ are naturally identified with graphic matroids as follows:
\begin{align}\label{eq-submatr}
M(G)_F =  M(G[F]),\quad M(G)^F = M(G/F),
\end{align}
where $G[F]$ is the subgraph of $G$ induced by $F$ and $G/F$ is the graph obtained from $G$ by contracting all edges of $F$ (in any order), see \cite[p.~61 and p.~63]{welsh1976matroid}.
If $G$ is a graph with $k$ connected components, then we have
\begin{align}\label{eq-c-chr}
\chi_{M(G)}(t)={t^{-k}}{\chi_G(t)},
\end{align}
where $\chi_G(t)$ is the chromatic polynomial of $G$, see
\cite[p.~262]{welsh1976matroid}.
If a given composition $C\in \mathcal{C}(G)$ corresponds to a flat $F$ of $M(G)$, then we also use $G[C]$ to denote the graph $G[F]$ and use $G/C$ to denote the graph $G/F$. Further, suppose that $C=\{V_1,V_2, \ldots, V_{|C|}\}$, and then it is readily to see that $G[C]$ is just the graph union $\cup_{i=1}^{|C|} G[V_i]$, where each $G[V_i]$ is the subgraph of $G$ induced by $V_i$.
For example,  given a graph $G$ as shown in  Figure  \ref{fig:gc}, let  $C=\{ \{1,3,4\}, \{2,5\} ,\{6,7,10\},\{8\} ,\{9,11,12\}  \}$. Then $E(G[C])$  is highlighted by bold lines  in  Figure  \ref{fig:gc} and $G/C$ is isomorphic to the graph  in  Figure  \ref{fig:g/c}.

\begin{minipage}[b]{0.5\textwidth}
\centering
\begin{tikzpicture}[line cap=round,line join=round,>=triangle 45]
\draw (-0.5,3)  node[circle,fill,inner sep=1pt,label=above  :1]   (1) {} ;
\draw (1.5,3)  node[circle,fill,inner sep=1pt,label=above  :2]   (2) {} ;
\draw (-1,2)  node[circle,fill,inner sep=1pt,label=above  :3]   (3) {} ;
\draw (0.5,2) node[circle,fill,inner sep=1pt,label=above  :4]   (4) {} ;
\draw (2.5,2)  node[circle,fill,inner sep=1pt,label=above  :5]   (5) {} ;
\draw (-2,1)    node[circle,fill,inner sep=1pt,label=above  :6]   (6) {} ;
\draw (-1,1)   node[circle,fill,inner sep=1pt,label=above  :7]   (7) {} ;
\draw  (0.5,1)  node[circle,fill,inner sep=1pt,label=below right  :8]   (8) {} ;
\draw (2,1)   node[circle,fill,inner sep=1pt,label=right :9]   (9) {} ;
\draw (-1.5,0)  node[circle,fill,inner sep=1pt,label=below  :10]   (10) {};
\draw (0.,0.)   node[circle,fill,inner sep=1pt,label=below  :11]   (11) {};
\draw (1.5,0)  node[circle,fill,inner sep=1pt,label=below  :12]   (12) {};
\draw  (11)-- (8)  (9)-- (8)  (11)--(10) (11)-- (3) (8)-- (1)  (4)-- (5)   (2)-- (4)  (8)-- (2)   (7)-- (8)  ;
\draw[line width=1.5pt,color=blue]  (1)-- (3) (4)-- (3) (1)-- (4)   ;
\draw[line width=1.5pt,color=blue]  (6)-- (10) (6)-- (7)  (10)-- (7)   ;
\draw[line width=1.5pt,color=blue]  (11)-- (12)   (12)-- (9)  ;
\draw[line width=1.5pt,color=blue]   (5)-- (2) ;
\end{tikzpicture}
\captionsetup{font=footnotesize,type=figure}

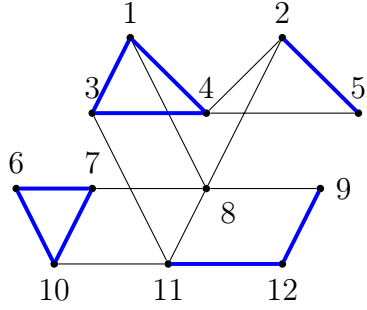
\captionof{figure}{G and  G[C]}
\label{fig:gc}
\end{minipage}
\begin{minipage}[b]{0.4\textwidth}
\centering
\begin{tikzpicture}[line cap=round,line join=round,>=triangle 45]
\draw (-0.5,1.5) node[circle,fill,inner sep=1pt]  (1) {};
\draw (-1.3,1.5) node[draw=none,inner sep=1pt]  {\{1,3,4\}};
\draw (1.5,1.5)  node[circle,fill,inner sep=1pt]   (2) {} ;
\draw (2.2,1.5)  node[draw=none,inner sep=1pt]   {\{2,5\}} ;

\draw (-1,0.5)  node[circle,fill,inner sep=1pt]   (3) {} ;
\draw (-1.9,0.5)  node[draw=none,inner sep=1pt]   {\{6,7,10\}} ;
\draw (0.5,0.5)  node[circle,fill,inner sep=1pt,label=right :$\{8\}$] (4) {} ;

\draw (0.5,-0.5) node[circle,fill,inner sep=1pt]   (5) {} ;
\draw (0.5,-0.9)  node[draw=none,inner sep=1pt]   {\{9,11,12\}} ;

\draw (4)--(5)   (4)-- (1)   (1)--(2)    (4)-- (2)    (3)-- (4)  (1)--(5) (3)--(5) ;
\end{tikzpicture}
\captionsetup{font=footnotesize,type=figure}
\captionof{figure}{G/C}
\label{fig:g/c}
\end{minipage}


Now we can give an alternative description of the \KL  polynomial of $M(G)$. For notational convenience, we denote the \KL  polynomial of $M(G)$ by $P_G(t)$ instead of $P_{M(G)}(t)$.
We also say that $P_G(t)$ is the \KL polynomial of $G$.
By the preceding arguments, the polynomial $P_G(t)$ can be defined
in the following way:
\begin{itemize}
\item If $\rk G=0$, that is $E(G)=\emptyset$ , then $P_G(t)=1$.

\item If $\rk G>0$, then $\deg P_G(t) < \frac {1}{2} \rk G$, and moreover
    \begin{align}\label{klpol-reform}
    t^{\rk G}P_G(t^{-1}) = \sum_{ C\in \mathcal{C}(G)} t^{-|C|} \chi_{G[C]}(t) P_{G/C}(t).
    \end{align}
\end{itemize}
Note that an edge contraction operation may result in a graph with multiple edges even if the original graph was a simple graph. Since for any graph $G$ and its associated simple graph $G'$ we have $P_{G}(t)=P_{G'}(t)$, we always assume that all graphs appeared in the above definition are simple. In this sense we can consider $G/C$ as a quotient graph by  identifying  vertices in each block  of $C$.

We proceed to show the multiplicativity of the \KL polynomial $P_G(t)$, which will be frequently used in Section \ref{sec-3}.
Recall that a connected graph is called biconnected if the resulting graph remains connected whenever any vertex is removed.
A biconnected component of a graph is a maximal biconnected subgraph. We have the following result.

\begin{lem}\label{graphsum}
Suppose that $G$ is a graph with $k$ connected components and $m$ biconnected components, say $\{G_1,\ldots,G_m\}$. Then we have
\begin{align}
P_{G}(t)&=\prod_{i=1}^{m}{P_{G_i}(t)},\label{multi-klpol}\\
\chi_{G}(t)&=t^{-(m-k)}\prod_{i=1}^{m}{\chi_{G_i}(t)}.\label{multi-chrpol}
\end{align}
\end{lem}

\begin{proof}
By definition we know that  $M(G)=\oplus_{i=1}^{m}{M(G_i)}$.
It is known that  for any matroids $M_1$ and  $M_2$ there holds $P_{M_1\oplus M_2}(t)=P_{M_1}(t)P_{M_2}(t)$, see \cite[Proposition 2.7]{elias2016kazhdan}.  Iteration of  this property leads to \eqref{multi-klpol}.
It is also known that the characteristic polynomial $\chi_{M(G)}(t)$ is also multiplicative on direct sums, see \cite[p.~265]{welsh1976matroid}.
Thus, we have
$$\chi_{M(G)}(t)=\prod_{i=1}^{m}{\chi_{M(G_i)}(t)}.$$
Combining the above identity and \eqref{eq-c-chr},
we immediately obtain \eqref{multi-chrpol}.
\end{proof}

Next we give an alternative description of the
the $Z$-polynomial of $M(G)$. For notational convenience, we denote the $Z$-polynomial of $M(G)$ by $Z_G(t)$ instead of $Z_{M(G)}(t)$. As before, we also say that $Z_G(t)$ is the $Z$-polynomial of $G$. Using the language of graph theory, the $Z$-polynomial of $G$ can be defined by
\begin{align}\label{zpol-reform}
Z_G(t):= \sum_{ C\in \mathcal{C}(G)} {t^ {\rk G[C]} P_{G/C}(t)}.
\end{align}

From the multiplicativity of $P_G(t)$ it follows that the
$Z$-polynomial $Z_G(t)$  also admits the same property.

\begin{lem}\label{graphsum-z}
Suppose that $G$ is a graph with $k$ connected components and $m$ biconnected components, say $\{G_1,\ldots,G_m\}$. Then we have
\begin{align}
Z_{G}(t)&=\prod_{i=1}^{m}{Z_{G_i}(t)}.\label{multi-zpol}
\end{align}
\end{lem}

\section{Ordinary generating functions}\label{sec-2.5}

In this section we aim to give a variation of the usual composition formula of ordinary generating functions, which will be used to determine the ordinary generating functions of the \KL polynomials of wheel graphs.

Let us first give an overview of the combinatorial significance of the derivative formula, the product formula and the composition formula of ordinary generating functions.

Suppose that we are interested in enumerating a family $\mathcal{A}$ of combinatorial structures. For each $a\in \mathcal{A}$, let $|a|$ denote its size. If $|a|=n$, we also say that $a$ is a combinatorial structure of type $\mathcal{A}$
which is built on an interval $I$ of size $n$. Let
$\mathcal{A}_n$ denote the set of combinatorial structures of type $\mathcal{A}$ which can be built on an interval $I$ of size $n$. We always assumed that the allowed structures depend only on the size of the interval $I$. If each element $a$ of $\mathcal{A}$ is weighted by ${w_{\mathcal{A}}}(a)$, then define the ordinary generating function of $\mathcal{A}$ with respect to the weight function $w_{\mathcal{A}}$ as the formal power series
$$A(u)=\sum_{a\in \mathcal{A}} w_{\mathcal{A}}(a) u^{|a|},$$
or equivalently,
$$A(u)=\sum_{n \in \mathbb{A} } \left(\sum_{a\in \mathcal{A}_n}w_{\mathcal{A}}(a)\right) u^{n},$$
where  $\mathbb{A}=\{n\in \mathbb{N}\, |\, \mathcal{A}_n\neq \emptyset\}$.

The derivative of the ordinary generating function $A(u)$ admits a very natural combinatorial interpretation, as illustrated below. For $n\in \mathbb{A}\backslash \{0\}$ let $\mathcal{\dot{A}}_n$ denote the set of pointed type $\mathcal{A}$ structures built on an interval $I$ of size $n$ by a special ``pointer'' of size $0$ attached to an element of $I$.
Let $\mathcal{\dot{A}}=\cup_{n=1}^{\infty} \mathcal{\dot{A}}_n$.
Each pointed structure $\dot{a}$ will have the same weight as the underlying structure $a$. For this reason, we still denote the weight function of $\mathcal{\dot{A}}$ by $w_{\mathcal{A}}$.
 Note that the coefficient of $u^n$ in $A(u)$, denoted by $[u^n]A(u)$, is the weighted sum of structures in $\mathcal{A}_n$. Thus $n[u^n]A(u)$ can be interpreted as the weighted sum of structures in $\mathcal{\dot{A}}_n$.
 Therefore, we have the following result.

\begin{prop}\label{gf-pointed}
The ordinary generating function of $\mathcal{\dot{A}}$ is given by
$$\dot{A}(u)=\sum_{n \in \mathbb{A}\backslash \{0\}} \left(\sum_{\dot{a}\in \mathcal{\dot{A}}_n}w_{\mathcal{A}}(\dot{a})\right) u^{n}=uA'(u).$$
\end{prop}

Given two types of combinatorial structures, say $\mathcal{A}$ and $\mathcal{B}$, denote their weight functions respectively by $w_{\mathcal{A}}$ and $w_{\mathcal{B}}$, and denote their ordinary generating functions respectively by $A(u)$ and $B(u)$.
Suppose that $\mathcal{C}$ is another type of combinatorial structures with the weight function $w_{\mathcal{C}}$, which is built from $\mathcal{A}$ and $\mathcal{B}$ in the following way: a structure $c$ of type $\mathcal{C}$ built on an interval $I$ is obtained by first splitting $I$
into two intervals $I_1$ and $I_2$, and then building a structure $a$ of type $\mathcal{A}$ on $I_1$ and a structure $b$ of type $\mathcal{B}$ on $I_2$; and moreover $w_{\mathcal{C}}(c)=w_{\mathcal{A}}(a)w_{\mathcal{B}}(b)$. Here $I_1$ and $I_2$ are allowed to be empty. If $\mathcal{C}$ is obtained in such a way, we also write $\mathcal{C}=\mathcal{A}\times \mathcal{B}$.
The product formula of ordinary generating functions is stated as follows.

\begin{prop}\label{prop-product-formula} Suppose that $\mathcal{A},\mathcal{B},\mathcal{C}$ are three types of combinatorial structures such that $\mathcal{C}=\mathcal{A}\times \mathcal{B}$. If we let
\begin{align*}
A(u)&=\sum_{n=0}^{\infty} \left(\sum_{a\in \mathcal{A}_n}w_{\mathcal{A}}(a)\right) u^{n},\\
B(u)&=\sum_{n=0}^{\infty} \left(\sum_{b\in \mathcal{B}_n}w_{\mathcal{B}}(b)\right) u^{n},\\
C(u)&=\sum_{n=0}^{\infty} \left(\sum_{c\in \mathcal{C}_n}w_{\mathcal{C}}(c)\right) u^{n},
\end{align*}
then
$$C(u)=A(u)B(u).$$
\end{prop}

Next we consider the composition formula of ordinary generating functions. Let $\mathcal{A}$ and $\mathcal{B}$ be two types of combinatorial structures as before. Note that for a type $\mathcal{B}$ structure $b$ of size $n$ we may also say that $b$ is built on an ordered $n$-element set since the allowed structures only depend on the size of the interval.
Suppose that $\mathcal{C}$ is a combinatorial structure with the weight function $w_{\mathcal{C}}$, which is built from $\mathcal{A}$ and $\mathcal{B}$ in the following way: a structure $c$ of type $\mathcal{C}$ built on an interval $I$ is obtained by first splitting $I$  into an unspecified number of nonempty intervals, say $I_1,I_2,\ldots,I_k$ listed from left to right, then building a structure $a_{j}$ of type $\mathcal{A}$ on each interval  $I_j$, and then building a type $\mathcal{B}$ structure $b$ on ordered intervals $(I_1,I_2,\ldots,I_k)$; and moreover $w_{\mathcal{C}}(c)=w_{\mathcal{B}}(b) \prod_{j=1}^k{w_{\mathcal{A}}(a_j)} $. If $\mathcal{C}$ is obtained in such a way, we also write $\mathcal{C}=  \mathcal{B} \circ \mathcal{A}$.
The composition formula of ordinary generating functions is stated as follows.

\begin{prop}\label{prop-composition-formula} Suppose that $\mathcal{A},\mathcal{B},\mathcal{C}$ are three types of combinatorial structures such that $\mathcal{C}=\mathcal{B} \circ \mathcal{A}$ and $0 \not \in \mathbb{A} $, where  $\mathbb{A}=\{n\in \mathbb{N}\, |\, \mathcal{A}_n\neq \emptyset\}$ as defined before.
If we let
\begin{align*}
A(u)&=\sum_{n \in \mathbb{A} } \left(\sum_{a\in \mathcal{A}_n}w_{\mathcal{A}}(a)\right) u^{n},\\
B(u)&=\sum_{n=0}^{\infty} \left(\sum_{b\in \mathcal{B}_n}w_{\mathcal{B}}(b)\right) u^{n},\\
C(u)&=\sum_{n=0}^{\infty} \left(\sum_{c\in \mathcal{C}_n}w_{\mathcal{C}}(c)\right) u^{n},
\end{align*}
then
$$C(u)=B(A(u)).$$
\end{prop}


We proceed to consider a variation of the above composition formula, which is very useful for determining the ordinary generating functions of certain combinatorial structures related to cycles. From two types of combinatorial structures $\mathcal{A}$ and $\mathcal{B}$ we construct the third family of combinatorial structures $\mathcal{C}$, which will build
a structure $c$ on an interval $I$ of size $n$ in the following way: first decompose a cycle of length $n$ into at least two nonempty segments, say $\{I_1,I_2,\ldots,I_k\}$, then build  a structure $a_{j}$ of type $\mathcal{A}$ for each $I_j$ (considered as an interval), and then build  a type $\mathcal{B}$ structure $b$ on ordered segments $(I_1,I_2,\ldots,I_{k})$; and moreover $w_{\mathcal{C}}(c)=w_{\mathcal{B}}(b) \prod_{j=1}^{k}{w_{\mathcal{A}}(a_j)} $. If $\mathcal{C}$ is obtained in such a way, we also write $\mathcal{C}=  \mathcal{B} \bullet  \mathcal{A}$. We have the following result.

\begin{prop}\label{prop-composition-formula-variant}
Suppose that $\mathcal{A},\mathcal{B},\mathcal{C}$ are three types of combinatorial structures such that $\mathcal{C}=\mathcal{B} \bullet \mathcal{A},0 \not \in \mathbb{A} $ and $0,1 \not \in \mathbb{B}$, where  $\mathbb{A}=\{n\in \mathbb{N}\, |\, \mathcal{A}_n\neq \emptyset\}$ and $\mathbb{B}=\{n\in \mathbb{N}\, |\, \mathcal{B}_n\neq \emptyset\}$. If we let
\begin{align*}
A(u)&=\sum_{n \in \mathbb{A} } \left(\sum_{a\in \mathcal{A}_n}w_{\mathcal{A}}(a)\right) u^{n},\\
B(u)&=\sum_{n\in \mathbb{B}} \left(\sum_{b\in \mathcal{B}_n}w_{\mathcal{B}}(b)\right) u^{n},\\
C(u)&=\sum_{n=2}^{\infty} \left(\sum_{c\in \mathcal{C}_n}w_{\mathcal{C}}(c)\right) u^{n},
\end{align*}
then
$$C(u)=uA'(u)\frac{B(A(u))}{A(u)}  .$$
\end{prop}

\begin{proof}
First we give an alternative description of cycle decompositions involved in type $\mathcal{C}$ structures.
Let $\mathcal{C}^{\langle 1 \rangle}_n$ denote the set of cycle decompositions of an $n$-cycle (with vertices labelled $1,2,\ldots,n$) into at least two segments.
Let $\mathcal{C}^{\langle 2 \rangle}_n$ denote the set of weak integer compositions of $n$, say $(\alpha_0,\alpha_1,\ldots,\alpha_{k-1},\alpha_{k})$, satisfying $k\in \mathbb{B},a_1\geq 1,a_0+a_1 \in \mathbb{A}$ and $\alpha_i \in \mathbb{A} $ for $2\leq i\leq k$.
We proceed to show that there is a one to one correspondence between $\mathcal{C}^{\langle 1 \rangle}_n$ and $\mathcal{C}^{\langle 2 \rangle}_n$ for each $n\geq 2$. The case of $n=2$ or $\mathcal{C}^{\langle 1 \rangle}_n  =\mathcal{C}^{\langle 2\rangle}_n = \emptyset$ is obvious. For $n\geq 3$ and $\mathcal{C}^{\langle 1 \rangle}_n \not = \emptyset$,
we construct a bijective map  $\phi$ from $\mathcal{C}^{\langle 1 \rangle}_n$ to $\mathcal{C}^{\langle 2 \rangle}_n$. For a given element $C$ in $\mathcal{C}^{\langle 1 \rangle}_n$ containing $k\geq 2$ segments, say $\{I_1, I_2,\ldots,I_k\}$ arranged clockwise in the cycle with $1\in I_1$, we can define $\phi(C)$ as follows.
Then let $\phi(C)=(\alpha_0,\alpha_1,\ldots,\alpha_{k-1},\alpha_{k})$,
where $\alpha_0$ is the integer such that $n-\alpha_0$ is the maximal number in $I_k$, $\alpha_1=|I_1|-\alpha_0$ and $\alpha_i=|I_i|$ for $2 \leq i \leq k$. It is clear that $\phi(C)\in \mathcal{C}^{\langle 2 \rangle}_n$ and the map $\phi$ is injective. Conversely, given a weak composition $(\alpha_0,\alpha_1,\ldots,\alpha_{k-1},\alpha_{k})\in \mathcal{C}^{\langle 2 \rangle}_n$, we can define a cycle decomposition $C'$ with $k$ segments $I_1,\ldots,I_k$ such that $I_1=[n-\alpha_0+1,n]\cup[1,\alpha_1]$ and
$I_i=[\alpha_1+\cdots+\alpha_{i-1}+1,\alpha_1+\cdots+\alpha_{i}]$ for
$2\leq i\leq k$. It is not difficult to show that $\phi(C')$ is just $(\alpha_0,\alpha_1,\ldots,\alpha_{k-1},\alpha_{k})$. Thus $\phi$ is also surjective.

To prove the desired result, we shall introduce three more families of combinatorial structures.
Let $\mathcal{\dot{A}}$ be the set of pointed type $\mathcal{A}$ structures defined as before.
From $\mathcal{{B}}$ we may define another family of combinatorial structures, denoted by $\mathcal{\bar{B}}$, simply by requiring that $\mathcal{\bar{B}}_{n}=\mathcal{{B}}_{n+1}$ for $n\geq 0$. (The equality $\mathcal{\bar{B}}_{n}=\mathcal{{B}}_{n+1}$ means that all possible type $\mathcal{\bar{B}}$ structures built on an interval of size $n$ are exactly those possible type $\mathcal{{B}}$ structures built on an interval of size $n+1$.) A type $\mathcal{\bar{B}}$ structure $\bar{b}$ will be weighted as a type $\mathcal{{B}}$ structure $b$. Therefore, the ordinary generating function of $\mathcal{\bar{B}}$ is given by
\begin{align*}
\bar{B}(u)&=\sum_{n=1}^{\infty} \left(\sum_{\bar{b}\in \mathcal{\bar{B}}_n}w_{\mathcal{B}}(\bar{b})\right) u^{n}=\sum_{n=1}^{\infty} \left(\sum_{{b}\in \mathcal{{B}}_{n+1}}w_{\mathcal{B}}({b})\right) u^{n}=\frac{B(u)}{u}.
\end{align*}

Now we define another family of combinatorial structures, denoted by $\mathcal{\tilde{C}}$, in the following way: a structure $\tilde{c}$ built  built on an interval $\tilde{I}$ is obtained by first splitting $I$ into $k\in \mathbb{B}$ intervals  $(\tilde{I}_1,\tilde{I}_2,\ldots,\tilde{I}_k)$, build a type $\mathcal{\dot{A}}$ structure $\dot{a}$ on $\tilde{I}_1$,
for each $2\leq j\leq k$ build a type $\mathcal{A}$ structure $a_j$ on $\tilde{I}_j$, and then build a type $\mathcal{\bar{B}}$ structure $\bar{b}$ on ordered intervals $(\tilde{I}_2,\ldots,\tilde{I}_k)$. We define the weight function $w_{\mathcal{\tilde{C}}}$ of $\mathcal{\tilde{C}}$ by
$$w_{\mathcal{\tilde{C}}}(\tilde{c})=w_{\mathcal{A}}(\dot{a})
w_{\mathcal{A}}({a}_2)\cdots w_{\mathcal{A}}({a}_k)
w_{\mathcal{B}}(\bar{b}).$$
It is readily to see that
$$\mathcal{\tilde{C}} =\mathcal{\dot{A}}\times (\mathcal{\bar{B}}  \circ \mathcal{A}).$$
By Propositions \ref{prop-product-formula} and \ref{prop-composition-formula}, the generating function of $\mathcal{\tilde{C}}$ is given by
$$\tilde{C}(u)=\dot{A}(u)\times \bar{B}(A(u))=uA'(u)\frac{{B}(A(u))}{A(u)}.$$

With the above bijection $\phi$ between $\mathcal{C}^{\langle 1 \rangle}_n$ and $\mathcal{C}^{\langle 2 \rangle}_n$, we are able to establish a weight preserving bijection (also denoted by $\phi$) between $\mathcal{C}_n$ and $\mathcal{\tilde{C}}_n$.
Given a structure $c\in \mathcal{C}_n$, we define a structure
$\tilde{c}\in \mathcal{\tilde{C}}_n$ in the following way.
Suppose that $c$ corresponds to a cycle decomposition $C\in \mathcal{C}^{\langle 1 \rangle}_n$ containing $k\geq 2$ segments, say $\{I_1, I_2,\ldots,I_k\}$ arranged clockwise in the cycle with $1\in I_1$. If $\phi(C)=(\alpha_0,\alpha_1,\ldots,\alpha_{k-1},\alpha_{k})\in \mathcal{C}^{\langle 2 \rangle}_n$, then $|I_1|=\alpha_0+\alpha_1$ and $|I_j|=\alpha_j$ for each $2\leq j\leq k$. Recall that $c$ is obtained by first taking the cycle decomposition $C=\{I_1, I_2,\ldots,I_k\}$, then building a type $\mathcal{A}$ structure $a_j$ for each $I_j$, and then building a type $\mathcal{B}$ structure $b$ for the ordered blocks $(I_1, I_2,\ldots,I_k)$. To define $\tilde{c}$, we first decompose an interval $\tilde{I}$ of length $n$ into $k$ intervals $\tilde{I}_1,\tilde{I}_2,\ldots,\tilde{I}_k$ with $|\tilde{I}_1|=\alpha_0+\alpha_1$ and $|\tilde{I}_j|=\alpha_j$ for each $2\leq j\leq k$; assign the pointed   type $\mathcal{A}$  structure $\dot{a}_1$ to $\tilde{I}_1$ with its $(\alpha_0+1)$-th element attached to the pointer of size $0$;
for $2\leq j\leq k$ assign the type $\mathcal{A}$ structure $a_j$ to each $I_j$; and finally assign the type $\mathcal{\bar{B}}$ structure $b$ to the ordered set $(I_2,\ldots,I_k)$. Let $\phi(c)=\tilde{c}$. It is clear that $w_{\mathcal{\tilde{C}}}(\tilde{c})=w_{\mathcal{{C}}}({c})$, and hence $\phi$ is a weight preserving bijection between $\mathcal{{C}}_n$ and $\mathcal{\tilde{C}}_n$.
Thus, we have
$$C(u)=\tilde{C}(u).$$
This completes the proof.
\end{proof}

\section{Kazhdan-Lusztig polynomials}\label{sec-3}

The aim of this section is to prove Theorem \ref{klcoef}, namely, to determine the \KL  polynomials of fan graphs, squares of paths, wheel graphs and whirl matroids.
We will first determine the \KL  polynomials of fan graphs, which form the basis for computing the other three families of \KL  polynomials. It should be mentioned that generating function methodology is crucial for our computations.

\subsection{Fan graphs}\label{sec:non-ekl}

Let
\begin{align}\label{eqn-fankl-gf}
\Phi_F(t,u) := \sum_{n=0}^\infty P_{F_n}(t) u^{n},
\end{align}
where $F_0$ is the single-vertex graph.
The main result of this subsection is as follows.

\begin{thm}\label{thm-klpol-fan}
We have
\begin{align}\label{eq-gf-klpol-fan}
\Phi_F(t,u)=1+\frac{2 u}{1-u+\sqrt{(u-1)^2-4 tu^2}}.
\end{align}
\end{thm}

Our proof of Theorem \ref{thm-klpol-fan} is based on the recursive definition of the \KL polynomials. In fact, we could derive a functional equation satisfied by $\Phi_F(t,u)$ from \eqref{klpol-reform}. Due to the unique existence of the \KL polynomials, we will complete the proof by verifying that the right hand side of \eqref{eq-gf-klpol-fan}
satisfies this functional equation. We now  proceed to get such a equation. By \eqref{klpol-reform} we obtain
$$t^{\rk F_n} P_{F_n}(t^{-1}) =\sum_{C\in \mathcal{C}(F_n)}t^{-|C|}\chi_{F_n[C]}(t) P_{F_n/C}(t) .$$
Now multiply both sides by $u^n$ and then sum over all $n\geq 0$. For the left hand side, we have
\begin{align*}
\sum_{n=0}^{\infty}\left(t^{\rk F_n}P_{F_n}(t^{-1})\right)u^n
&=\sum_{n=0}^{\infty}\left(t^n P_{F_n}(t^{-1})\right)u^n
=\sum_{n=0}^{\infty}P_{F_n}(t^{-1})(tu)^n
=\Phi_F(t^{-1},tu)
\end{align*}
by the fact $\rk F_n=n$ and the equation \eqref{eqn-fankl-gf}.
Thus, we have
\begin{align}\label{eq-compostion}
\Phi_F(t^{-1},tu)=\sum_{n=0}^{\infty}
\left(\sum_{C\in \mathcal{C}(F_n)}t^{-|C|}\chi_{F_n[C]}(t) P_{F_n/C}(t)\right)u^n.
\end{align}
If the right hand side of \eqref{eq-compostion} can be expressed in terms of $\Phi_F(t,u)$, then we will obtain a functional equation satisfied by $\Phi_F(t,u)$.

For this purpose, we first give a characterization of those compositions appeared in \eqref{eq-compostion}.
Suppose that the fan graph $F_n$ has vertex set $V(F_n)=[0,n]$, where we use the convention that $[a,b]=\{a,a+1,\ldots,b\}$ for $a\leq b$. The size of the interval $[a,b]$ is defined to be $b-a+1$.
If vertex $0$ is adjacent to  each vertex in $[1,n]$, then $\mathcal{C}(F_n)$ can be characterized by using certain
set of weak integer compositions of $n$. Recall that a weak composition of $n$ is a sequence $(a_1,a_2,\ldots,a_k)$ of nonnegative integers such that
$a_1+a_2+\cdots+a_k=n$, denoted by $(a_1,a_2,\ldots,a_k)\models n$.
Let $\mathcal{S}_n$ denote the set of compositions with each part strictly greater than zero.
Let $\mathcal{E}_n$ denote the set of weak integer compositions of $n$ with even number of parts, say $(a_1,a_2,\ldots,a_{2k-1},a_{2k})$, satisfying
$a_i\geq 1$ for $1< i< 2k$. Note that $\mathcal{E}_0=\{(0,0)\}$.
For convenience we set
$\mathcal{S}_0=\{(\,)\}$.
For each $\sigma = (a_1,a_2,\ldots,a_{2k-1},a_{2k})\in \mathcal{E}_n$,
let
$$\theta(\sigma)=\{
(A_1,A_2,\ldots,A_{2k-1},A_{2k})\, | \,
A_{2i-1}=(a_{2i-1}) \mbox{ and } A_{2i}\in \mathcal{S}_{a_{2i}} \mbox{ for } 1\leq i\leq k.\}$$
Let
\begin{align}\label{eq-cnprime}
\mathcal{C}'_n=
\cup_{\sigma\in \mathcal{E}_n} \theta(\sigma).
\end{align}
Then we have the following result.

\begin{lem}\label{combij}
For any $n\geq 1$,  there exists a bijective map $\phi$ from $\mathcal{C}(F_n)$ to $\mathcal{C}'_n$.
\end{lem}

\begin{proof}
To construct a map $\phi$ from
$\mathcal{C}(F_n)$ and $\mathcal{C}'_n$,
one may draw $F_n$ in the plane such that $0$ is connected to
the $n$-vertex path with vertices labeled $1,2,\ldots,n$ from left to right. For each composition $C \in  \mathcal{C}(F_n)$, consider the unique component containing $0$ of its induced subgraph. Removing $0$ from this component leads to a sequence of subpaths, which lie on the $n$-vertex path from left to right, say $T_1,T_2,\ldots,T_k$. It is clear that these subpaths naturally divide the other components into $k+1$ segments  from left to right, say $S_1,S_2,\ldots,S_{k+1}$. Each $S_i$ can be considered as a forest (actually a sequence of paths), which naturally contribute a composition
$A_i$ of $|V(S_i)|$ with parts being the cardinalities of its ordered connected components.
Note that $S_1$ and $S_{k+1}$ might be the null graph, but each $S_i$ for $2\leq i\leq k$ is not the null graph since $C$ is a graph composition.
For example, for the composition  $C=\{\{0,1,2,7,9,10\},\{3,4\},\{5\},\{6\},\{8\} ,\{11\},\{12\}  \}$ of $F_{12}$, the corresponding $S_i$'s and $T_i$'s are illustrated in Figure \ref{f12}.

\begin{figure}[H]
\centering
 \begin{tikzpicture}[line cap=round,line join=round]
\draw (2,0)   node[circle,fill,inner sep=0pt,label=above:0] (0) {};
\foreach \s in {1,2,...,12}   \draw (\s-4,-3)   node[circle,fill,inner sep=0.5pt,label=below:\s] (\s) {};
\draw[color=blue,line width=1.5pt] (-3,-3)--(-2,-3)--(2,0)--cycle;
\draw[color=blue,line width=1.5pt] (5,-3)--(6,-3)--(2,0)--cycle;
\draw[color=blue,line width=1.5pt] (2,0)--(3,-3);
\draw[color=blue,line width=1.5pt]  (-1,-3)--(0,-3);
\foreach \s in {3,4,...,6}  \draw (0)--(\s);
\foreach \s in {8,11,12}  \draw (0)--(\s);
\draw (2)--(3);
\draw (4)--(9);
\draw (10)--(12);
\foreach \s in {1,2,...,12}   \draw (\s-4,-3.8)   node[circle,fill,inner sep=0.5pt,label=below:] (\s) {};
\draw[color=blue,line width=1.5pt] (1)--(2);
\node[draw=none,minimum size=3mm,inner sep=0pt,yshift=-8pt,label=below:$S_1$]   at  (-4,-3.7)   {};
\node[draw=none,minimum size=3mm,inner sep=0pt,yshift=-8pt,label=below:$T_1$]   at  (-2.5,-3.7)   {};
\draw [decorate,decoration={brace,amplitude=5pt},xshift=-0.0cm,yshift=-4pt]
(2,-3.8) -- (-1,-3.8) node [black,midway,xshift=-0pt,yshift=-15pt]  {$S_2$};
\draw[color=blue,line width=1.5pt] (3)--(4);
\node[draw=none,minimum size=3mm,inner sep=0pt,yshift=-8pt,label=below:$T_2$]   at  (3,-3.7)   {};
\node[draw=none,minimum size=3mm,inner sep=0pt,yshift=-8pt,label=below:$S_3$]   at  (4,-3.7)   {};
\node[draw=none,minimum size=3mm,inner sep=0pt,yshift=-8pt,label=below:$T_3$]   at  (5.5,-3.7)   {};
\draw[color=blue,line width=1.5pt] (9)--(10);
\draw [decorate,decoration={brace,amplitude=3pt},xshift=-0.0cm,yshift=-4pt]
(8,-3.8) -- (7,-3.8) node [black,midway,xshift=-0pt,yshift=-15pt]  {$S_4$};
\end{tikzpicture}
\captionsetup{font=footnotesize}
\captionof{figure}{Construction of $S_i$'s and $T_i$'s}
\label{f12}
\end{figure}
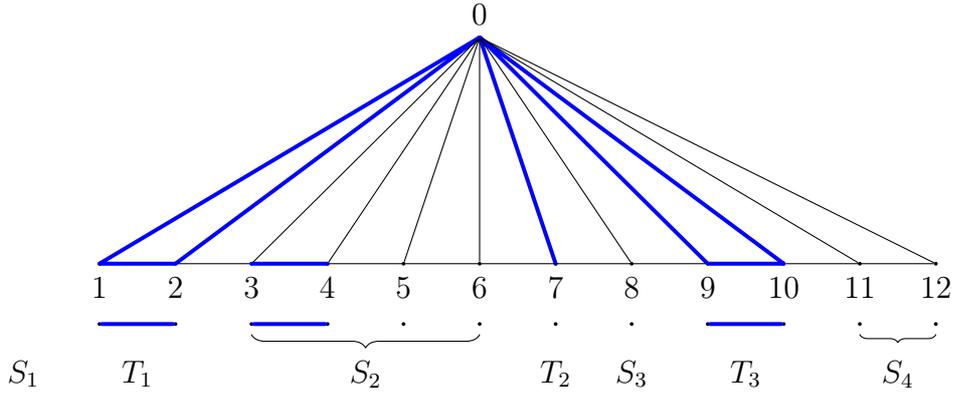
If $S_1$ is the null graph, let
$$\phi(C)=((|V(T_1)|),A_2,(|V(T_2)|),A_3,\cdots,A_k,(|V(T_k)|),A_{k+1}).$$
If $S_1$ is not empty, we add the null graph on the left of $S_1$ and then let
$$\phi(C)=((\,),A_1,(|V(T_1)|),A_2,\cdots,A_k,(|V(T_k)|),A_{k+1}).$$
In both cases, we have $\phi(C)\in \mathcal{C}'_n$.

The injectivity of $\phi$ is obvious by its construction.
It remains to show that $\phi$ is surjective.
Note that removing all parentheses from
a given element $(A_1,A_2,\ldots,A_{2k-1},A_{2k})\in\mathcal{C}'_n$
will lead to a composition $\sigma$ of $n$. If $A_1=()$
or (resp. and) $A_{2k}=()$
then we add a component $0$ at the beginning of $\sigma$
or (resp. and) add a component $0$ at the end of $\sigma$.
Suppose that the resulting weak composition is  $(c_1,c_2,\ldots,c_m)$.
Then decompose the $n$-vertex path into a series of subpaths $T_1,T_2,\ldots, T_m$ from left to right
with $|V(T_i)|=c_i$. Suppose that for $1\leq j\leq k$ the part $c_{i_j}$ originally comes from $A_{2j-1}$. Then we can get a set partition $C$ of $[0,n]$ with one block being $\{\{0\}\cup V(T_{i_1})\cup \cdots \cup V(T_{i_k})\}$ and other blocks being $V(T_i)$'s where $i\neq i_j$ for $1\leq j\leq k$. Clearly, this set partition is a graph composition of
$F_n$, and $\phi(C)=(A_1,A_2,\ldots,A_{2k-1},A_{2k})$. This completes the proof.
\end{proof}

With the above lemma, we proceed to show that for any
$C\in\mathcal{C}(F_n)$ the summand $t^{-|C|}\chi_{F_n[C]}(t) P_{F_n/C}(t) u^n$ can be evaluated by weighting the corresponding element $\phi(C)$ of $\mathcal{C}'_n$. Given $A=(A_1,A_2,\ldots,A_{2k-1},A_{2k}) \in \mathcal{C}'_n$,
suppose that $A_{2i-1}=(a_{2i-1})$ and $A_{2i}=(b_{i1},b_{i2},\ldots,b_{i\ell_i})\in\mathcal{S}_{a_{2i}}$ for $1\leq i\leq k$. Then define the weight  of $A$ to be
\begin{align}\label{eq-weight-function}
w(A)=\prod_{i=1}^{k}\frac{\chi_{F_{a_{2i-1}}}(t)}{t} \cdot P_{F_{\ell_i}}(t) \cdot\prod_{j=1}^{\ell_i} \frac{\chi_{H_{b_{ij}}}(t)}{t},
\end{align}
where we use $H_b$ to denote a path with $b$ vertices.
We have  the following result.

\begin{lem}\label{wterm}
For any $C\in \mathcal{C}(F_n)$, we have
\begin{align}\label{eq-transform}
t^{-|C|}\chi_{F_n[C]}(t) P_{F_n/C}(t)=w(\phi(C)),
\end{align}
where $\phi(C)$ is defined as in Lemma \ref{combij} and
the weight function $w$ is given by \eqref{eq-weight-function}.
\end{lem}

\begin{proof}
Suppose that
$$\phi(C)=(A_1,A_2,\ldots,A_{2k-1},A_{2k}),$$
where $A_{2i-1}=(a_{2i-1})$ and $A_{2i}=(b_{i1},b_{i2},\ldots,b_{i\ell_i})\in \mathcal{S}_{a_{2i}}$
for $1\leq i\leq k$.
By the construction of $\phi(C)$ in Lemma \ref{combij}, it is clear that
$$|C|=1+\ell_1+\cdots+\ell_k.$$
As illustrated in Figure \ref{fig:f12c}, the induced subgraph $F_n[C]$ is composed of subpaths $H_{b_{ij}}$ (where $i$ varies from $1$ to $k$ and $j$
varies from $1$ to $\ell_i$) and the unique connected component containing $0$.
The latter is obtained from fan graphs
$F_{a_1},F_{a_3},\ldots,F_{a_{2k-1}}$ by identifying their unique $0$ vertices.
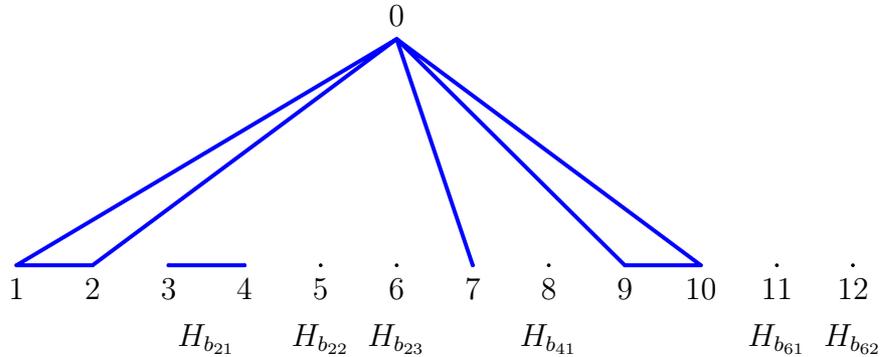
\begin{figure}[H]
\centering
\begin{tikzpicture}[line cap=round,line join=round]
\draw (2,0)   node[circle,fill,inner sep=0pt,label=above:0] (0) {};
\foreach \s in {1,2,...,12}   \draw (\s-4,-3)   node[circle,fill,inner sep=0.5pt,label=below:\s] (\s) {};
\draw[color=blue,line width=1.5pt] (-3,-3)--(-2,-3)--(2,0)--cycle;
\draw[color=blue,line width=1.5pt] (5,-3)--(6,-3)--(2,0)--cycle;
\draw[color=blue,line width=1.5pt] (2,0)--(3,-3);
\draw[color=blue,line width=1.5pt]  (-1,-3)--(0,-3);
\node[draw=none,minimum size=3mm,inner sep=0pt,yshift=-8pt,label=below:$H_{b_{21}}$]   at  (-0.5,-3.2)   {};
\node[draw=none,minimum size=3mm,inner sep=0pt,yshift=-8pt,label=below:$H_{b_{22}}$]   at  (1,-3.2)   {};
\node[draw=none,minimum size=3mm,inner sep=0pt,yshift=-8pt,label=below:$H_{b_{23}}$]   at  (2,-3.2)   {};
\node[draw=none,minimum size=3mm,inner sep=0pt,yshift=-8pt,label=below:$H_{b_{41}}$]   at  (4,-3.2)   {};
\node[draw=none,minimum size=3mm,inner sep=0pt,yshift=-8pt,label=below:$H_{b_{61}}$]   at  (7,-3.2)   {};
\node[draw=none,minimum size=3mm,inner sep=0pt,yshift=-8pt,label=below:$H_{b_{62}}$]   at  (8,-3.2)   {};
\end{tikzpicture}
\captionsetup{font=footnotesize}
\captionof{figure}{$F_{12}[\{\{0,1,2,7,9,10\},\{3,4\},\{5\},\{6\},\{8\} ,\{11\},\{12\}  \}]$}
\label{fig:f12c}
\end{figure}
By Lemma \ref{graphsum}, we have
$$\chi_{F_n[C]}(t)=t^{-k+1}\prod_{i=1}^{k} \chi_{F_{a_{2i-1}}}(t) \cdot \prod_{j=1}^{\ell_i} \chi_{H_{b_{ij}}}(t).$$
We proceed to compute $P_{F_n/C}(t)$. Recall that $F_n/C$ can be considered as the quotient graph of $F_n$ by identifying each block of $C$ as a single vertex, see Figure \ref{fig:f12/c}.
Thus $F_n/C$ is isomorphic to the graph obtained from fan graphs
$F_{\ell_1},F_{\ell_2},\ldots,F_{\ell_{k}}$ by identifying their unique $0$ vertices.
\begin{figure}[H]
\centering
\begin{tikzpicture}[line cap=round,line join=round]
\tikzstyle{every node}=[draw,shape=circle]
\draw (2,0.3)   node[draw=none,inner sep=0pt] (0) {\{0,1,2,7,9,10\}};
\draw (2,0)   node[circle,fill,inner sep=0pt] (0) {};
\foreach \s in {4,5,6,8,11,12}   \draw (\s-4,-3)   node[circle,fill,inner sep=0.5pt,label=below:] (\s) {};
\foreach \s in {4,5,6,8,11,12}  \draw (0)--(\s);
\draw (4)--(6) (11)--(12);
\node[draw=none,minimum size=3mm,inner sep=0pt]   at  (-0.2,-3.4)   {\{3,4\}};
\node[draw=none,minimum size=3mm,inner sep=0pt]   at  (1,-3.4)   {\{5\}};
\node[draw=none,minimum size=3mm,inner sep=0pt]   at  (2,-3.4)   {\{6\}};
\node[draw=none,minimum size=3mm,inner sep=0pt]   at  (4,-3.4)   {\{8\}};
\node[draw=none,minimum size=3mm,inner sep=0pt]   at  (7,-3.4)   {\{11\}};
\node[draw=none,minimum size=3mm,inner sep=0pt]   at  (8,-3.4)   {\{12\}};
 \end{tikzpicture}
\captionsetup{font=footnotesize}
\captionof{figure}{$F_{12}/\{\{0,1,2,7,9,10\},\{3,4\},\{5\},\{6\},\{8\} ,\{11\},\{12\}  \}$}
\label{fig:f12/c}
\end{figure}
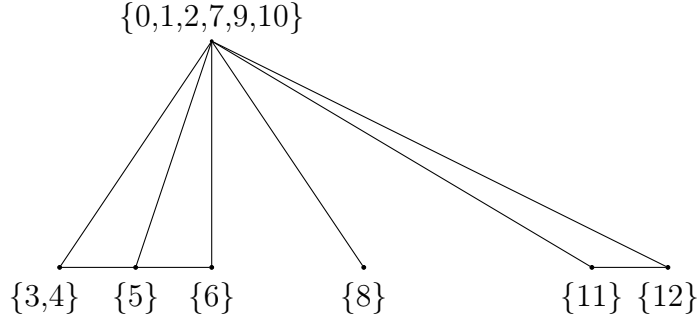

Again by Lemma \ref{graphsum},
we have
$$P_{F_n/C}(t)=\prod_{i=1}^{k}{P_{F_{\ell_i}}(t)}.$$
Combining the above identities, we obtained the desired result.
\end{proof}

In view of Lemma \ref{combij} and Lemma \ref{wterm}, it immediately follows from  \eqref{eq-compostion} that
\begin{align*}
\Phi_F(t^{-1},tu)=\sum_{n= 0}^{\infty} \left(\sum_{A\in \mathcal{C}'_n}w(A)\right)u^n,
\end{align*}
where $w(A)$ is given by \eqref{eq-weight-function}.

Let
\begin{align*}
\Psi(u)=\sum_{n= 0}^{\infty} \left(\sum_{A\in \mathcal{C}'_n}w(A)\right)u^n.
\end{align*}
Then
\begin{align}\label{eq-key-gf}
\Psi(u)=\Phi_F(t^{-1},tu).
\end{align}
In order to obtain a functional equation satisfied by $\Phi_F(t,u)$, we shall give another expression of $\Psi(u)$ in terms of $\Phi_F(t,u)$ by using the method of generating functions.

Note that each $A=(A_1,A_2,\ldots,A_{2k-1},A_{2k}) \in \mathcal{C}'_n$ may be considered as a combinatorial structure, say type $\mathcal{A}$ structure, on an interval of size $n$. In order to use the generating function methodology, let $\mathcal{A}^o$ and $\mathcal{A}^e$ denote the two types of structures respectively corresponding to the components of $A$. Precisely, type $\mathcal{A}^o$ structure will assign to an interval of size $n$ the weak composition $(n)$ with the weight function $w^o$ defined by $$w^o((n))=\frac{\chi_{F_{n}}(t)}{t},$$ and type $\mathcal{A}^e$ structure will assign to an interval of size $n$ a composition $(b_1,\ldots,b_k)\in\mathcal{S}_n$
with the weight function $w^e$ defined by
$$w^e((b_1,\ldots,b_k))=P_{F_{k}}(t) \cdot\prod_{j=1}^{k} \frac{\chi_{H_{b_{j}}}(t)}{t}.$$
Note that the unique $\mathcal{A}^o$ structure of size $0$ is $(0)$, weighted by $1$, and the unique $\mathcal{A}^e$ structure of size $0$ is $(\,)$, also weighted by $1$.
Let $\mathcal{A}^o_n$ (resp. $\mathcal{A}^e_n$) denote the set of
type $\mathcal{A}^o$ (resp. $\mathcal{A}^e$) structures which can be built on an interval of size $n$.
Let
\begin{align}
\Psi^o(u)&=\sum_{n= 1}^{\infty} \left(\sum_{A^o\in \mathcal{A}^o_n}w^o(A^o)\right)u^n,\label{gf-def-1}\\
\Psi^e(u)&=\sum_{n= 1}^{\infty} \left(\sum_{A^e\in \mathcal{A}^e_n}w^e(A^e)\right)u^n,\label{gf-def-2}
\end{align}
where the index $n$ in each summation starts with $1$ other than $0$.

We claim that the following results hold.
\begin{lem}\label{gen-o and e}
We have
\begin{align}
\Psi^o(u)&=\frac{(t-1) u}{1-(t-2) u},\label{gf-1}\\
\Psi^e(u)&=\Phi_F\left(t, \frac{u}{1-(t-1)u}\right)-1\label{gf-2}.
\end{align}
\end{lem}
\begin{proof}
Let us first prove \eqref{gf-1}.
By the definition of the chromatic polynomial of a graph, it is clear that
$$\chi_{F_n}(t)=t (t-1)(t-2)^{n-1}.$$
Thus, by \eqref{gf-def-1}, we get that
\begin{align*}
\Psi^o(u)&=\sum_{n= 1}^{\infty} \frac{\chi_{F_n}(t)}{t} u^n=\sum_{n= 1}^{\infty}(t-1)(t-2)^{n-1} u^n=\frac{(t-1) u}{1-(t-2) u}
\end{align*}
by a straightforward computation.

We proceed to prove \eqref{gf-2}. Note that a weighted  composition $(b_1,\ldots,b_k)\in\mathcal{S}_n$
can be considered as a composition of two weighted structures. Precisely, we first split the interval $[1,n]$ into a sequence of $k$ nonempty intervals, weight the interval of size $b_j$ by $\frac{\chi_{H_{b_{j}}}(t)}{t}$, and then weight the sequence by $P_{F_{k}}(t)$. Recall that
\begin{align}
\chi_{H_{b_j}}(t)&=t(t-1)^{b_j-1}\label{eq-path-chro}\\[5pt]
\sum_{b_j=1}^{\infty}{\frac{\chi_{H_{b_j}}(t)}{t} u^{b_j}}&=\frac{u}{1-(t-1)u}\label{eq-path-gf}\\[5pt]
\sum_{k=1}^{\infty}P_{F_{k}}(t) u^k&= \Phi_F(t,u)-1.
\end{align}
By the composition formula of generating functions, we obtain
\begin{align}
\Psi^e(u)&=\Phi_F\left(t, \frac{u}{1-(t-1)u}\right)-1,
\end{align}
as desired.
\end{proof}

Let $\mathcal{A}^{eo}$ be the set of pairs $(A^e,A^o)$, where $A^e$ is a structure of type $\mathcal{A}^e$, and $A^o$ is of type $\mathcal{A}^o$, and moreover neither  $A^e$ nor $A^o$ is empty. The weight function $w^{eo}$ of
$\mathcal{A}^{eo}$ is defined by
\begin{align}\label{eq-weight-eo}
w^{eo}((A^e,A^o))=w^e(A^e)w^o(A^o).
\end{align}
Let $\mathcal{A}^{eo}_n$ denote the set of  type $\mathcal{A}^{eo}$  structures which can be build on an interval of size $n$.
Consider the following generating function
\begin{align}
\Psi^{eo}(u)&=\sum_{n= 0}^{\infty} \left(\sum_{A^{eo}\in \mathcal{A}^{eo}_n}w^{eo}(A^{eo})\right)u^n.\label{gf-def-3-eo}
\end{align}
By the product formula of generating functions, we get that
\begin{align}\label{psi_eo_gf}
\Psi^{eo}(u)&=\Psi^{e}(u)\Psi^{o}(u).
\end{align}
Furthermore, let $\mathcal{A}^{m}$ be the set of combinatorial structures each of which is a sequence $(A^{eo}_1,\ldots,A^{eo}_k)$ of $\mathcal{A}^{eo}$ structures. Define the weight function $w^m$ of $\mathcal{A}^{m}$ as
$$w^{m}((A^{eo}_1,\ldots,A^{eo}_k))=\prod_{i=1}^k w^{eo}(A^{eo}_i).$$
Let $\mathcal{A}^m_n$ denote the set of  type $\mathcal{A}^m$  structures which can be build on an interval of size $n$.
Consider the following generating function
\begin{align}
\Psi^m(u)&=\sum_{n= 0}^{\infty} \left(\sum_{A^m\in \mathcal{A}^m_n}w^m(A^m)\right)u^n.\label{gf-def-3}
\end{align}
Applying the composition formula of generating functions once again, we get the following result, whose proof is omitted here.
\begin{lem}\label{gen-oe}
We have
\begin{align}
\Psi^m(u)&=\frac{1}{1-\Psi^e(u)\Psi^o(u) }.\label{gf-3}
\end{align}
\end{lem}

Now we are in the position to give another expression of $\Psi(u)$ in terms of $\Phi_F(t,u)$, which is stated as below.

\begin{lem}\label{fan-eqn}
Let $\Psi^o(u)$ and $\Psi^e(u)$ be as in Lemma \ref{gen-o and e}. Then
\begin{align}\label{eq-another}
\Psi(u)=\frac{(1+\Psi^o(u))(1+\Psi^e(u))}{1-\Psi^e(u)\Psi^o(u)}.
\end{align}
\end{lem}

\begin{proof}
Giving a structure $A=(A_1,A_2,\ldots,A_{2k-1},A_{2k})$ of $\mathcal{A}$,
the component $A_1$ is a structure $A^o$ of type $\mathcal{A}^o$, the component $A_{2k}$ is a structure $A^e$ of type $\mathcal{A}^e$,
and the subsequence $(A_2,A_3,\ldots,A_{2k-1})$ could be considered as a structure $A^m$ of type $\mathcal{A}^m$.
Note that $A_0$ is allowed to be $(0)$ and $A_{2k}$ is allowed to be $(\,)$. Moreover, it is straightforward to verify that
$$w(A)=w^o(A^o)w^m(A^m)w^e(A^e).$$
Therefore, the structure $\mathcal{A}$ can be considered as the product
$\mathcal{A}^o\times \mathcal{A}^{m} \times \mathcal{A}^e$ of three types of structures. By the product formula of generating functions, we have
\begin{align*}
\Psi(u)&=\sum_{n= 0}^{\infty} \left(\sum_{A^o\in \mathcal{A}^o_n}w^o(A^o)\right)u^n\times
\sum_{n= 0}^{\infty} \left(\sum_{A^m\in \mathcal{A}^m_n}w^m(A^m)\right)u^n \times \sum_{n= 0}^{\infty} \left(\sum_{A^e\in \mathcal{A}^e_n}w^e(A^e)\right)u^n\\
&=(1+\Psi^o(u))\times \frac{1}{1-\Psi^e(u)\Psi^o(u)}\times (1+\Psi^e(u)),
\end{align*}
as desired. This completes the proof.
\end{proof}

Finally, we come to the proof of Theorem \ref{thm-klpol-fan}.

\noindent \textit{Proof of Theorem \ref{thm-klpol-fan}.}
We proceed to prove \eqref{eq-gf-klpol-fan}.
Combining \eqref{eq-key-gf} and \eqref{eq-another}, we see that
$\Phi_F(t,u)$
satisfies the functional equation
\begin{align}\label{eq-gf-klpol-fan-funceqn}
 \Phi_F(t^{-1},tu)=\frac{(1+\Psi^o(u))(1+\Psi^e(u))}{1-\Psi^e(u)\Psi^o(u)}.
\end{align}

To complete the proof, we need to verify that
the above equation still holds if we substitute
$\Phi_F(t,u)$ by using the right hand side of \eqref{eq-gf-klpol-fan}. Though this could be verified by a tedious computation,
we prefer to give a computer aided proof as follows.


{\renewcommand\baselinestretch{2.5}
\begin{mma}
\In \Phi_{F}[u\_]:=\frac{2u}{1-u+\sqrt{(u-1)^{2}-4tu^2}}+1;\\
\In \Psi^{o}[u\_]:=\frac{(t-1)u}{1-(t-2)u};\\
\In \Psi^{F}[u\_]:=\frac{(1+\Psi^{o}[u])(1+\Psi^{e}[u])}{1-\Psi^{e}[u]\Psi^{o}[u]};\\
\In \Psi^{e}[u\_]:=\Phi_{F}\left[\frac{u}{1-(t-1)u}\right]-1;\\
\end{mma}
}
\begin{mma}
\In |Simplify|[(\Phi_F[u]/.\{t\to t^{-1},u\to t u\})==\Psi_F[u],|Assumptions|\to 1-(t-1)u>0]\\
\end{mma}
\begin{mma}
\Out  |True|\\
\end{mma}

The assumption $1-(t-1) u>0$ in the last step is reasonable since  $|u|$ is sufficiently small. This completes the proof. \qed

Now we can prove \eqref{eq-klpol-fan} of Theorem \ref{klcoef}.

\begin{proof}[Proof of \eqref{eq-klpol-fan}.]
It is sufficient to show the equivalence between \eqref{eq-klpol-fan} and \eqref{eq-gf-klpol-fan}. This equivalence might be known, see A055151 in \cite{oeis}. To be self-contained, we shall give a proof by utilizing  two Mathematica packages, one of which is  \textit{fastZeil} due to Paule and Schorn \cite{paule1995mathematica}  and  the other is \textit{GeneratingFunctions} due to Mallinger \cite{GeneratingFunctionsThesis}. To this end, let
$$a(n,k)={\frac{1}{k+1}\binom{n-1}{k,k,n-2k-1}}$$
and
$$a_n(t)=\sum_{k=0}^{n} a(n,k)t^k.$$
Note that
$P_{F_n}(t)=a_n(t)$, but here, to be compatible with \textit{fastZeil}, the upper bound of summation is set to be $n$ other than $\left\lfloor \frac{n-1}{2}\right\rfloor $.
To prove
\begin{align}\label{eq-temp-fan}
\sum_{n=0}^{\infty}a_n(t)u^n=1+\frac{2 u}{1-u+\sqrt{(u-1)^2-4 tu^2}},
\end{align}
we first import the packages and define one variable.
\begin{mma}
\In <<|RISC| ~\grave{} |fastZeil|~\grave{};\\
\In <<|RISC| ~\grave{} |GeneratingFunctions|~\grave{};\\
\end{mma}
\begin{mma}
\In  a[n\_, k\_] :=|FunctionExpand|\left[\frac{1}{k + 1}|Multinomial|[k,k,n-2 k- 1]\right];\\
\end{mma}
Then we use the command \textbf{Zb} to obtain a recursive relation of $a_n(t)$.
\begin{mma}
\In |ReleaseHold|[|First|[|Zb|[a[n, k] t^k, {k, 0, n}, n] /. |SUM| \to  a]];  \\
\end{mma}
\begin{mma}
\In rec = |Simplify|[|FunctionExpand|[\%], |Assumptions| \to n  \in \mathbb{Z}] \\
\Out n (-1+4 t) a[n]+(3+2 n) a[1+n]==(3+n) a[2+n] \\
\end{mma}
%

Together with the initial values $a_0(t)=a_1(t)=1$, the above recurrence can be transformed into an equivalent differential equation satisfied by
$$f(u)=\sum_{n=0}^{\infty}{a_n(t) u^n}.$$
This could be done by using the command \textbf{RE2DE} automatically.

\begin{mma}
\In de=|RE2DE|[{rec, a[0] == 1, a[1] == 1}, a[n], f[u]]\\
\Out  \{1+u+(-1+u)f[u]+(-u+2u^2-u^3+4tu^3)f'[u]==0,f[0]==1\}\\
\end{mma}

%

It remains to show that the right hand side of \eqref{eq-temp-fan} satisfies the differential equation \textbf{de}.
\begin{mma}
\In |Simplify|[|de|/.f\to \phi_F]\\
\Out \{|True|,|True|\}\\
\end{mma}
Thus we complete the proof of the equivalence between \eqref{eq-klpol-fan} and \eqref{eq-gf-klpol-fan}.
\end{proof}

\subsection{Wheel graphs}\label{subsect-wheel}

Based on the preceding results on fan graphs, we are able to determine the \KL polynomials of wheel graphs. As before, we consider the following generating function of $P_{W_n}(t)$ given by
\begin{align}\label{eqn-wheelkl-gf}
\Phi_W(t,u) := \sum_{n=2}^\infty P_{W_n}(t) u^{n},
\end{align}
where $W_2$ is a simple circle with three vertices.
Although the graph  $W_2$ is not a wheel graph,  we can consider it as a wheel graph in some sense.  Let $W_2^{'}$ be a multigraph with $V(W_2^{'})=\{0,1,2\}$ and $E(W_2^{'})=\{(0,1),(0,2),(1,2),(1,2)\}$. Then the graph $W_2^{'}$  can be considered as a wheel graph and  $W_2$ can be  obtained by deleting an edge $(1,2)$ from $W_2^{'}$.
Note that $P_{W_2}(t)$ and $P_{W_2^{'}}(t)$ are equal to each other.
 As will be shown later, it is convenient to include $P_{W_2}(t) u^{2}$ in the above summation.

The main result of this subsection is as follows.

\begin{thm}\label{thm-klpol-wheel}
We have
\begin{align}
\Phi_W(t,u)&=\frac{2 (u-1)}{\sqrt{(u-1)^2-4 t u^2}-u+1}-\frac{2 \left(u^2+u-1\right)}{(u+1) \left(\sqrt{(u-1)^2-4 t u^2}+u+1\right)}\nonumber\\
&\quad +\frac{2 u}{(u+1) \sqrt{(u-1)^2-4 t u^2}}.\label{eq-gf-klpol-wheel}
\end{align}
\end{thm}

We first give an outline of the proof of Theorem  \ref{thm-klpol-wheel}. By using the same arguments as for fan graphs, it is not difficult to show that
\begin{align}\label{eq-compostion-wheel}
\Phi_W(t^{-1},tu)=\sum_{n=2}^{\infty}
\left(\sum_{C\in \mathcal{C}(W_n)}t^{-|C|}\chi_{W_n[C]}(t) P_{W_n/C}(t)\right)u^n
\end{align}
in view of \eqref{klpol-reform} and the fact that $\rk W_n=n$ for any $n\geq 2$. Denote the right hand side of \eqref{eq-compostion-wheel} as $\Psi_W(u)$.
As in the case of fan graphs, we shall further give an alternative expression of $\Psi_W(u)$
in terms of $\Phi_W(t,u)$, which will lead to a functional equation satisfied by $\Phi_W(t,u)$. Then Theorem \ref{thm-klpol-wheel} could be proved with the help of computer packages.

To determine $\Psi_W(u)$, we need to analyze what a composition
of $W_n$ could be. Suppose that  the  wheel graph $W_n$ has vertex set $V(W_n)=[0,n]$, where the induced subgraph $W_n[[1,n]]$ is its outer cycle. Graphically, one may draw $W_n$ in the plane such that $0$ is connected to the $n$-vertex cycle  with vertices labeled $1,2,\ldots,n$ in clockwise when $n\geq 3$. It is not easy to show that $\mathcal{C}(W_n)$ is the disjoint union of the following three families:
\begin{align}
\mathcal{C}^{\langle 1\rangle}(W_n) &= \{\{[0,n]\},\{\{0\},[1,n]\} \}, \\
\mathcal{C}^{\langle 2\rangle}(W_n) &= \{C \in \mathcal{C}(W_n)|[0] \in C \text{ and } |C|\geq 3  \}, \label{eq-wheel-gd-2}\\
\mathcal{C}^{\langle 3\rangle}(W_n) &= \{C \in \mathcal{C}(W_n)|[0] \not \in C \text{ and } |C|\geq 2  \}.\label{eq-wheel-gd-3}
\end{align}
Correspondingly, for $i=1,2,3$ let
\begin{align}\label{eq-psi-123}
\Psi_i(u)=\sum_{n=2}^{\infty}\left(\sum_{C\in \mathcal{C}^{\langle i\rangle}(W_n)}t^{-|C|}\chi_{W_n[C]}(t) P_{W_n/C}(t)\right)u^n.
\end{align}
Then
$$\Psi_W(u)=\Psi_1(u)+\Psi_2(u)+\Psi_3(u).$$
In the following we shall determine $\Psi_1(u),\Psi_2(u)$ and $\Psi_3(u)$ successively.

Firstly, the series $\Psi_1(u)$ is given by the following result.

\begin{lem}\label{w1} We have
\begin{align}
 \Psi_1(u)=\frac{(t-2) (t-1) u^2}{(u+1) (1-(t-2) u)}+\frac{(t-1) u^2}{(u+1) (1-(t-1) u)}.
\end{align}
\end{lem}
\begin{proof}
Note that if $C=\{[0,n]\}$ then $W_n[C]$ is just the graph
$W_n$ and $W_n/C$ is just the single-vertex graph. Recall that
the \KL polynomial of the single-vertex graph is equal to $1$ and $$\chi_{W_n}(t)=t((t-2)^{n}-(-1)^{n-1}(t-2)),$$
see \cite[p. 68]{biggs1993algebraic}.
Next we consider the composition $C=\{\{0\},[1,n]\}$, in which case $W_n/C \cong K_2$ and hence
$P_{W_n/C}(t)=1$. Moreover, $W_2[C]$ is the union of the single-edge graph and the single-vertex graph, and for $n\geq 3$  the subgraph $W_n[C]$ is the union of a cycle of length $n$ and the single-vertex graph. In any case we have
$$\chi_{W_n[C]}(t)=t((t-1)^n+(-1)^n (t-1)),$$
for the chromatic polynomial of a circle see \cite[p. 65]{biggs1993algebraic}.
Therefore, we have
\begin{align*}
\Psi_1(u)&=\sum_{n=2 } u^n \left(\frac{\chi_{W_n}(t)}{t}+\frac{\chi_{Q_n}(t)}{t}\right)\\
&=\sum_{n=2 } u^n \left((t-2)^{n}-(-1)^{n-1}(t-2)+ \frac{ (t-1)^n+(-1)^n (t-1) }{t}\right)\\
&=\frac{(t-2) (t-1) u^2}{(u+1) (1-(t-2) u)}+\frac{(t-1) u^2}{(u+1) (1-(t-1) u)},
\end{align*}
as desired.
\end{proof}

We proceed to compute the series $\Psi_2(u)$ by using the method of generating functions.

\begin{lem}\label{w2}
We have
\begin{align*}
\Psi_2(u)&=\frac{1}{1-u(t-1)}\times \Phi_W\left(t,\frac{u}{1-u(t-1)}\right).
\end{align*}
\end{lem}
\begin{proof}
Note that a composition $C \in \mathcal{C}^{\langle 2\rangle}(W_n)$ naturally decomposes the outer cycle of $W_n$
into $|C|-1$ paths. Together with the single vertex $0$, these paths form the induced subgraph $W_n[C]$.
Suppose that these paths have  $i_1,i_2,\ldots,i_{|C|-1}$ vertices   respectively.
For example, for $C=\{\{0\},\{12,1,2,3\},\{4\},\{5,6,7\},\{8,9\},\{10\},\{11\}\}\in \mathcal{C}(W_{12})$, the outer cycle of $W_{12}$ is decomposed into 6 paths, as depicted in Figure \ref{fig:w12}.
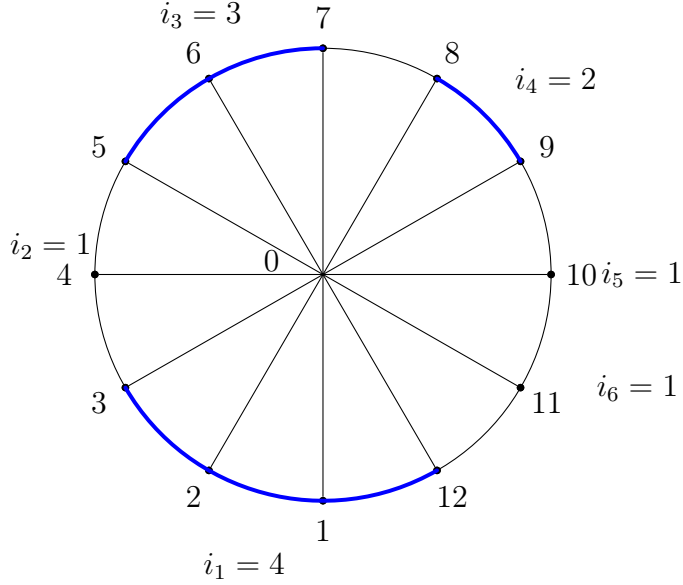
\begin{figure}[H]
\centering
\begin{tikzpicture}[line cap=round,line join=round]
\node[draw=none]  at  (165:0.7)   {0};
\draw (0,0) circle (3);
\foreach \s in {1,2,...,12} \draw node[circle,fill,inner sep=1pt]  (\s) at  ({-30*(\s-1)-90}:3)   {};
\foreach \s in {1,2,...,12} \draw (0,0)--(\s);
\foreach \s in {1,2,...,12}  \node[draw=none]  at  ({-30*(\s-1)-90}:3.4)   {\s};
\foreach \s in {1,2,...,12} \draw node[circle,fill,inner sep=1pt]  at  ({-30*(\s-1)-90}:3)   {};
\draw[color=blue,line width=1.5pt]  (12)   arc (-60:-150:3);
\draw[color=blue,line width=1.5pt]  (5)   arc (150:90:3);
\draw[color=blue,line width=1.5pt]  (8)   arc (60:30:3);
\node[draw=none]     at  ({-30*(1.5-1)-90}:4)     {$i_1=4$};
\node[draw=none]     at  ({-30*(4-1)-90-6}:3.6)     {$i_2=1$};
\node[draw=none]     at  ({-30*(6-1)-90-5}:4-0.2)     {$i_3=3$};
\node[draw=none]     at  ({-30*(8.5-1)-90-5}:4)     {$i_4=2$};
\node[draw=none]     at  ({-30*(10-1)-90}:4.2)     {$i_5=1$};
\node[draw=none]     at  ({-30*(11-1)-90+10}:4.4)     {$i_{6}=1$};
\end{tikzpicture}
\captionsetup{font=footnotesize}
\captionof{figure}{$W_{12}[\{\{0\},\{12,1,2,3\},\{4\},\{5,6,7\},\{8,9\},\{10\},\{11\}\}]$}
\label{fig:w12}
\end{figure}

It is clear that $W_n/C$ is isomorphic to $W_{|C|-1}$.
Thus, by Lemma \ref{graphsum} we have
\begin{align}\label{w2-weight}
t^{-|C|}\chi_{W_n[C]}(t) P_{W_n/C}(t)=P_{W_{|C|-1}}(t)\times \prod_{j=1}^{|C|-1}{\frac{\chi_{H_{i_j}}(t)}{t}}
\end{align}
in view of the obvious fact that the chromatic polynomial of the single-vertex graph is equal to $t$.
Therefore, the generating function $\Psi_2(u)$ could be considered as the ordinary generating function of type $\mathcal{B} \bullet  \mathcal{A}$ structures, as defined immediately before Proposition \ref{prop-composition-formula-variant}, with $\mathcal{A}_n=\{(n)\}$ for $n\geq 1$ weighted by $w_{\mathcal{A}}((n))=\frac{\chi_{H_{n}}(t)}{t}$ and $\mathcal{B}_n=\{(n)\}$ for $n\geq 2$ weighted by $w_{\mathcal{B}}((n))=P_{W_n}(t)$. (Here we assume that $\mathcal{A}_0, \mathcal{B}_0$ and $\mathcal{B}_1$ are empty.)

By \eqref{eq-path-gf} we know that the generating function of $\mathcal{A}$ is given by
\begin{align*}
A(u)&=\frac{u}{1-(t-1)u}.
\end{align*}
While, by \eqref{eqn-wheelkl-gf}, the generating function of $\mathcal{B}$ is given by $B(u)=\Phi_W(t,u)$.
Now by Proposition  \ref{prop-composition-formula-variant} we obtain that
\begin{align*}
 \Psi_2(u)&=uA'(u)\frac{B(A(u))}{A(u)}=\frac{1}{1-u(t-1)}\times \Phi_W\left(t,\frac{u}{1-u(t-1)}\right).
\end{align*}
This completes the proof.
\end{proof}

Analogously, we could determine the series $\Psi_3(u)$.

\begin{lem}\label{w3}
 We have
\begin{align*}
 \Psi_3(u)=u   \frac{\partial\Psi^{eo}(u)}{\partial u} \left(\frac{1}{1-\Psi^{eo}(u)} \right).
\end{align*}
\end{lem}
\begin{proof}
Given $C \in \mathcal{C}^{\langle 3\rangle}(W_n)$, there exists a unique and positive integer $k$, denoted by $\kappa_C$, satisfying the following condition:

($\diamond$)\textit{Removing all edges incident to $0$ from the induced subgraph $W_n[C]$ will result in a cyclically ordered sequence of non-empty forests $S_1,T_1,S_2,T_2,\ldots,S_k,T_k$ on the outer cycle where each $S_i$ is a sequence of paths with any of its vertices not adjacent to $0$, while each $T_i$ is a path with each of its vertices adjacent to $0$.}

Taking $C=\{\{0,4,9,10,11\},\{12,1,2\},\{3\},\{5\},\{6,7\},\{8\}\}\in \mathcal{C}(W_{12})$, we see that $\kappa_C=2$ as shown in
Figure \ref{fig:w12_2}.

\begin{figure}[H]
\centering
\begin{tikzpicture}[line cap=round,line join=round]
\node[draw=none]  at  (165:0.7)   {0};
\draw (0,0) circle (3);
\foreach \s in {1,2,...,12} \draw node[circle,fill,inner sep=1pt]  (\s) at  ({-30*(\s-1)-90}:3)   {};
\foreach \s in {1,2,...,12} \draw (0,0)--(\s);
\foreach \s in {1,2,...,12}  \node[draw=none]  at  ({-30*(\s-1)-90}:3.4)   {\s};
\foreach \s in {1,2,...,12} \draw node[circle,fill,inner sep=1pt]  at  ({-30*(\s-1)-90}:3)   {};
\foreach \s in {4,9,10,11} \draw[color=blue,line width=1.5pt]  (\s)--(0,0);;
\draw[color=blue,line width=1.5pt]  (9)   arc (-30*(9-1)-90:-30*(11-1)-90:3);
\draw[color=blue,line width=1.5pt]  (12)   arc (-60:-120:3);
\draw[color=blue,line width=1.5pt]  (6)   arc (120:90:3);

\node[draw=none]    at  (4.5,2)     {$S_1=$};
\foreach \s in {12} \draw node[circle,fill,inner sep=1pt,label=above:\s] (m\s) at  (\s-3.5-3,2) {};
\foreach \s in {1,2,3} \draw node[circle,fill,inner sep=1pt,label=above:\s] (m\s) at  (\s-3.5+9,2) {};

\draw[color=blue,line width=1.5pt]  (m12)--(m1)--(m2);
\node[draw=none]     at  (4.5,1)   {$T_1=$};
\draw node[circle,fill,inner sep=1pt,label=above:4]  at  (5.5,1) {};
\node[draw=none]     at  (4.5,0)   {$S_2=$};
\foreach \s in {5,6,7,8} \draw node[circle,fill,inner sep=1pt,label=above:\s] (t\s) at  (\s+0.5,0) {};
\draw[color=blue,line width=1.5pt]  (t6)--(t7);

\node[draw=none]     at  (4.5,-1)   {$T_2=$};
 
\foreach \s in {9,10,11} \draw node[circle,fill,inner sep=1pt,label=above:\s] (n\s) at  (\s-3.5,-1) {};
\draw[color=blue,line width=1.5pt]  (n9)--(n11);
\end{tikzpicture}
\captionsetup{font=footnotesize}
\captionof{figure}{Construction of $S_i$'s and $T_i$'s}
\label{fig:w12_2}
\end{figure}
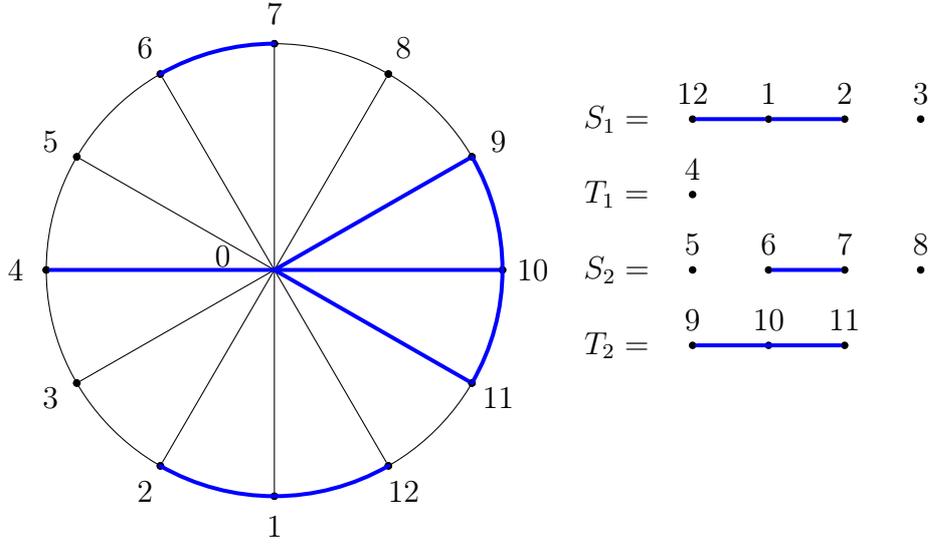

This is very similar to the case of fan graphs as discussed in Lemma \ref{combij}.
According to the value of $\kappa_C$, the set $\mathcal{C}^{\langle 3\rangle}(W_n)$ is divided into the following two subsets:
\begin{align*}
\mathcal{C}^{\langle 3,1\rangle}(W_n)&=\{C\in \mathcal{C}^{\langle 3\rangle}(W_n)\,|\,\kappa_C=1\},\\
\mathcal{C}^{\langle 3,2\rangle}(W_n)&=\{C\in \mathcal{C}^{\langle 3\rangle}(W_n)\,|\,\kappa_C>1\}.
\end{align*}
For $i=1,2$, let
\begin{align*}
\Psi_{3,i}(u)=\sum_{n=2}^{\infty}\left(\sum_{C\in \mathcal{C}^{\langle 3,i\rangle}(W_n)}t^{-|C|}\chi_{W_n[C]}(t) P_{W_n/C}(t)\right)u^n.
\end{align*}
Thus, we have $\Psi_{3}(u)=\Psi_{3,1}(u)+\Psi_{3,2}(u)$.

Next we shall first compute $\Psi_{3,1}(u)$. For each $C\in \mathcal{C}^{\langle 3,1\rangle}(W_n)$, let $S_1,T_1$
be those two forests appeared in condition ($\diamond$).
Clearly, the induced subgraph of $W_n[C]$ is composed of $S_1$ and the fan graph, denoted $F_{T_1}$, by connecting the vertex $0$ to each vertex of $T_1$. If the ordered paths in $S_1$ are $H_{b_{11}},\ldots,H_{b_{1,|C|-1}}$, then $W_n/C$ is isomorphic to $F_{|C|-1}$, a fan graph with $|C|$ vertices. Again by Lemma \ref{graphsum},
we get that
\begin{align*}
t^{-|C|}\chi_{W_n[C]}(t) P_{W_n/C}(t)=\left( P_{F_{|C|-1}}(t)\cdot \prod_{j=1}^{|C|-1}{\frac{\chi_{H_{b_{1j}}}(t)}{t}}\right)\times \frac{\chi_{F_{T_1}}(t)}{t}.
\end{align*}
If we weight $S_1$ by $P_{F_{|C|-1}}(t)\cdot \prod_{j=1}^{|C|-1}{(\chi_{H_{b_{1j}}}(t)/(t))}$ and $T_1$ by $\chi_{F_{T_1}}(t)/(t)$, then the ordered pair $(S_1,T_1)$
is essentially a weighted structure of type $\mathcal{A}^{eo}$ in view of \eqref{eq-weight-eo} as defined in Subsection \ref{sec:non-ekl}. Note that when we consider $(S_1,T_1)$ as a type $\mathcal{A}^{eo}$ structure, the first vertex of $S_1$ is always assumed to be $1$. But for each $C\in \mathcal{C}^{\langle 3,1\rangle}(W_n)$, the first vertex of the corresponding $S_1$ is fixed, which could be any value in the interval $[1,n]$. For this reason, each $C$ could be considered as a pointed type $\mathcal{A}^{eo}$ structure, where the first vertex of $S_1$ is attached to the pointer of size $0$. Therefore, by Proposition \ref{gf-pointed} we have
\begin{align}\label{eq-phi31}
\Psi_{3,1}(u)&=u\frac{\partial\Psi^{eo}(u)}{\partial u}
\end{align}
since both $\mathcal{A}^{eo}_0$ and $\mathcal{A}^{eo}_1$ are empty, where $\Psi^{eo}(u)$ is defined by \eqref{gf-def-3-eo}.

We proceed to determine $\Psi_{3,2}(u)$. Given $C\in \mathcal{C}^{\langle 3,2\rangle}(W_n)$,
let $S_1,T_1,S_2,T_2,\ldots,S_{\kappa_C},T_{\kappa_C}$ be those forests appeared in condition ($\diamond$).
By the same argument as above,
each ordered pair $(S_i,T_i)$ could be considered as  a structure $A_i^{eo}$ of type $\mathcal{A}^{eo}$ weighted by
$$
w^{eo}(A_i^{eo})=\left( P_{F_{\ell_i}}(t)\cdot \prod_{j=1}^{\ell_i}{\frac{\chi_{H_{b_{ij}}}(t)}{t}}\right)\times \frac{\chi_{F_{T_i}}(t)}{t},
$$
where $H_{b_{i1}},\ldots,H_{b_{i\ell_i}}$ are the ordered paths in $S_i$.
Moreover, it is routine to verify that
\begin{align*}
t^{-|C|}\chi_{W_n[C]}(t) P_{W_n/C}(t)=\prod_{i=1}^{\kappa_C}w^{eo}(A_i^{eo}),
\end{align*}
since $W_n/C$, as shown in Figure \ref{fig:w12_2/c}, is just the graph obtained from some fan graphs by identifying their $0$ vertices.

\begin{figure}[H]
\centering
\begin{tikzpicture}[line cap=round,line join=round]
\draw node[circle,fill,inner sep=1pt]  at  (0,0)  {};
\node[draw=none]  at  (180:1.6)   {$\{0,4,9,10,11\}$};
\foreach \s in {1,2,6,7,8} \draw node[circle,fill,inner sep=1pt]  (\s) at  ({-30*(\s-1)-90}:3)   {};
\foreach \s in {1,2,6,7,8} \draw (0,0)--(\s);
\foreach \s in {1,2,6,7,8}  \node[draw=none]  at  ({-30*(\s-1)-90}:3.4)  {};
\draw  (6)   arc (-30*(6-1)-90:-30*(8-1)-90:3);
\draw  (1)   arc (-30*(1-1)-90:-30*(2-1)-90:3);
\node[draw=none]  at  ({-30*(6-1)-90}:3.5)   {$\{5\}$};
\node[draw=none]  at  ({-30*(7-1)-90}:3.5)   {$\{6\}$};
\node[draw=none]  at  ({-30*(8-1)-90}:3.5)   {$\{7\}$};
\node[draw=none]  at  (-120:3.5)   {$\{3\}$};
\node[draw=none]  at  (-90:3.5)   {$\{12,1,2\}$};
\end{tikzpicture}
\captionsetup{font=footnotesize}
\captionof{figure}{$W_{12}/\{\{0,4,9,10,11\},\{12,1,2\},\{3\},\{5\},\{6,7\},\{8\}\}$}
\label{fig:w12_2/c}
\end{figure}
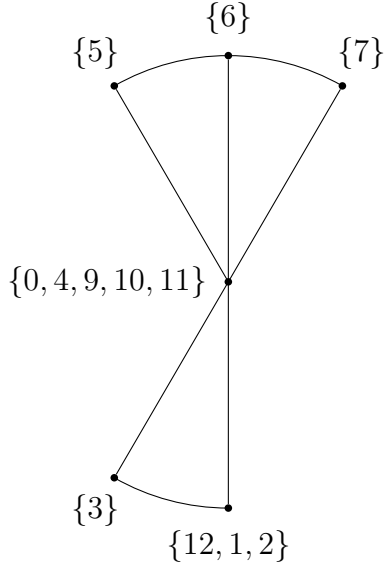

Thus $\Psi_{3,2}(u)$ could be considered as the ordinary generating function of type $\mathcal{B} \bullet  \mathcal{A}^{eo}$ structures, with $\mathcal{B}_n=\{(n)\}$ for $n\geq 2$ weighted by $w_{\mathcal{B}}((n))=1$.
The generating function of $\mathcal{B}$ is given by
$$B(u)=\sum_{n=2}^{\infty}w_{\mathcal{B}}((n))u^n=\frac{u^2}{1-u}.$$
Recall that the generating function of $\mathcal{A}^{eo}$ is
$\Psi^{eo}(u)$.
Therefore, by Proposition \ref{prop-composition-formula-variant}, we have
\begin{align}
\Psi_{3,2}(u)&=u\frac{\partial\Psi^{eo}(u)}{\partial u}\times \frac{B(\Psi^{eo}(u))}{\Psi^{eo}(u)}\nonumber\\
&=u\frac{\partial\Psi^{eo}(u)}{\partial u} \times \frac{1}{\Psi^{eo}(u)} \times \left(\frac{(\Psi^{eo}(u))^2}{1-\Psi^{eo}(u)} \right)\nonumber\\
&=u\frac{\partial\Psi^{eo}(u)}{\partial u} \left(\frac{\Psi^{eo}(u)}{1-\Psi^{eo}(u)} \right).\label{eq-phi32}
\end{align}
Combining \eqref{eq-phi31} and \eqref{eq-phi32}, we obtain
\begin{align*}
 \Psi_3(u)&=\Psi_{3,1}(u)+\Psi_{3,2}(u)\\
 &=u\frac{\partial\Psi^{eo}(u)}{\partial u}+u\frac{\partial\Psi^{eo}(u)}{\partial u} \left(\frac{\Psi^{eo}(u)}{1-\Psi^{eo}(u)} \right)\\
 &=u\frac{\partial\Psi^{eo}(u)}{\partial u} \left(\frac{1}{1-\Psi^{eo}(u)} \right).
\end{align*}
This completes the proof.
\end{proof}

Combining Lemma \ref{w1}, Lemma \ref{w2} and Lemma \ref{w3},  we obtain an explicit formula of $\Psi_W(u)$. Recall that $\Psi_W(u)$ denotes the right hand side of \eqref{eq-compostion-wheel}.

\begin{lem}\label{psi-w}
We have
\begin{align*}
 \Psi_W(u)&=\frac{(t-2) (t-1) u^2}{(u+1) (1-(t-2) u)}+\frac{(t-1) u^2}{(u+1) (1-(t-1) u)}
 \\
 &+\frac{1}{1-u(t-1)}\times \Phi_W\left(t,\frac{u}{1-u(t-1)}\right)+u \frac{\partial\Psi^{eo}(u)}{\partial u} \left(\frac{1}{1-\Psi^{eo}(u)} \right).
\end{align*}
\end{lem}

Now we come to the proof of the main result of this subsection.

\begin{proof}[{Proof of Theorem \ref{thm-klpol-wheel}.}]
By \eqref{eq-compostion-wheel} and Lemma \ref{psi-w}, we obtain  that
$\Phi_W(t,u)$  satisfies the functional equation
\begin{align*}
 \Phi_W(t^{-1},tu)=&\frac{(t-2) (t-1) u^2}{(u+1) (1-(t-2) u)}+\frac{(t-1) u^2}{(u+1) (1-(t-1) u)}
 \\
 &+\frac{1}{1-u(t-1)}\times \Phi_W\left(t,\frac{u}{1-u(t-1)}\right)+u \frac{\partial\Psi^{eo}(u)}{\partial u} \left(\frac{1}{1-\Psi^{eo}(u)} \right).
\end{align*}

To complete the proof of the theorem, we further verify that
the above equation still holds if we substitute
$\Phi_W(t,u)$ by using the right hand side of \eqref{eq-gf-klpol-wheel}. As in the case of fan graphs,
we prefer to give a computer aided proof as follows.
%
%
%
%
%

{\renewcommand\baselinestretch{2.5}
\begin{mma}
\In \Psi ^{eo}[u\_]:=\Psi^e[u] \Psi ^o[u];\\
\In \Psi _W[u\_]:=\frac{(t-2) (t-1) u^2}{(u+1) (1-(t-2) u)}+\frac{(t-1) u^2}{(u+1) (1-(t-1) u)}\linebreak+\frac{1}{1-(t-1) u}\Phi _W\left[\frac{u}{1-(t-1) u}\right] +u \frac{\partial \Psi ^{\text{eo}}(u)}{\partial u} \frac{1}{1-\Psi ^{\text{eo}}(u)};\\
\In \Phi _W[u\_] :=\frac{2 (u-1)}{\sqrt{(u-1)^2-4 t u^2}-u+1}-\frac{2 \left(u^2+u-1\right)}{(u+1) \left(\sqrt{(u-1)^2-4 t u^2}+u+1\right)}\linebreak+\frac{2 u}{(u+1) \sqrt{(u-1)^2-4 t u^2}} ;\\
\end{mma}
}
\begin{mma}
\In|Simplify|[(\Phi_W[u]/.\{t\to t^{-1},u\to tu\})==\Psi_W[u],|Assumptions|\to1-(t-1)u>0]\\
\Out |True| \\
\end{mma}

Note that here we assume that $1-(t-1) u>0$ because $u$ is sufficiently small. This completes the proof.
\end{proof}

Based on \eqref{eq-gf-klpol-wheel} we are able to prove
\eqref{eq-klpol-wheel} of Theorem \ref{klcoef}.

\begin{proof}[Proof of \eqref{eq-klpol-wheel}.]
It is sufficient to show that \eqref{eq-klpol-wheel} and \eqref{eq-gf-klpol-wheel} are equivalent to each other.
From \eqref{eq-klpol-wheel} it follows that $P_{W_{n+2}}(t)=a_n(t)$, where
$$a_n(t)=\sum_{k=0}^{n} a(n,k)t^k$$
and
\begin{align*}
a(n,k)=\left(\frac{k+1}{n-k+2} +\frac{k}{n-k+3}-\frac{k}{n-k+1} \right)\binom{n+2}{k,k+1,n-2k+1}.
\end{align*}
Note that for $n\geq k>\left\lfloor \frac{n+1}{2}\right\rfloor $ we have $a(n,k)=0$.
Since \eqref{eqn-wheelkl-gf} could be restated as
$$
\frac{\Phi_W(t,u)}{u^2}=\sum_{n=0}^{\infty}{P_{W_{n+2}}(t)u^n},
$$
to prove the equivalence between \eqref{eq-klpol-wheel} and \eqref{eq-gf-klpol-wheel} it suffices to show that
\begin{align}
\sum_{n=0}^{\infty}a_n(t)u^n&=\frac{2 (u-1)}{u^2 \left(\sqrt{(u-1)^2-4 t u^2}-u+1\right)}-\frac{2 \left(u^2+u-1\right)}{u^2 (u+1) \left(\sqrt{(u-1)^2-4 t u^2}+u+1\right)}\nonumber\\
&+\frac{2 u}{u^2 (u+1) \sqrt{(u-1)^2-4 t u^2}}.\label{eq-wheel-de}
\end{align}

To this end, we use the same method as in the proof of Theorem \ref{thm-klpol-fan}, but omit some details here. The following lines enable us to obtain a recurrence relation of $a_n(t)$.

\begin{mma}
\In  a[n\_, k\_] := \left(\frac{k+1}{n-k+2}+\frac{k}{n-k+3}-\frac{k}{n-k+1} \right)|Multinomial|[k, k + 1, n -2 k+1];\\
\end{mma}
\begin{mma}
\In |ReleaseHold|[|First|[|Zb|[|FunctionExpand|[a[n, k]] t^k, {k, 0, n}, n] /. |SUM| \to  a]];  \\
\end{mma}
\begin{mma}
\In|Simplify|[|FunctionExpand|[\%], |Assumptions| \to n  \in \mathbb{Z}]; \\
\begin{mma}
\end{mma}
\In  rec = |Collect|[\%, a[\_], |Factor|] \\
\Out  (-60-12 n+1758 t+1372 n t+446 n^2 t+68 n^3 t+4 n^4 t-738 t^2-881 n t^2-426 n^2 t^2-85 n^3 t^2-6 n^4 t^2+84 t^3+166 n t^3+106 n^2 t^3+26 n^3 t^3+2 n^4 t^3) a[2+n]+(7+n) (6-150 t-85 n t-21 n^2 t-2 n^3 t+12 t^2+22 n t^2+12 n^2 t^2+2 n^3 t^2) a[3+n]==(3+n) t (-1+4 t) (6-258 t-133 n t-27 n^2 t-2 n^3 t+48 t^2+52 n t^2+18 n^2 t^2+2 n^3 t^2) a[n]+(-18-6 n+906 t+693 n t+214 n^2 t+33 n^3 t+2 n^4 t-4956 t^2-4198 n t^2-1408 n^2 t^2-224 n^3 t^2-14 n^4 t^2+264 t^3+952 n t^3+618 n^2 t^3+146 n^3 t^3+12 n^4 t^3) a[1+n]\\
\end{mma}

To obtain an equivalent differential equation satisfied by $\sum_{n=0}^{\infty} a_n(t) u^n$, we also need some initial terms, which could be easily implemented by using the following command.
%

\begin{mma}
\In |Table|[a[n] == |Sum|[a[n, k] t^k, {k, 0, n}], {n, 0, 2}]\\
\Out \{a[0] == 1, a[1] == 1 + t, a[2] == 1 + 5 t\};\\
\end{mma}

Now we can obtain a differential equation satisfied by  $\sum_{n=0}^{\infty} a_n(t) u^n$ by using the function \textbf{RE2DE}.
%

\begin{mma}
\In de=|RE2DE|[\{rec,a[0]==1,a[1]==1+t,a[2]==1+5t\}, a[n], f[u]]\\
\Out \{-24+120 t+6 u-108 t u+336 t^2 u-6 (-4+20 t+6 u-53 t u+16 t^2 u-2 u^2+66 t u^2-326 t^2 u^2-34 t^3 u^2-3
 t u^3+141 t^2 u^3-540 t^3 u^3+96 t^4 u^3) f[u]-6 (-u+20 t u+2 u^2-75 t u^2+18 t^2 u^2-u^3+85 t u^3-500 t^2 u^3+78 t^3 u^3-t u^4+155 t^2 u^4-660 t^3 u^4+224 t^4 u^4) f^{'}[u]-3 (24 t u^2-90 t u^3+71 t^2 u^3-14 t^3 u^3+72 t u^4-474 t^2 u^4+210 t^3 u^4+109 t^2 u^5-500 t^3 u^5+256 t^4 u^5) f^{''}[u]-(23 t u^3-14 t^2 u^3-60 t u^4+73 t^2 u^4-22 t^3 u^4+37 t u^5-252 t^2 u^5+170 t^3 u^5+45 t^2 u^6-216 t^3 u^6+144 t^4 u^6) f^{(3)}[u] -2 (t u^4-t^2 u^4-2 t u^5+3 t^2 u^5-t^3 u^5+t u^6-7 t^2 u^6+6 t^3 u^6+t^2 u^7-5 t^3 u^7+4 t^4 u^7) f^{(4)}[u]==0,f[0]==1,f^{'}[0]==1+t, f^{''}[0]==2 (1+5 t),f^{(3)}[0]==6 (1+11 t+5 t^2) \}\\
\end{mma}

Next, we have to show that the right hand side of \eqref{eq-wheel-de} is indeed the solution of the above differential equation \textit{de}. To this end, we also
need to verify its value at $u=0$. Since each of three terms on the right hand side of \eqref{eq-wheel-de} is not defined for $u=0$, we should simplify the summation of these terms.
We denote the resulting function by $\phi1(u)$.
%

\begin{mma}
\In \Phi1[u\_]=\frac{\Phi_W[u]}{u^2};\\
\scalebox{0.92}[1]{\In \frac{2u\left(\sqrt{(u-1)^2-4 t u^2}-4 t u+2 t+u-3\right)+4}{\sqrt{(u-1)^2-4 t u^2} \left(\sqrt{(u-1)^2-4 t u^2}+u+1\right)\left(\sqrt{(u-1)^2-4tu^2}-u(2tu+1)+1\right)};\\}
\par
\In \phi1[u\_ ]:=\%;\\
\end{mma}
\begin{mma}
\In |Simplify|[\phi1[u]==\Phi1[u]]\\
\Out |True| \\
\end{mma}

Finally, we verify that $\phi1(u)$ satisfies the differential equation \textit{de} with the correct initial values.
%

\begin{mma}
\In |Simplify|\left[|de|\text{/.}\, f\to \phi1\right]\\

\Out \{|True|,|True|,|True|,|True|,|True|\}\\
\end{mma}

Thus we establish the equivalence between
\eqref{eq-klpol-wheel} and \eqref{eq-gf-klpol-wheel}.
\end{proof}

\subsection{Whirl matroids}\label{subsect-whirl}

The aim of this subsection is to prove \eqref{eq-klpol-whirl}.  As in the case of wheel graphs, the computation of \KL  polynomials for whirl matroids is also based on that for fan graphs. We consider the following generating function of $P_{W^n}(t)$ given by
\begin{align}\label{eqn-whirlkl-gf}
\Phi^W(t,u) := \sum_{n=1}^\infty P_{W^n}(t) u^n,
\end{align}
where $W^1$ is  the graphic  matroid of a path with two vertices  and  $W^2$ is a simple matroid obtained from the graphic matroid  of $W_2^{'}$ (defined at the beginning of Subsection \ref{subsect-wheel}) by declaring $\{(1,2),(1,2)\}$ is also an  independent set.

The main result of this subsection is as follows.

\begin{thm}\label{thm-klpol-whirl}
We have
\begin{align}
\Phi^W(t,u)&=\frac{u+1}{2 (t u+1) \sqrt{(u-1)^2-4 t u^2}}-\frac{1}{2 (t u+1)}.\label{eq-gf-klpol-whirl}
\end{align}
\end{thm}

We first give an outline of the proof of Theorem  \ref{thm-klpol-whirl}. By using the same arguments as for fan graphs and wheel graphs, it is not difficult to show that
\begin{align}\label{eq-whirl}
\Phi^W(t^{-1},tu)=\sum_{n=1}^{\infty}
\left(\sum_{F\in L(W^n) }{\chi_{W^n_F}(t) P_{W^n/F}(t)}  \right)u^n
\end{align}
by the  definition of the \KL  polynomial and  the fact that $\rk W^n=n$ for any $n\geq 1$. Denote the right hand side of \eqref{eq-whirl} as $\Psi^W(u)$.
As in the case of fan graphs and wheel graphs, we shall further give an alternative expression of $\Psi^W(u)$ in terms of $\Phi^W(t,u)$, which will lead to a functional equation satisfied by $\Phi^W(t,u)$. Then Theorem \ref{thm-klpol-whirl} could be proved with the help of a computer algebra system.

To determine $\Psi^W(u)$, we need to analyze what a flat of $W^n$ could be.
Note that  for $n\geq 2$ the matroid  $W^n$ is not a graphic matroid any more. Due to the close relationship between whirl matroids and wheel matroids, it is possible to describe the flats of $W^n$ in terms of the edges of the wheel graph $W_n$. For $n\geq 3$, let $O_n$ be the edge set of the outer cycle of $W_n$, and let
\begin{align*}\label{flat-whirl-l12}
L_1(W^n)&=\cup_{e \in O_n}\{O_n\setminus \{e\}\},\\[5pt]
L_2(W^n)&=\{E(W_n)\} \cup  \{F \in L(M(W_n))\,|\, O_n  \not \subset F\}.
\end{align*}
For $n=1$ we may assume that $W^1=M(G)$ with $V(G)=\{0,1\}$ and $E(G)=\{(0,1)\}$ and further let $L_1(W^1)=\{\emptyset\}$ and $L_2(W^1)=\{E(G)\}$.
For $n=2$ we may assume that the unique pair of multiple edges of $W_2^{'}$ is $\{e_1,e_2\}$ and further let  $L_1(W^2)=\{\{e_1\},\,\{e_2\}\}$ and $L_2(W^2)=\{\emptyset,\{(0,1)\},\{(0,2)\},E(W_2^{'})\}$.
We have the following result.

\begin{lem}\label{whirl-flats}
For any $n\geq 1$, the set $L(W^n)$ of flats is the disjoint union of $L_1(W^n)$ and $L_2(W^n)$.
\end{lem}

\begin{proof}
We only need to prove the lemma for $n\geq 3$.
Let  $\rk_W$ (resp. $\rk^W$) be the  rank function  of  $M(W_n)$ (resp. $W^n$). Note that  $\rk_W X= \rk^W X$ for any $X\neq O_n$,
which will be frequently used in our proof.

We first show that each element of
$L_1(W^n)$ or $L_2(W^n)$ is a flat of $W^n$. It is clear that each element of $L_1(W^n)$ is a flat of $W^n$ since the outer cycle is an independent set. It is also clear that $E(W_n)$ is
a flat of $W^n$, which is actually the maximal flat in $L(W^n)$.
It remains to show that each $F \in L(M(W_n))$ satisfying $O_n  \not \subset F$ is also a flat of $W^n$. In this case we must have $|O_n\cap F| \leq n-2 $, which means that $L_1(W^n)$ and $L_2(W^n)$ are disjoint.  Otherwise, $|O_n\cap F| = n-1$.
Since $F$ is a flat of $W_n$, we must have $O_n\subset F$, a contradiction.  Hence, neither $F$ nor $F \cup{e}$ is $\{O_n\}$ for any  $e\in E(W_n)\setminus F$. Recall that the ground set of $W^n$ is $E(W_n)$. Thus  $rk^W F \cup\{e\}=\rk_W F \cup\{e\}>\rk_W F=rk^W F$ for any  $e\in W^n\setminus F$. That is to say $F$ is a flat of $W^n$.

It remains to show that each flat $F \in L(W^n)$ belongs to either $L_1(W^n)$ or $L_2(W^n)$. From the previous characterization of compositions of $W_n$
it follows that $L(M(W_n))=L_2(W^n)\cup\{O_n\}$.
Suppose that $F\not\in L_2(W^n)$. We proceed to show that $F\in L_1(W^n)$. Since $O_n$ is not a flat of $W^n$, we have $F\neq O_n$. Thus we must have $F \not \in L(M(W_n))$, which implies the existence of some edge $e \not \in F$ satisfying $\rk_W F\neq  \rk_W F\cup \{e\}$. On the other hand, we have $\rk^W F= \rk^W F\cup \{e\}$ since $F$ is a flat of $W^n$.
Recalling that $\rk_W F=\rk^W F$ for $F\neq O_n$, we get that $\rk_W F\cup \{e\}\neq \rk^W F\cup \{e\}$, and hence $F\cup \{e\}=O_n$. Therefore $F \in L_1(W^n)$, as desired. This completes the proof.
\end{proof}

By the above lemma  $\Psi^W(u)$ admits the following decomposition
\begin{align}
\Psi^W(u)=\Psi^W_1(u)+\Psi^W_2(u),
\end{align}
where
 \begin{align}
\Psi^W_1(u)&=\sum_{n=1}^{\infty}\left( \sum_{F \in L_1(W^n) }{\chi_{W^n_F}(t) P_{(W^n)^F}(t)} \right) u^n,\label{eq-whirl-l1}\\[5pt]
\Psi^W_2(u)&=\sum_{n=1}^{\infty}\left(\sum_{F \in L_2(W^n) }{\chi_{W^n_F}(t) P_{(W^n)^F}(t)}\right)u^n.\label{eq-whirl-l2}
\end{align}

We first determine $\Psi^W_1(u)$.

\begin{lem}\label{whirl1}
We have
\begin{align*}
\Psi^W_1(u)=\frac{1}{1-(t-1) u} \times \frac{u}{1-(t-1) u}.
\end{align*}
\end{lem}

\begin{proof}
Given a flat $F\in L_1(W^n)$, we need to
analyze what kinds of matroids $W^n_F$ and
$(W^n)^F$ are.
For $n=1$ we see that $L_1(W^1)=\{\emptyset\}$,
the matroid $W^n_{\emptyset}$ is the empty matroid
and $(W^1)^{\emptyset}$ is $W^1$. Thus, we have
\begin{align*}
\chi_{W^1_{\emptyset}}(t)=1,\quad P_{(W^1)^{\emptyset}}(t)=P_{W^1}(t)=1.
\end{align*}
Now suppose that $F \in L_1(W^n)$ for $n\geq 2$.
By the definition of $L_1(W^n)$, the matroid
$W^n_F$ is isomorphic to the graphic matroid of a path with $n$ vertices, and
$(W^n)^F$ is a matroid with its  simplification isomorphic to $W^1$. Thus, for $n\geq 2$, we have
\begin{align*}
\chi_{W^n_{F}}(t)=(t-1)^{n-1},\quad P_{(W^n)^{F}}(t)=P_{W^1}(t)=1.
\end{align*}
In view of the fact that $L_1(W^n)$ has exactly $n$ distinct flats, we get that
\begin{align*}
\Psi^W_1(u)&=\sum_{n=1}^{\infty}{n (t-1)^{n-1} u^n}=\frac{1}{1-(t-1) u} \times \frac{u}{1-(t-1) u}.
\end{align*}
This completes the proof.
\end{proof}

We proceed to determine $\Psi^W_2(u)$.
By \eqref{eq-wheel-gd-2}, \eqref{eq-wheel-gd-3} and the definition of $L_2(W^n)$, we get the following decomposition:
$$L_2(W^n)=L_2^{\langle 1\rangle}(W^n)\uplus L_2^{\langle 2\rangle}(W^n)\uplus L_2^{\langle 3\rangle}(W^n),$$
where $L_2^{\langle 1\rangle}(W^n)$ contains the ground set of $W^n$ as its unique flat, both $L_2^{\langle 2\rangle}(W^1)$ and $L_2^{\langle 3\rangle}(W^1)$ are empty, and
\begin{align*}
L_2^{\langle 2\rangle}(W^n)=\{E(W_n[C])\,|\,C \in \mathcal{C}^{\langle 2\rangle}(W_n)\}\\[5pt]
L_2^{\langle 3\rangle}(W^n)=\{E(W_n[C])\,|\,C \in \mathcal{C}^{\langle 3\rangle}(W_n)\}
\end{align*}
for any $n\geq 2$. Accordingly, for $i=1,2,3$ let
\begin{align}\label{eq-whirl-gf-l2-123}
\Psi^W_{2,i}(u)=\sum_{n=1}^{\infty}\left(\sum_{F \in L_2^{\langle i\rangle}(W^n) }{\chi_{(W^n)_F}(t) P_{(W^n)^F}(t)}\right) u^n.
\end{align}
Then we have
$$\Psi^W_2(u)=\Psi^W_{2,1}(u)+\Psi^W_{2,2}(u)+\Psi^W_{2,3}(u).$$
Next we shall compute $\Psi^W_{2,1}(u), \Psi^W_{2,2}(u)$ and $\Psi^W_{2,3}(u)$ successively.

\begin{lem}\label{whirl21}
Let $\Psi^W_{2,1}(u)$ be defined as in \eqref{eq-whirl-gf-l2-123}.
Then
\begin{align*}
\Psi^W_{2,1}(u)&=\frac{1}{1-(t-2) u}-\frac{1}{u+1}.
\end{align*}
\end{lem}

\begin{proof}
Suppose that $F$ is the unique flat in $L_2^{\langle 1\rangle}(W^n)$. As $F$ is the ground set of $W^n$, the matroid  $(W^n)_F$ is indeed $W^n$ and the matroid $(W^n)^F$ is the empty matroid. It is known that $\chi_{W^n}(t)=(t-2)^{n}-(-1)^{n}$, which can be easily proved  by using the deletion-contraction rule. Note that the \KL polynomial of the empty matroid is equal to $1$. Thus we have
$$\Psi^W_{2,1}(u)=\sum_{n=1}^{\infty}{\chi_{W^n}(t)P_{(W^n)^F}(t) u^n}=\sum_{n=1}^{\infty}{((t-2)^{n}-(-1)^{n}) u^n}=\frac{1}{1-(t-2) u}-\frac{1}{1+u}.$$
\end{proof}

\begin{lem}\label{whirl22}
Let $\Phi^W(t,u)$ be defined as in \eqref{eqn-whirlkl-gf}. Then
we have
\begin{align*}
\Psi^W_{2,2}(u)&=\frac{1}{1-(t-1) u} \left(\Phi^W\left(t,\frac{u}{1-(t-1) u}\right)-\frac{u}{1-(t-1) u}\right).
\end{align*}
\end{lem}

\begin{proof}
First we notice that
\begin{align}\label{eq-rewritten-0}
\Psi^W_{2,2}(u)&=\sum_{n=2}^{\infty}\left(\sum_{F \in L_2^{\langle 2\rangle}(W^n) }{\chi_{(W^n)_F}(t) P_{(W^n)^F}(t)}\right) u^n
\end{align}
since $L_2^{\langle 2\rangle}(W^1)=\emptyset$.
Note that for $n\geq 2$ the set $L_2^{\langle 2\rangle}(W^n)$ is exactly composed of those flats of $M(W_n)$ corresponding to the compositions of $\mathcal{C}^{\langle 2\rangle}(W_n)$.
Given  $C \in \mathcal{C}^{\langle 2\rangle}(W_n)$, let $F=E(W_n[C])$. On one hand, considering $F$ as a flat of $M(W_n)$, we have
\begin{align*}
t^{-|C|}\chi_{W_n[C]}(t)=\chi_{(M(W_n))_F}(t),\quad
\end{align*}
which can be obtained by the previous discussion in Section \ref{sec-2}.
On the other hand, considering $F$ as a flat of $W^n$,
we have $O_n \not \subset F$ since $F \in L_2^{\langle 2\rangle}(W^n)$, which implies that
$M(W_n)_F$ and $(W^n)_F$ have the same independent sets.
Therefore,
\begin{align}\label{eq-l2}
M(W_n)_F\simeq (W^n)_F,\quad \forall F \in L_2^{\langle 2\rangle}(W^n).
\end{align}
Thus
\begin{align*}
t^{-|C|}\chi_{W_n[C]}(t)=\chi_{(W^n)_F}(t).
\end{align*}
Due to the bijection between between $L_2^{\langle 2\rangle}(W^n)$ and $\mathcal{C}^{\langle 2\rangle}(W_n)$,
we also use $(W^n)^C$  to represent the simplification of $(W^n)^F$ if $F=E(W_n[C])$.
With this notation \eqref{eq-rewritten-0} could be rewritten as
\begin{align}\label{eq-rewritten-00}
\Psi^W_{2,2}(u)&=\sum_{n=2}^{\infty}\left(\sum_{C\in \mathcal{C}^{\langle 2\rangle}(W_n)}t^{-|C|}\chi_{W_n[C]}(t) P_{(W^n)^C}(t)\right)u^n.
\end{align}

To prove this lemma, we will use the same arguments of Lemma \ref{w2}, which gives an expression of
\begin{align}\label{eq-psi-2}
\Psi_2(u)&=\sum_{n=2}^{\infty}\left(\sum_{C\in \mathcal{C}^{\langle 2\rangle}(W_n)}t^{-|C|}\chi_{W_n[C]}(t) P_{W_n/C}(t)\right)u^n.
\end{align}
Recall that for each $C\in \mathcal{C}^{\langle 2\rangle}(W_n)$
the graph $W_n/C$ is isomorphic to $W_{|C|-1}$. By the definition of contraction, it is also clear that $(W^n)^C$ is isomorphic to $W^{|C|-1}$ since $E(W_n[C])\subsetneq O_n$.
By comparing \eqref{eq-rewritten-00} and \eqref{eq-psi-2} and using the same reasoning as in Lemma \ref{w2}, the generating function $\Psi^W_{2,2}(u)$ could be considered as the ordinary generating function of type $\mathcal{B} \bullet  \mathcal{A}$ structures, where $\mathcal{A},\mathcal{B}$ are the same as
before except that we take $w_{\mathcal{B}}((n))=P_{W^n}(t)$.
Thus we get that
\begin{align*}
\Psi^W_{2,2}(u)&=\frac{1}{1-u(t-1)}\times\sum_{n=2}^{\infty}{P_{W^n} \left(\frac{u}{1-u(t-1)}\right)^n}\\
&=\frac{1}{1-(t-1) u} \left(\Phi^W\left(t,\frac{u}{1-(t-1) u}\right)-\frac{u}{1-(t-1) u}\right).
\end{align*}
This completes the  proof.
\end{proof}

\begin{lem}\label{whirl23}
We have
\begin{align*}
\Psi^W_{2,3}(u)&=u\frac{\partial\Psi^{eo}(u)}{\partial u} \left(\frac{1}{1-\Psi^{eo}(u)} \right).
\end{align*}
\end{lem}

\begin{proof}
Since $L_2^{\langle 3\rangle}(W^1)=\emptyset$, we have
\begin{align}\label{eq-23-1}
\Psi^W_{2,3}(u)&=\sum_{n=2}^{\infty}\left(\sum_{F \in L_2^{\langle 3\rangle}(W^n) }{\chi_{(W^n)_F}(t) P_{(W^n)^F}(t)}\right) u^n.
\end{align}
In the following we always assume that $n\geq 2$.
Note that each $F \in L_2^{\langle 3\rangle}(W^n)$
is also a flat of $M(W_n)$. We proceed to show that
$M(W_n)_F\simeq (W^n)_F$ and $(M(W_n))^F\simeq (W^n)^F$.
The former can be proved along the same lines as in
the proof of \eqref{eq-l2}. It remains to prove the latter.
Recalling the definition of $L_2^{\langle 3\rangle}(W^n)$,
there must exist a positive  integer $i$ such that $(0,i)\in F$.
It is clear that $M(W_n)/(0,i)\simeq M(F_{n-1})$. In fact, by definition of contraction we also have $W^n/(0,i)\simeq M(F_{n-1})$ since $(0,i)\not\in O_n$. Thus $(M(W_n))^F\simeq (W^n)^F$ turns out to be valid since
$$(M(W_n))^F\simeq (M(W_n)^{(0,i)})^{F\setminus\{(0,i)\} }
\simeq (M(F_{n-1}))^{F\setminus\{(0,i)\} } $$
and
$$(W^n)^F\simeq ((W^n)^{(0,i)})^{F\setminus\{(0,i)\} }
\simeq (M(F_{n-1}))^{F\setminus\{(0,i)\} }. $$
Therefore \eqref{eq-23-1} could be rewritten as
\begin{align}\label{eq-23-2}
\Psi^W_{2,3}(u)&=\sum_{n=2}^{\infty}\left(\sum_{F \in L_2^{\langle 3\rangle}(W^n) }{\chi_{(M(W_n))_F}(t) P_{(M(W_n))^F}(t)}\right) u^n.
\end{align}
Note that for each $F \in L_2^{\langle 3\rangle}(W^n)$ there exists a unique graph composition $C\in \mathcal{C}^{\langle 3\rangle}(W_n)$ such that $F=E(W_n[C])$. By the previous discussion in Section \ref{sec-2}, from \eqref{eq-psi-123}
it follows that $\Psi^W_{2,3}(u)=\Psi_3(u)$.
By Lemma \ref{w3} we obtain the desired result.
\end{proof}

Combining Lemmas \ref{whirl1}, \ref{whirl21}, \ref{whirl22} and \ref{whirl23}, we get the following result.

\begin{lem}\label{psi-whirl}
We have
$$
\Psi^W(u)=\frac{1}{1-(t-2) u}-\frac{1}{u+1}+\frac{1}{1-(t-1) u} \Phi^W\left(t,\frac{u}{1-(t-1) u}\right)+u \frac{\partial\Psi^{eo}(u)}{\partial u}  \frac{1}{1-\Psi^{eo}(u)}.
$$
\end{lem}

Now we come to the proof of the main result of this subsection.

\begin{proof}[{Proof of Theorem \ref{thm-klpol-whirl}.}]
The proof is similar to that of Theorem \ref{thm-klpol-wheel}.
By   Lemma \ref{psi-whirl}, we obtain  that
$\Phi^W(t,u)$  satisfies the functional equation
\begin{align*}
 \Phi^W(t^{-1},tu)=&\frac{1}{1-(t-2) u}-\frac{1}{u+1}+\frac{1}{1-(t-1) u} \Phi^W\left(t,\frac{u}{1-(t-1) u}\right)\\
 &+u \frac{\partial\Psi^{eo}(u)}{\partial u}  \frac{1}{1-\Psi^{eo}(u)}.
\end{align*}

To complete the proof of the theorem, we further verify that
the above equation still holds if we substitute
$\Phi^W(t,u)$ by using the right hand side of \eqref{eq-gf-klpol-whirl}. As in the case of fan graphs and wheel graphs, we prefer to give a computer aided proof as follows.

%
\begin{mma}
\In   \Phi^W [u\_ ]:=\frac{u+1}{2  (t u+1) \sqrt{(u-1)^2-4 t u^2}}-\frac{1}{2  (t u+1)};\\
\end{mma}
\begin{mma}
\In \Psi ^W[ u$\_$ ]:=\frac{1}{1-(t-2) u}-\frac{1}{u+1}+\frac{1}{1-(t-1) u}\Phi ^W\left[\frac{u}{1-(t-1) u}\right]+u\frac{\partial  \Psi^{eo}[u] }{\partial u}  \frac{1}{1-\Psi^{eo}[u]};\\
\end{mma}
\begin{mma}
\In |Simplify| [(\Phi^W[u]/.\{t\to t^{-1},u\to tu\})==\Psi^W[u],|Assumptions|\to 1-(t-1)u>0]\\
\Out |True| \\
\end{mma}
Note that here we assume that $1-(t-1) u>0$ because $u$ is sufficiently small. This completes the proof.
\end{proof}

Now \eqref{eq-klpol-whirl} of Theorem \ref{klcoef} can be proved
based on \eqref{eq-gf-klpol-whirl}.

\begin{proof}
As before, it is sufficient to show that \eqref{eq-klpol-whirl} and \eqref{eq-gf-klpol-whirl} are equivalent to each other.
From \eqref{eq-klpol-whirl} it follows that $P_{W^{n+1}}(t)=a_n(t)$, where
$$a_n(t)=\sum_{k=0}^{n} a(n,k)t^k$$ and
 \begin{align*}
a(n,k)=\frac{n+1}{n+1-k} \binom{n}{k, k, n-2k}.
\end{align*}
Note that for $n\geq k>\left\lfloor \frac{n}{2}\right\rfloor $ we have $a(n,k)=0$.
Since \eqref{eqn-wheelkl-gf} could be restated as
$$
\frac{\Phi^W(t,u)}{u}=\sum_{n=0}^{\infty}{P_{W^{n+1}}(t)u^n},
$$
to prove the equivalence between \eqref{eq-klpol-whirl} and \eqref{eq-gf-klpol-whirl} it suffices to show that
\begin{align}
\sum_{n=0}^{\infty}a_n(t)u^n&=\frac{u+1}{2u (t u+1) \sqrt{(u-1)^2-4 t u^2}}-\frac{1}{2u (t u+1)}.\label{eq-whirl-de}
\end{align}

To this end, we use the same method as in the proofs of Theorems \ref{thm-klpol-fan} and \ref{thm-klpol-wheel}. The following lines enable us to obtain a recurrence relation of $a_n(t)$.


\begin{mma}
\In  a[n\_, k\_] := \frac{n+1}{n+1-k}|Multinomial|[k, k, n -2 k];\\
\end{mma}

\begin{mma}
\In |ReleaseHold|[|First|[|Zb|[|FunctionExpand|[a[n, k]] t^k, {k, 0, n}, n] /. |SUM| \to a]];  \\
\end{mma}
\begin{mma}
\In|Simplify|[|FunctionExpand|[\%], |Assumptions| \to n  \in \mathbb{Z}]; \\
\end{mma}
\begin{mma}
\In  rec = |Collect|[\%, a[\_], |Factor|] \\
\Out  (34 + 24 n + 4 n^2 - 46 t - 35 n t - 6 n^2 t + 16 t^2 + 12 n t^2 +
     2 n^2 t^2) a[2 + n] + (4 + n) (-5 - 2 n + 4 t + 2 n t) a[
    3 + n] == (2 + n) t (-1 + 4 t) (-7 - 2 n + 6 t + 2 n t) a[
    n] + (14 + 11 n + 2 n^2 - 102 t - 78 n t - 14 n^2 t + 74 t^2 +
     62 n t^2 + 12 n^2 t^2) a[1 + n]\\
\end{mma}
To obtain an equivalent differential equation satisfied by $\sum_{n=0}^{\infty} a_n(t) u^n$, we also need some initial terms, which could be easily implemented by using the following command.
\begin{mma}
\In |Table|[a[n] == |Sum|[a[n, k] t^k, {k, 0, n}], {n, 0, 2}]\\
\Out \{a[0] == 1, a[1] == 1, a[2] == 1 + 3 t\};\\
\end{mma}

Now we can obtain a differential equation satisfied by  $\sum_{n=0}^{\infty} a_n(t) u^n$ by using the function \textbf{RE2DE}.
\begin{mma}
\In de=|RE2DE|[\{rec,a[0]==1,a[1]==1+t,a[2]==1+3t\}, a[n], f[u]]\\
\Out \{-1 + 2 t - (-1 + 2 t - 2 u + 5 u^2 - 38 t u^2 + 24 t^2 u^2 +
14 t u^3 - 68 t^2 u^3 + 48 t^3 u^3) f[u] - (3 u - 2 t u - 12 u^2 + 17 t u^2 - 6 t^2 u^2 + 9 u^3-64 t u^3 + 50 t^2 u^3 + 13 t u^4 - 64 t^2 u^4 +48 t^3 u^4) f^{'}[u]-2 (u^2 - t u^2 - 2 u^3 + 3 t u^3 - t^2 u^3 + u^4 - 7 t u^4 +6 t^2 u^4 + t u^5 - 5 t^2 u^5 + 4 t^3 u^5) f^{''}[u] == 0, f[0] == 1, f^{'}[0] == 1\} \\
\end{mma}

Next, we have to show that the right hand side of \eqref{eq-whirl-de} is indeed the solution of the above differential equation \textit{de}. To this end, we also
need to verify its value at $u=0$. Since each of two terms on the right hand side of \eqref{eq-whirl-de} is not defined for $u=0$, we should simplify the summation of these terms. We denote the resulting function by $\phi2(u)$.
%

\begin{mma}
\In \phi2[u\_]:=\frac{2}{\sqrt{(u-1)^2-4 t u^2} \left(\sqrt{(u-1)^2-4 t u^2}+u+1\right)};\\
\In |Simplify|[ \phi2 [u ]==\frac{\Phi^W[u]}{u}]\\
\Out |True| \\
\end{mma}

Finally, we verify that $\phi2(u)$ satisfies the differential equation \textit{de} with the correct initial values.
%

\begin{mma}
\In |Simplify|\left[de\text{/.}\, f\to \phi2\right]\\
\Out \{|True|,|True|,|True|\}\\
\end{mma}

Thus we establish the equivalence between
\eqref{eq-klpol-whirl} and \eqref{eq-gf-klpol-whirl}.
\end{proof}

\subsection{Squares of paths}

The main objective of this subsection is prove \eqref{eq-klpol-square} of Theorem \ref{klcoef}.
Our proof of \eqref{eq-klpol-square} is based on
the isomorphism between the graphic matroid $M(S_n)$ and the graphic matroid $M(F_n)$, which is implied by
Whitney's 2-isomorphism theorem.

Now let us recall some related concepts. Throughout this subsection all graphs are assumed to be free of isolated vertices. Given two graphs $G$ and $G'$, we say that $G$ is $2$-isomorphic to $G'$ if $G'$ can be obtained from $G$ by a sequence of operations of the following three types:

\begin{enumerate}[(a)]
 \item Vertex identification. This operation acts on a graph $G$ by identifying $v_1$ and $v_2$ as a new vertex ${v}$,
     where $v_1$ and $v_2$ are two vertices lying in distinct components of $G$.

 \item Vertex splitting.  This operation is inverse to vertex identification. However, a graph can be split only at a cut-vertex.

  \item Twisting.  This operation acts on a graph $G$ in the following way. Suppose that there are two vertices $u$ and $v$ in $G$ such that $G$ can be obtained from two disjoint graphs $G_1$ and $G_2$  by identifying $u_1\in V(G_1)$ and $u_2\in V(G_2)$ as $u$, and identifying $v_1\in V(G_1)$ and $v_2\in V(G_2)$ as $v$. Then we
      identify $u_1$ with $v_2$ and identify $v_1$ with $u_2$
      to obtain a twisting $\tilde{G}$ of $G$ about $\{u, v\}$.
\end{enumerate}

The celebrated Whitney's 2-isomorphism theorem is stated as follows.

\begin{thm}\cite{whitney19332}\label{thm-whitney}
Let $G$ and $G'$ be graphs having no isolated vertices. Then $M(G)$ and $M(G')$ are isomorphic if and only if $G$ and $G'$ are 2-isomorphic.
\end{thm}

We would like to point out that it is easy to prove one direction of the above theorem, namely, if two graphs  $G$ and $G'$ are 2-isomorphic then  $M(G)$ and $M(G')$ are isomorphic. While it is difficult to prove the other direction. Note that only the easy part of Whitney's 2-isomorphism theorem will be used for our purpose here. With this theorem, we are able to prove the following result.

\begin{thm}\label{thm-sq-fan}
For any $n\geq 1$ the graphic matroid $M(S_n)$ and the graphic matroid $M(F_n)$ are isomorphic.
\end{thm}

\begin{proof}
For $1\leq n\leq 4$, it is routine to verify that $S_n$ and $F_n$ are isomorphic graphs, and hence their graphic matroids are isomorphic. For $n\geq 5$, it suffices to prove that the graph $S_n$ and the graph $F_n$ are 2-isomorphic by Theorem \ref{thm-whitney}. We shall show that $S_n$ can be obtained from $F_n$ by a sequence of twisting operations. To illustrate this process, we first label the vertices of $F_n$ as in Figure \ref{fw1}.


\begin{figure}[H]
\centering
\begin{tikzpicture}[line cap=round,line join=round]
\draw (1,0)   node[circle,fill,inner sep=1pt,label=below:$0$] (0) {};
\draw (-1,0)   node[circle,fill,inner sep=1pt,label=below:$n$] (n) {};
\draw (0,1.7)   node[circle,fill,inner sep=1pt,label=above:$n-1$] (n-1) {};
\draw (2,1.7)   node[circle,fill,inner sep=1pt,label=above:$n-2$] (n-2) {};
\draw (3.2,1.7)   node[circle,fill,inner sep=1pt,label=above:$n-3$] (n-3) {};
\draw (4.3,1.7)   node[circle,fill,inner sep=1pt,label=above:$n-4$] (n-4) {};
\draw (4.8,1.7)   node[circle,fill,inner sep=1pt] (2) {};
\draw (5.3,1.7)   node[circle,fill,inner sep=1pt] (4) {};
\draw (6,1.7)   node[circle,fill,inner sep=1pt] (3) {};
\draw (7,1.7)   node[circle,fill,inner sep=1pt,label=above:$1$] (1) {};
\node[draw=none]  at  (5.4,2)   {$\cdots$};
\node[draw=none]  at  (6.4,2)   {$\cdots$};
\draw (1)--(n-1)--(n) ;
\foreach \x in {n,n-1,n-2,n-3,n-4,1,2,3,4} \draw (0)--(\x) ;
\end{tikzpicture}
\caption{$F_n$ }
\label{fw1}
\end{figure}

Then we construct a sequence of graphs $F_n^{\langle 0\rangle},F_n^{\langle 1\rangle},F_n^{\langle 2\rangle},\ldots,F_n^{\langle n-3\rangle}$ such that
$F_n^\langle 0\rangle$ is the original fan graph $F_n$,
and for each $1\leq i\leq n-3$ the graph $F_n^{\langle i\rangle}$
is the twisting of $F_n^{\langle i-1\rangle}$ about $\{0,n-1-i\}$.
The first two operations are illustrated as in  Figure \ref{fw2} and Figure \ref{fw3}.

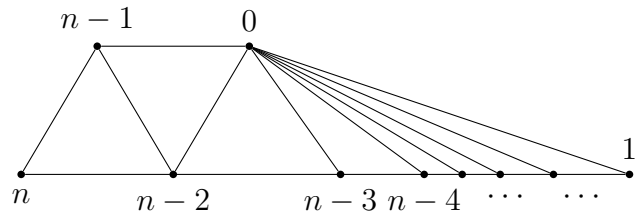
\begin{figure}[H]
\centering
\begin{tikzpicture}[line cap=round,line join=round]
\draw (2,1.7)    node[circle,fill,inner sep=1pt,label=above:$0$] (0) {};
\draw (-1,0)   node[circle,fill,inner sep=1pt,label=below:$n$] (n) {};
\draw (0,1.7)   node[circle,fill,inner sep=1pt,label=above:$n-1$] (n-1) {};
\draw  (1,0) node[circle,fill,inner sep=1pt,label=below:$n-2$] (n-2) {};
\draw (3.2,0)   node[circle,fill,inner sep=1pt,label=below:$n-3$] (n-3) {};
\draw (4.3,0)   node[circle,fill,inner sep=1pt,label=below:$n-4$] (n-4) {};
\draw (4.8,0)   node[circle,fill,inner sep=1pt] (2) {};
\draw (5.3,0)   node[circle,fill,inner sep=1pt] (4) {};
\draw (6,0)   node[circle,fill,inner sep=1pt] (3) {};
\draw (7,0)   node[circle,fill,inner sep=1pt,label=above:$1$] (1) {};
\node[draw=none]  at  (5.4,-0.3)   {$\cdots$};
\node[draw=none]  at  (6.4,-0.3)   {$\cdots$};
\draw (1)--(n)--(n-1)--(n-2) ;
\foreach \x in {n-1,n-2,n-3,n-4,1,2,3,4} \draw (0)--(\x) ;
\end{tikzpicture}
\caption{$F_n^{\langle 1\rangle}$: twisting of $F_n^{\langle 0\rangle}$ about $\{0,n-2\}$ }
\label{fw2}
\end{figure}

\begin{figure}[H]
\centering
\begin{tikzpicture}[line cap=round,line join=round]
\draw  (3,0)   node[circle,fill,inner sep=1pt,label=below:$0$] (0) {};
\draw (-1,0)   node[circle,fill,inner sep=1pt,label=below:$n$] (n) {};
\draw (0,1.7)   node[circle,fill,inner sep=1pt,label=above:$n-1$] (n-1) {};
\draw (2,1.7)   node[circle,fill,inner sep=1pt,label=above:$n-3$] (n-3) {};
\draw   (1,0) node[circle,fill,inner sep=1pt,label=below :$n-2$] (n-2) {};
\draw (4,1.7)   node[circle,fill,inner sep=1pt,label=above:$n-4$] (n-4) {};
\draw (5.2,1.7)   node[circle,fill,inner sep=1pt,label=above:$n-5$] (n-5) {};
\draw (6,1.7)   node[circle,fill,inner sep=1pt] (4) {};
\draw (6.6,1.7)   node[circle,fill,inner sep=1pt] (3) {};
\draw (7.2,1.7)   node[circle,fill,inner sep=1pt] (2) {};
\draw (8,1.7)   node[circle,fill,inner sep=1pt,label=above:$1$] (1) {};
\node[draw=none]  at  (6.5,2)   {$\cdots$};
\node[draw=none]  at  (7.3,2)   {$\cdots$};
\draw (1)--(n-1)--(n) --(n-2)--(n-3) (n-1)--(n-2);
\foreach \x in {n-2,n-3,n-4,n-5,1,2,3,4} \draw (0)--(\x) ;
\end{tikzpicture}
\caption{$F_n^{\langle 2\rangle}$: twisting of $F_n^{\langle 1\rangle}$ about $\{0,n-3\}$}\label{fw3}
\end{figure}
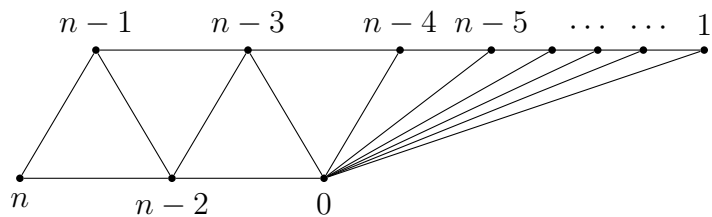

It is straightforward to show that
$F_n^{\langle n-3\rangle}$ is isomorphic to $S_n$.
Figure \ref{fw4} and Figure \ref{fw5} give the resulting labeled graphs according to the parity of $n$.

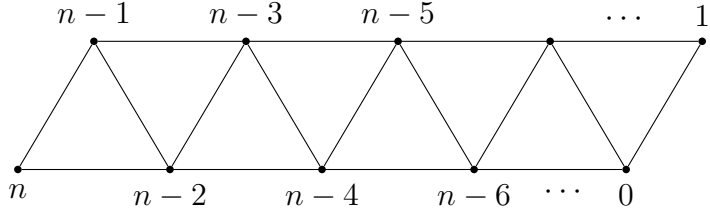
\begin{figure}[H]
\centering
\begin{tikzpicture}[line cap=round,line join=round]
\draw (-1,0)   node[circle,fill,inner sep=1pt,label=below:$n$] (n) {};
\draw   (1,0) node[circle,fill,inner sep=1pt,label=below :$n-2$] (n-2) {};
\draw  (3,0)   node[circle,fill,inner sep=1pt,label=below:$n-4$] (n-4) {};
\draw  (5,0)   node[circle,fill,inner sep=1pt,label=below:$n-6$] (n-6) {};
\draw  (7,0)   node[circle,fill,inner sep=1pt,label=below:$0$] (0) {};
\draw (0,1.7)   node[circle,fill,inner sep=1pt,label=above:$n-1$] (n-1) {};
\draw (2,1.7)   node[circle,fill,inner sep=1pt,label=above:$n-3$] (n-3) {};
\draw (4,1.7)   node[circle,fill,inner sep=1pt,label=above:$n-5$] (n-5) {};
\draw (6,1.7)   node[circle,fill,inner sep=1pt] (n-7) {};
\draw (8,1.7)   node[circle,fill,inner sep=1pt,label=above:$1$] (1) {};
\draw (n)--(n-1)--(n-2)--(n-3)--(n-4)--(n-5)--(n-6)--(n-7)--(0)--(1);
\draw (n-1)--(1);
\draw (n)--(0);
\node[draw=none]  at  (7,2)   {$\cdots$};
\node[draw=none]  at  (6.2,-0.3)   {$\cdots$};
\end{tikzpicture}
\caption{$F_n^{\langle n-3\rangle}$ for odd $n$: twisting  of $F_n^{\langle n-4\rangle}$ about $\{0,2\}$}\label{fw4}
\end{figure}

\begin{figure}[H]
\centering
\begin{tikzpicture}[line cap=round,line join=round]
\draw (-1,0)   node[circle,fill,inner sep=1pt,label=below:$n$] (n) {};
\draw   (1,0) node[circle,fill,inner sep=1pt,label=below :$n-2$] (n-2) {};
\draw  (3,0)   node[circle,fill,inner sep=1pt,label=below:$n-4$] (n-4) {};
\draw  (5,0)   node[circle,fill,inner sep=1pt,label=below:$n-6$] (n-6) {};
\draw  (7,0)   node[circle,fill,inner sep=1pt,label=below:$2$] (2) {};
\draw  (9,0)   node[circle,fill,inner sep=1pt,label=below:$1$] (1) {};

\draw (0,1.7)   node[circle,fill,inner sep=1pt,label=above:$n-1$] (n-1) {};
\draw (2,1.7)   node[circle,fill,inner sep=1pt,label=above:$n-3$] (n-3) {};
\draw (4,1.7)   node[circle,fill,inner sep=1pt,label=above:$n-5$] (n-5) {};
\draw (6,1.7)   node[circle,fill,inner sep=1pt] (n-7) {};
\draw (8,1.7)   node[circle,fill,inner sep=1pt,label=above:$0$] (0) {};
\draw (n-1)--(n-2)--(n-3)--(n-4)--(n-5)--(n-6)--(n-7)--(2)--(0)--(1)--(n) --(n-1)--(0);
\node[draw=none]  at  (7,2)   {$\cdots$};
\node[draw=none]  at  (6.2,-0.3)   {$\cdots$};
\end{tikzpicture}
\caption{$F_n^{\langle n-3\rangle}$ for even $n$: twisting  of $F_n^{\langle n-4\rangle}$ about $\{0,2\}$}\label{fw5}
\end{figure}
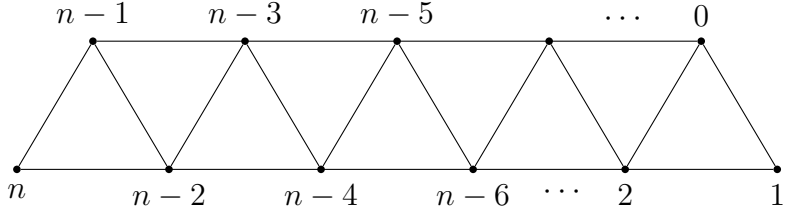
This completes the proof.
\end{proof}

Now we come to the proof of \eqref{eq-klpol-square} of Theorem \ref{klcoef}.

\noindent \textit{Proof of  \eqref{eq-klpol-square}.}
Combining Theorem \ref{thm-sq-fan} and Theorem \ref{thm-klpol-fan}, we immediately obtain the desired result. \qed

\section{Real zeros of \KL polynomials}\label{sec-4}

In Section \ref{sec-3} we obtained explicit expressions of
the \KL polynomials of square of paths, fan graphs, wheel graphs and whirl matroids. The main objective of this section is to
prove that these polynomials are real-rooted.

Let us first consider the \KL polynomials of square of paths and fan graphs. We have the following result.

\begin{thm}\label{kl-fan-roots}
For any $n\geq 3$ the  polynomial $P_{F_n}(t)$ has only negative zeros, so does $P_{S_n}(t)$. Moreover, we have
$P_{F_n}(t) \preceq P_{F_{n+1}}(t)$.
\end{thm}

\begin{proof}
By Theorem \ref{klcoef}, we know that
 \begin{align*}
P_{F_n}(t)=P_{S_n}(t)&=\sum_{k=0}^{\lfloor \frac{n-1}{2}\rfloor} {\frac{1}{k+1}\binom{n-1}{k,k,n-2k-1} t^k}.
\end{align*}
There is a close relationship between $P_{F_n}(t)$ and
the classical Narayana polynomial
 \begin{align*}
N_{n}(t)=\sum_{k=0}^{n-1} \, \frac{1 }{ n} \, \binom{n}{k}\binom{n}{k+1}  t^{k}.
\end{align*}
Precisely, we have
\begin{align}\label{eq-nara-kl}
N_n(t)=(1+t)^{n-1} P_{F_n}\left(\frac{t}{(1+t)^2}\right)
\end{align}
see \cite{coker2003enumerating}.
It is  well known that $N_n(t)$ has only simple negative zeros and moreover $N_n(t) \preceq N_{n+1}(t)$.
Suppose that $t_1, \ldots, t_{\lfloor \frac{n-1}{2}\rfloor}$ are those distinct zeros of $N_n(t)$ in the interval $(-1,0)$. From the symmetry of $N_n(t)$ it follows that $t^{-1}_1,\ldots, t^{-1}_{\lfloor \frac{n-1}{2}\rfloor}$ are those distinct zeros
in the interval $(-\infty,-1)$. In addition, for even $n$ we have $N_n(-1)=0$, while for odd $n$ we have $N_n(-1)\neq 0$.
Note that $\deg(P_{F_n}(t))=\lfloor \frac{n-1}{2}\rfloor$ and for $t\neq -1$ there holds
$$\frac{t^{-1}}{(1+t^{-1})^2}=\frac{t}{(1+t)^2}.$$
Thus, by \eqref{eq-nara-kl}, we see that
$$\frac{t_1}{(1+t_1)^2},\ldots,\frac{t_{\lfloor \frac{n-1}{2}\rfloor}}{\left(1+t_{\lfloor \frac{n-1}{2}\rfloor}\right)^2}$$
are exactly all zeros of $P_{F_n}(t)$.
It is easy to see that $\frac{t}{(1+t)^2}$ is a strictly increasing function on the interval $(-1,0)$. Thus we have the desired result.
\end{proof}

Next we consider the \KL polynomials of wheel  graphs. We have the following result.

\begin{thm}\label{thm-wheel-zeros}
For any $n\geq 3$ the  polynomial $P_{W_n}(t)$ has only negative zeros.
\end{thm}

Our proof of Theorem \ref{thm-wheel-zeros}  is based on the  theory of multiplier sequences and the theory of $n$-sequences,
for which we refer the reader to \cite{craven1983location1,craven1980intersections,craven2004composition}.
Let us recall some related concepts and results.
A sequence $\Gamma=\{\gamma_k\}_{k=0}^{\infty}$ of real numbers  is called a \emph{multiplier sequence} if, whenever any  real polynomial
$$f(t)=\sum_{k=0}^{n}a_kt^k$$ has only real zeros, so does the polynomial $$\Gamma[f(t)]=\sum_{k=0}^{n}\gamma_ka_kt^k.$$
A sequence $\Gamma=\{\gamma_i\}_{k=0}^{n}$  is called an \emph{$n$-sequence} if for every polynomial $f(t)$ of degree less than or equal to $n$ and with only real zeros, the polynomial $\Gamma[f(t)]$ also has only real zeros.

We will need  the following two lemmas about multiplier sequences.

\begin{lem}\label{ms-shift}
For any non-negative integer $c$, the sequence  $\{\frac{1}{(k+c)!}\}_{k=0}^{\infty}$
is a multiplier sequence.
\end{lem}

This lemma can be obtained  by applying   a classical theorem due to Laguerre to the  gamma function $\Gamma(x)$, see \cite[p .270]{titchmarsh1939theory} or a  more recent literature   \cite[Theorem 4.1]{craven2004composition}.

\begin{lem}\label{ms-rev}
For any positive integer $n$, the sequence  $\{\frac{1}{(n-k)!}\}_{k=0}^{\infty}$
is a multiplier sequence.
\end{lem}

This lemma can be obtained  by applying  the   characterization theorem of  multiplier sequences due to P{\'o}lya and Schur, see \cite{polya1914zwei} or
\cite[Theorem 3.3]{craven2004composition}.

We also need the following algebraic characterization of $n$-sequences.

\begin{thm}\cite{craven1983location1}\label{n-char}
Let $\Gamma=\{\gamma_k\}_{k=0}^{n}$ be a sequence of  real numbers. Then $\Gamma$ is an $n$-sequence if and only if the zeros of the polynomial
$\Gamma[(1+t)^n]$ are all real and of the same sign.
\end{thm}

To use the theory of multiplier sequences and the theory of $n$-sequences to prove the real-rootedness of $P_{W_n}(t)$, we shall rewrite its coefficient sequence $\{[t^k]P_{W_n}(t)\}_{k=0}^{\lfloor \frac{n-1}{2}\rfloor}$ as the Hadmard product of
three sequences $\{a_k\}_{k=0}^{\lfloor \frac{n-1}{2}\rfloor}$, $\{b_k\}_{k=0}^{\lfloor \frac{n-1}{2}\rfloor}$ and $\{c_k\}_{k=0}^{\lfloor \frac{n-1}{2} \rfloor}$, namely
\begin{align}\label{eq-hadmard}
[t^k]P_{W_n}(t)=a_kb_kc_k,
\end{align}
where
\begin{align}
a_k&=(k+1) n^2-(2k^2+4k) n+k^3+3 k^2-k-1,\label{seqa}\\[5pt]
b_k&=\frac{n!}{(n-1) (k+1)! (n+1-k)!},\label{seqb}\\[5pt]
c_k&=\frac{(n-1) (n-2-k)!}{k! (n-1-2 k)!}.\label{seqc}
\end{align}
It is straightforward to verify \eqref{eq-hadmard}.
We shall subsequently prove that the sequence $\{a_k\}_{k=0}^{\lfloor \frac{n-1}{2}\rfloor}$ is a $\lfloor \frac{n-1}{2} \rfloor$-sequence, the sequence $\{b_k\}_{k=0}^{\infty}$
is a multiplier sequence and the polynomial $\sum_{k=0}^{\lfloor \frac{n-1}{2}\rfloor}c_kt^k$ has only real zeros. Firstly, we prove the following result.

 \begin{lem}\label{wheel-n-seq}
  For any $n\geq 3$ the sequence
  $\{a_k\}_{k=0}^{\lfloor \frac{n-1}{2} \rfloor}$ given by \eqref{seqa} is a $\lfloor \frac{n-1}{2} \rfloor$-sequence.
 \end{lem}

 \begin{proof}
 To simplify notation, let $m=\lfloor \frac{n-1}{2} \rfloor$.
  Note that for any $0 \leq k \leq m$
 there exists $x\geq 0$ such that $n=2k+1+x$ since $k\leq \frac{n-1}{2}$. Moreover, $x$ and $k$ can not be $0$ simultaneously since $n\geq 3$. Now it is routine to verify that $$a_k=k^2 + k^3 + 2 x + 2 k x + 2 k^2 x + x^2 + k x^2>0$$
 for any $0 \leq k \leq m$.

By Theorem \ref{n-char}, it suffices to show that the  polynomial
 \begin{align*}
  f_n(t)=\sum_{k=0}^{m} a_k\binom{m}{k}t^k
  \end{align*}
  has only negative zeros.
By straightforward computations, we get that
$$f_3(t)=2(t+4),\quad f_4(t)=5 (2 t+3),\quad f_5(t)=4(t+3) (3 t+2),\quad f_6(t)=29 t^2+76 t+35.$$
It is easy to see that each of $f_3(t),f_4(t),f_5(t)$ and $f_6(t)$ has only real zeros.

It remains to show that $f_n(t)$
has only negative zeros for any $n\geq 7$. Note that $a_k$, considered as a polynomial of $k$, can be expanded in the falling factorials basis as given below:
  \begin{align*}
a_k
&= (k)_3 -  (2n-6)\cdot (k)_2 + (n^2 - 6 n + 3) \cdot  (k)_1 + (n^2 - 1),
  \end{align*}
where $(k)_i=k(k-1)\cdots (k-i+1)$.
Thus, letting $g(t)=(1+t)^m$, we get
 \begin{align*}
  f_n(t)=t^3g'''(t) -  (2n-6)t^2g''(t)+(n^2 - 6 n + 3)tg'(t)+(n^2 - 1)g(t),
  \end{align*}
where $g'(t)$ (resp. $g''(t)$ or $g'''(t)$) are the first order (resp. the second order, or the third order) derivative of $g(t)$. Therefore, for $n\geq 7$ and hence $m\geq 3$, we have $$f_n(t)=(1+t)^{m-3}h(t),$$ where
 \begin{align*}
  h(t)=&m(m-1)(m-2)t^3 -  m(m-1)(2n-6)t^2(1+t)\\
  &+m(n^2 - 6 n + 3)t(1+t)^2+(n^2 - 1)(1+t)^3.
  \end{align*}
It suffices to show that $h(t)$ has only negative zeros.
By the definition of $h(t)$, we know that $h(t)$ is a  cubic polynomial with the leading coefficient equal to $a_m>0$.
Therefore,  we have  $h(-\infty)=-\infty$. By straightforward computations, we get that
\begin{align*}
 h(-2)&=-8 m^3+8 m^2 n-2 m n^2+4 m n+2 m-n^2+1,\\
 h(-1)&=-(m-2) (m-1) m,\\
 h(0)&=(n-1) (n+1).
\end{align*}
When $n$ is odd, we have $m=\frac{n-1}{2}$  and $h(-2)=(n-1)^2$.
When $n$ is even, we have $m=\frac{n-2}{2}$  and  $h(-2)=(n-6) (n-1)+1$.
Since  $h(-\infty)=-\infty,h(-2)>0,h(-1)<0$ and $h(0)>0$ by $n\geq 7$ and $m\geq 3$, from the intermediate value theorem it follows that $h(t)$ has three distinct negative zeros.
Thus, all zeros of $f_n(t)$ are real and have the same sign for any $n\geq 7$. This completes the proof.
\end{proof}

Secondly, we have the following result.

\begin{lem}\label{lem-ms-wheel}
  For any $n\geq 3$ the sequence
  $\{b_k\}_{k=0}^{\infty}$ given by \eqref{seqb} is a multiplier sequence.
\end{lem}

\begin{proof}
This immediately follows from Lemma \ref{ms-shift} and Lemma \ref{ms-rev}.
\end{proof}

Thirdly, to prove the real-rootedness of $P_{W_n}(t)$, we also need the following result.

\begin{lem}\label{lucas}
For any $n\geq 3$ the polynomial
\begin{align}
 f_n(t)=\sum _{k=0}^{m} c_k t^k
\end{align}
has only real zeros, where
$m=\lfloor \frac{n-1}{2}\rfloor$ and
$c_k$ is given by \eqref{seqc}.
 \end{lem}

 \begin{proof}
 Let  $L_n(t)$ be the $n$-th  Lucas polynomial.
 It is well know that
 $$L_n(t)=\sum_{k=0}^{\lfloor \frac{n}{2} \rfloor} \frac{ n (n-k-1)!}{k! (n-2 k)! } t^{n-2k}$$
 and  $L_n(t)$ has only pure imaginary zeros for $n\geq 2$
 see \cite{hoggatt1973roots,lachal2013trick}. Moreover, it is straightforward to verify that
 $$t^{(n-1)/2} L_{n-1}\left(t^{-1/2}\right)=f_n(t).$$
 Thus $f_n(t)$ has only real zeros. The proof is complete.
 \end{proof}

  Now we are in the position to prove Theorem \ref{thm-wheel-zeros}.

  \begin{proof}[Proof of Theorem \ref{thm-wheel-zeros} ]
  By \eqref{eq-hadmard}, we have
    $$P_{W_n}(t)=\sum_{k=0}^{\lfloor \frac{n-1}{2}\rfloor}a_kb_kc_kt^k.$$
  By Lemma \ref{lem-ms-wheel} and Lemma \ref{lucas}, we obtain that
  the polynomial
  $$\sum_{k=0}^{\lfloor \frac{n-1}{2}\rfloor}b_kc_kt^k$$
  has only real zeros. Then by Lemma \ref{wheel-n-seq} we get that $P_{W_n}(t)$
  has only real zeros. This completes the proof.
  \end{proof}

The third main result of this section is as follows, which
states the real-rootedness of the \KL polynomials of whirl matroids.

\begin{thm}\label{thm-whirl-zeros}
For any $n\geq 3$ the  polynomial $P_{W^n}(t)$ has only negative zeros.
\end{thm}

In order to use the theory of multiplier sequences to prove the real-rootedness of $P_{W^n}(t)$, we shall rewrite \eqref{eq-klpol-whirl} as
\begin{align}\label{rewrite-whirl}
P_{W^n}(t)&=\sum _{k=0}^{\left\lfloor \frac{n-1}{2}\right\rfloor }{\frac{n!}{k!(n-k)!}\binom{n - k - 1 }{k} t^k}.
\end{align}

By a similar proof to that of Lemma \ref{lem-ms-wheel}, we obtain the following result.

\begin{lem}\label{lem-ms-whirl}
  For any $n\geq 3$ the sequence
  $\{\frac{n!}{k!(n-k)!}\}_{k=0}^{\infty}$ is a multiplier sequence.
\end{lem}

To prove Theorem \ref{thm-whirl-zeros}, we also  need the  following lemma, which is similar to Lemma \ref{lucas}.

\begin{lem}\label{fib}
For $n\geq 3$, the polynomial
\begin{align}
 g_n(t)=\sum _{k=0}^{\left\lfloor \frac{n-1}{2}\right\rfloor }\binom{n - k - 1 }{k}t^k
\end{align}
has only real zeros.
 \end{lem}

 \begin{proof}
 Let  $F_n(t)$ be the $n$-th  Fibonacci  polynomial.
It is well know that
 $$F_n(t)=\sum_{k=0}^{\lfloor \frac{n-1}{2} \rfloor} \binom{n-k-1}{k} t^{n-2k-1}$$
 and $F_n(t)$ has only pure imaginary zeros  for $n\geq 3$,
 see \cite{hoggatt1973roots,lachal2013trick}.
 Moreover, it is straightforward to verify that
 $$t^{(n-1)/2} F_{n}\left(t^{-1/2}\right)=g_n(t).$$
 Thus $g_n(t)$ has only real zeros. The proof is complete.
\end{proof}

We proceed to prove Theorem \ref{thm-whirl-zeros}.

\begin{proof}[Proof of Theorem \ref{thm-whirl-zeros} ]
This is an immediate corollary of Lemma \ref{lem-ms-whirl} and Lemma \ref{fib} in view of \eqref{rewrite-whirl}.
\end{proof}

\section{\texorpdfstring{$Z$}{Z}-polynomials}\label{sec-5}

The aim of this section is to prove Theorems \ref{zcoef} and \ref{zroot}. We shall first determine the $Z$-polynomials of fan graphs, wheel graphs and whirl matroids and then prove their real-rootedness.

To determine the $Z$-polynomials, we will also use the method of generating functions as in Section \ref{sec-3}, though it will be slightly different from the previous arguments. To make this point clear, let us compare the defining relation of \KL polynomials with that of $Z$-polynomials.

Recall that the \KL polynomials  of a matroid  $M$ satisfy  the following relation:
\begin{align*}
t^{\rk M}P_M(t^{-1}) =\sum_{F \in L(M)} \chi_{M_F}(t) P_{M^F}(t),
\end{align*}
while the $Z$-polynomials of $M$ is defined by
\begin{align*}
Z_M(t):= \sum_{F \in L(M)}{t^ {\rk M_F} P_{M^F}(t)}.
\end{align*}
Suppose that $\{M_d,M_{d+1},M_{d+2},\ldots\}$ is a sequence of matriods with $\rk M_n=n$ for $n\geq d$. The key idea to determine the generating function
\begin{align*}
\Phi(t,u)&=\sum_{n=d}^{\infty}{P_{M_n}(t)u^n},
\end{align*}
is to interpret
\begin{align*}
\Phi(t^{-1},tu)&=\sum_{n=d}^{\infty}\left( t^{\rk M_n}P_{M_n}(t^{-1})\right)u^n,
\end{align*}
namely, the summation
\begin{align}\label{summation-1}
\sum_{n=d}^{\infty}\left(\sum_{F \in L(M_n)}   \chi_{(M_n)_F}(t) P_{(M_n)^F}(t)\right)u^n
\end{align}
by the defining relation of $P_{M_n}(t)$,
as certain generating function of weighted combinatorial structures. By the previous arguments in Section \ref{sec-3},
we know that the method of generating functions is applicable because both $\chi_{M}(t)$ and $P_{M}(t)$ are multiplicative on direct sums of matroids.

In order to use the method of generating functions
to determine
$$Z(t,u)=\sum_{n=d}^{\infty}{Z_{M_n}(t)u^n},$$
by the defining relation, it is desirable to interpret
\begin{align}\label{summation-2}
\sum_{n=d}^{\infty}\left(\sum_{F \in L(M_n)}
t^{\rk (M_n)_F} P_{(M_n)^F}(t)\right)u^n
\end{align}
as certain generating function of weighted combinatorial structures, which might be applicable since $t^{\rk M}$ is obviously multiplicative on direct sums of matriods.
As shown below, this is indeed doable for fan graphs, wheel graphs and whirl matroids. Since the index set of \eqref{summation-1} is the same as that of \eqref{summation-2},
we can use the same combinatorial structures for \KL polynomials to interpret the latter summation, while $t^{\rk M}$ will take the role of $\chi_{M}(t)$ in the corresponding weight functions.

\subsection{Fan graphs}\label{subsection-zpol-fan}

The main objective of this subsection is to determine the $Z$-polynomials of fan graphs. Let
\begin{align}\label{eqn-fanz-gf}
Z_F(t,u) := \sum_{n=0}^\infty Z_{F_n}(t) u^{n},
\end{align}
where $F_0$ is the single-vertex graph.
We have the following result.

\begin{thm}\label{thm-zpol-fan}
We have
\begin{align}\label{eq-gf-zpol-fan}
Z_F(t,u)=\frac{2}{\sqrt{(1-(t+1) u)^2-4 t u^2}-(t+1) u+1}.
\end{align}
\end{thm}

The proof of \eqref{eq-gf-zpol-fan} is very similar to that of \eqref{eq-gf-klpol-fan}. By the proceeding arguments, it suffices to establish a parallel result to Lemma \ref{fan-eqn}.
Recall that, to prove Lemma \ref{fan-eqn}, we introduce a sequence of combinatorial structures in Subsection \ref{sec:non-ekl}, and obtain many intermediate results.
Here we will introduce some parallel objects and state some parallel results without proofs.

Let $\mathcal{C}^{'}_n$ be defined as in \eqref{eq-cnprime}.
Parallel to \eqref{eq-weight-function}, we define another weight function on $\mathcal{C}^{'}_n$ by
\begin{align}
\tilde{w}(A)=\prod_{i=1}^{k}  t^{\rk  F_{a_{2i-1}}} \cdot P_{F_{\ell_i}}(t) \cdot\prod_{j=1}^{\ell_i}  t^{\rk H_{b_{ij}}},
\end{align}
for $A=(A_1,A_2,\ldots,A_{2k-1},A_{2k}) \in \mathcal{C}'_n$,
where $A_{2i-1}=(a_{2i-1}),A_{2i}=(b_{i1},b_{i2},\ldots,b_{i\ell_i})\in\mathcal{S}_{a_{2i}}$ for $1\leq i\leq k$.

Parallel to Lemma \ref{wterm}, we have the following result.

\begin{lem}\label{wterm-zpol}
For any $C\in \mathcal{C}(F_n)$, we have
\begin{align}\label{eq-transform-zpol}
{t^{\rk F_n[C]} P_{F_n/C}(t)} =\tilde{w}(\phi(C)),
\end{align}
where $\phi(C)$ is defined as in the  Lemma \ref{combij}.
\end{lem}

Recall that $\mathcal{A}$ is the disjoint union of $\mathcal{C}^{'}_n$ structures weighted by ${w}(A)$.
Parallel to that, we let $\tilde{\mathcal{A}}$ be the disjoint union of $\mathcal{C}^{'}_n$ structures weighted by $\tilde{w}(A)$. By \ref{zpol-reform}  we have
\begin{align*}
Z_F(t,u)=\sum_{n=0}^{\infty}
\left(\sum_{C\in \mathcal{C}(F_n)} {t^{\rk F_n[C]}(t) P_{F_n/C}(t)}\right)u^n.
\end{align*}
The above lemma implies that $Z_F(t,u)$ can be considered as the generating function of type $\tilde{\mathcal{A}}$ structures. Precisely, we have
\begin{align*}
Z_F(t,u)=\sum_{n= 0}^{\infty} \left(\sum_{A\in \mathcal{C}'_n}\tilde{w}(A)\right)u^n.
\end{align*}

To compute the right hand side of the above equation, we now
define combinatorial structures of type $\tilde{\mathcal{A}}^o,\tilde{\mathcal{A}}^e,\tilde{\mathcal{A}}^{eo}$ and  $\tilde{\mathcal{A}}^m$, which are respectively parallel to those combinatorial structures of type $\mathcal{A}^o,\mathcal{A}^e,\mathcal{A}^{eo}$ and  $\mathcal{A}^m$. Precisely, type $\tilde{\mathcal{A}}^o$ structure will assign to an interval of size $n$ the weak composition $(n)$ with the weight function $\tilde{w}^o$ defined by $$\tilde{w}^o((n))= t^{\rk  F_{n}},$$ and type $\tilde{\mathcal{A}}^e$ structure will assign to an interval of size $n$ a composition $(b_1,\ldots,b_k)\in\mathcal{S}_n$
with the weight function $\tilde{w}^e$ defined by
$$\tilde{w}^e((b_1,\ldots,b_k))=P_{F_{k}}(t) \cdot\prod_{j=1}^{k} t^{\rk H_{b_j}}.$$
Note that the unique $\tilde{\mathcal{A}}^o$ structure of size $0$ is $(0)$, weighted by $1$, and the unique $\tilde{\mathcal{A}}^e$ structure of size $0$ is $(\,)$, also weighted by $1$.
Let $\tilde{\mathcal{A}}^{eo}$ be the set of pairs $(\tilde{A}^e,\tilde{A}^o)$, where $\tilde{A}^e$ is a structure of type $\tilde{\mathcal{A}}^e$, and $A^o$ is of type $\tilde{\mathcal{A}}^o$, and moreover neither  $\tilde{A}^e$ nor $\tilde{A}^o$ is empty. The weight function $\tilde{w}^{eo}$ of
$\tilde{\mathcal{A}}^{eo}$ is defined by
\begin{align*}
\tilde{w}^{eo}((\tilde{A}^e,\tilde{A}^o))=\tilde{w}^e(A^e)\tilde{w}^o(A^o).
\end{align*}
Let $\tilde{\mathcal{A}}^{m}$ be the set of combinatorial structures each of which is a sequence $(\tilde{A}^{eo}_1,\ldots,\tilde{A}^{eo}_k)$ of $\tilde{\mathcal{A}}^{eo}$ structures. Define the weight function $\tilde{w}^m$ of $\tilde{\mathcal{A}}^{m}$ as
$$\tilde{w}^{m}((\tilde{\mathcal{A}}^{eo}_1,\ldots,\tilde{\mathcal{A}}^{eo}_k))=\prod_{i=1}^k \tilde{w}^{eo}(\tilde{\mathcal{A}}^{eo}_i).$$
Correspondingly, let $\tilde{\mathcal{A}}^o_n$ (resp. $\tilde{\mathcal{A}}^e_n$ or  $\tilde{\mathcal{A}}^{eo}_n$ or $\tilde{\mathcal{A}}^m_n$ ) denote the set of
type $\tilde{\mathcal{A}}^o$ (resp. $\tilde{\mathcal{A}}^e$) or $\tilde{\mathcal{A}}^{eo}$ or  $\tilde{\mathcal{A}}^m$)  structures which can be built on an interval of size $n$.

Let
\begin{align*}
\tilde{\Psi}^o(u)&=\sum_{n= 1}^{\infty} \left(\sum_{\tilde{A}^o\in \tilde{\mathcal{A}}^o_n}\tilde{w}^o(A^o)\right)u^n,\\
\tilde{\Psi}^e(u)&=\sum_{n= 1}^{\infty} \left(\sum_{\tilde{A}^e\in \tilde{\mathcal{A}}^e_n}\tilde{w}^e(A^e)\right)u^n,\\
\tilde{\Psi}^{eo}(u)&=\sum_{n= 0}^{\infty} \left(\sum_{\tilde{A}^{eo}\in \tilde{\mathcal{A}}^{eo}_n}\tilde{w}^{eo}(\tilde{\mathcal{A}}^{eo})\right)u^n,\\
\tilde{\Psi}^m(u)&=\sum_{n= 0}^{\infty} \left(\sum_{\tilde{A}^m\in \tilde{\mathcal{A}}^m_n}\tilde{w}^m(\tilde{A}^m)\right)u^n.
\end{align*}

Parallel to Lemma \ref{gen-o and e}, we have the following result.

\begin{lem}\label{lem-gf-o-e}
We have
\begin{align*}
\tilde{\Psi}^o(u)&=\frac{t u}{1-t u},\\
\tilde{\Psi}^e(u)&=\Phi_F\left(t, \frac{u}{1-t u}\right)-1,
\end{align*}
where $\Phi_F(t,u)$ is the generating function of \KL polynomials for fan graphs.
\end{lem}

Parallel to Lemma \ref{gen-oe}, we have the following result.

\begin{lem}\label{lem-gf-m}
We have
\begin{align*}
\tilde{\Psi}^{eo}(u)&=\tilde{\Psi}^{e}(u)\tilde{\Psi}^{o}(u),\\
\tilde{\Psi}^m(u)&=\frac{1}{1-\tilde{\Psi}^{eo}(u)}.
\end{align*}
\end{lem}

Finally, parallel to Lemma \ref{fan-eqn}, we obtain the following result.

\begin{lem}\label{eq-another-zpol}
We have
\begin{align*}
Z_F(t,u)&=\frac{(1+\tilde{\Psi}^o(u))(1+\tilde{\Psi}^e(u))}{1-\tilde{\Psi}^{e}(u)\tilde{\Psi}^{o}(u)}.
\end{align*}
\end{lem}

We proceed to prove Theorem \ref{thm-zpol-fan}.

\begin{proof}[Proof of Theorem \ref{thm-zpol-fan}.]
In Subsection \ref{sec:non-ekl} we already determined the value of $\Phi_F(t,u)$. Combining Lemmas \ref{lem-gf-o-e}, \ref{lem-gf-m} and \ref{eq-another-zpol}, it remains to verify the equality based on \eqref{eq-gf-klpol-fan}, which can be completed with the aid of Mathematica.
%

{\renewcommand\baselinestretch{2}
\begin{mma}
\In \tilde{\Psi }^o[u\_]|:=|\frac{t u}{1-t u}; \\
\end{mma}
\begin{mma}
\In \tilde{\Psi }^e[u\_]|:=|\phi _F\left[\frac{u}{1-t u}\right]-1; \\
\In \tilde{\Psi }^{eo}[u\_]|:=|\tilde{\Psi }^e[u] \tilde{\Psi }^o[u];\\
\In Z_{F}[t\_,u\_] |:=|\frac{(1+\tilde{\Psi }^e[u])  (\tilde{\Psi }^o[u]+1)}{1-\tilde{\Psi }^{eo}[u]};\\
\end{mma}
}
\begin{mma}
\In |FullSimplify|[Z_F(t,u)==\frac{2}{\sqrt{(1-(t+1) u)^2-4 t u^2}-(t+1) u+1},\linebreak|Assumptions|\to 1-t u>0]\\
\Out |True|;\\
\end{mma}
This completes the proof.
\end{proof}

Now we are able to prove \eqref{eq-zpol-fan-sect1} of Theorem \ref{zcoef}.

\begin{proof}[Proof of \eqref{eq-zpol-fan-sect1}.]
The equivalence  between  \eqref{eq-zpol-fan-sect1} and   \eqref{eq-gf-zpol-fan} is well  known, see  \cite[Section 4]{coker2003enumerating}. The proof is complete.
\end{proof}

\subsection{Wheel graphs}

The main objective of this subsection is to determine the $Z$-polynomials of wheel graphs.
Let
\begin{align}\label{eqn-wheelz-gf}
Z_W(t,u) := \sum_{n=2}^\infty Z_{W_n}(t) u^{n},
\end{align}
where $W_2$ is a simple circle with three vertices.

We have the following result.
\begin{thm}\label{thm-zpol-wheel}
We have
\begin{align}\label{eq-gf-zpol-wheel}
Z_W(t,u)=&-\frac{2 u (1-(t+1) u) (t (u+1)+1)}{1-(t+1) u-2 t u^2+\sqrt{(1-(t+1) u)^2-4 t u^2}}\nonumber\\[5pt]
&+\frac{1}{\sqrt{(1-(t+1) u)^2-4 t u^2}}-1.
\end{align}
\end{thm}

The proof of \eqref{eq-gf-zpol-wheel} is very similar to that of \eqref{eq-gf-klpol-wheel}.
By \ref{zpol-reform} we have
\begin{align*}
Z_W(t,u)=\sum_{n=2}^{\infty}
\left(\sum_{C\in \mathcal{C}(W_n)} {t^{\rk W_n[C]} P_{W_n/C}(t)}\right)u^n.
\end{align*}
Recall that the proof of \eqref{eq-gf-klpol-wheel} is based on the computation of the following summation
\begin{align*}
\sum_{n=2}^{\infty}
\left(\sum_{C\in \mathcal{C}(W_n)}t^{-|C|}\chi_{W_n[C]}(t) P_{W_n/C}(t)\right)u^n.
\end{align*}

For that purpose we divide $\mathcal{C}(W_n)$ into three disjoint parts  $\mathcal{C}^{\langle 1\rangle}(W_n),\mathcal{C}^{\langle 2\rangle}(W_n)$ and $\mathcal{C}^{\langle 3\rangle}(W_n)$, and introduce $\Psi_1(u),\,\Psi_2(u),\,\Psi_3(u)$ as defined in \eqref{eq-psi-123}. Again by the arguments immediately before Subsection \ref{subsection-zpol-fan}, we may introduce parallel functions to compute $Z_W(t,u)$. Correspondingly, for $i=1,2,3$ let
\begin{align}\label{eq-z-wheel-123}
Z_W^{\langle i\rangle}(u)=\sum_{n=2}^{\infty}\left(\sum_{C\in \mathcal{C}^{\langle i\rangle}(W_n)}{t^{\rk W_n[C]} P_{W_n/C}(t)}\right)u^n.
\end{align}
Then
$$Z_W(t,u)=Z_W^{\langle 1\rangle}(u)+Z_W^{\langle 2\rangle}(u)+Z_W^{\langle 3\rangle}(u).$$

In the following we shall determine $Z_W^{\langle 1\rangle}(u),Z_W^{\langle 2\rangle}(u)$ and $Z_W^{\langle 3\rangle}(u)$ successively.

Parallel to Lemma \ref{w1}, we have the following result.

\begin{lem}\label{z-w1} We have
\begin{align*}
Z_W^{\langle 1\rangle}(u)=\frac{t^2 u^2}{1-t u}+\frac{t u^2}{1-t u}.
\end{align*}
\end{lem}
\begin{proof}
Recall that $\mathcal{C}^{\langle 1\rangle}(W_n) = \{\{[0,n]\},\{\{0\},[1,n]\} \}$. Since $\rk W_n[\{[0,n]\}]=n$ and $\rk W_n[\{\{0\},[1,n]\}]=n-1$, we get the desired result.
\end{proof}

Recall that a composition $C \in \mathcal{C}^{\langle 2\rangle}(W_n)$ naturally decomposes the outer cycle of $W_n$  into $|C|-1$ paths.  Suppose that these paths have lengths $i_1,i_2,\ldots,i_{|C|-1}$ respectively. Parallel to \eqref{w2-weight}, we can  prove that
$${t^{\rk W_n[C]} P_{W_n/C}(t)} =  P_{W_{|C|-1}}(t)\times \prod_{j=1}^{|C|-1}{t ^{\rk  H_{i_j} }}.$$
Based on  this fact, $Z_W^{\langle 2\rangle}(u)$  could be considered as the ordinary generating function of type $\mathcal{B} \bullet  \mathcal{A}$ structures, with $\mathcal{A}_n=\{(n)\}$ for $n\geq 1$ weighted by $w_{\mathcal{A}}((n))= t^{\rk  H_{n}}$ and $\mathcal{B}_n=\{(n)\}$ for $n\geq 2$ weighted by $w_{\mathcal{B}}((n))=P_{W_n}(t)$. (Here we assume that $\mathcal{A}_0, \mathcal{B}_0$ and $\mathcal{B}_1$ are empty.)
Parallel to Lemma \ref{w2}, we have the following result.

\begin{lem}\label{z-w2}
We have
\begin{align*}
Z_W^{\langle 2\rangle}(u)&=\frac{1}{1-tu}\times \Phi_W\left(t,\frac{u}{1-tu}\right),
\end{align*}
where $\Phi_W(t,u)$ is the generating function of the \KL polynomials for wheel graphs.
\end{lem}

We proceed to determine $Z_W^{\langle 3\rangle}(u)$.
Parallel to Lemma \ref{w3}, we obtain the following result by replacing type $\mathcal{A}^{eo}$ structures by type $\tilde{\mathcal{A}^{eo}}$ structures.

\begin{lem}\label{z-w3}
We have
\begin{align*}
Z_W^{\langle 3\rangle}(u)&= {u \frac{\partial \tilde{\Psi} _{eo}(u)}{\partial u}}\frac{1}{1-\tilde{\Psi} _{eo}(u)},
\end{align*}
where $\tilde{\Psi} _{eo}(u)$ is defined as in Lemma \ref{lem-gf-m}.
\end{lem}

Combining  Lemmas \ref{z-w1},\ref{z-w2} and \ref{z-w3}, we get the following result.

\begin{lem}\label{eq-another-zpol-wheel}
Let $\Phi_W(t,u)$ be given by \eqref{eq-gf-klpol-wheel}.
Then
\begin{align*}
Z_W(t,u)&=\frac{t^2 u^2}{1-t u}+\frac{t u^2}{1-t u} +\frac{\Phi_W\left(t,\frac{u}{1-t u}\right)}{1-t u}+{u \frac{\partial \tilde{\Psi} _{eo}(u)}{\partial u}} \frac{1}{{1-\tilde{\Psi} _{eo}(u)}}.
\end{align*}
\end{lem}

Our proof of Theorem \ref{thm-zpol-wheel} is as follows.

\begin{proof}[Proof of Theorem \ref{thm-zpol-wheel}]
In Subsection \ref{subsect-wheel} we already determined the value of $\Phi_W(t,u)$. Combining Lemmas \ref{lem-gf-m} and \ref{eq-another-zpol-wheel}, it remains to verify the equality based on \eqref{eq-gf-klpol-wheel}, which can be completed with the aid of Mathematica.

\begin{mma}
\In Z_W[t\_,u\_]:=\frac{t^2 u^2}{1-t u}+\frac{t u^2}{1-t u}+\frac{1}{1-t u}\Phi _W\left[\frac{u}{1-t u}\right]+u \frac{\partial \tilde{\Psi }^{|eo|}(u)}{\partial u}\frac{1}{1-\tilde{\Psi }^{|eo|}(u)};\\
\In |T1|[u\_]:=\frac{1}{\sqrt{(1-(t+1) u)^2-4 t u^2}};\\
\end{mma}
\begin{mma}
\In |T2|[u\_]:=\frac{2 u (1-(t+1) u) (t (u+1)+1)}{1-(t+1) u-2 t u^2+\sqrt{(1-(t+1) u)^2-4 t u^2}};\\
\end{mma}
\begin{mma}
\In |Simplify|[Z_W(t,u)=|T1|[u]-|T2|[u]-1,|Assumptions|\to 1-t u>0];\\
\Out |True|\\
\end{mma}
This completes the proof.
\end{proof}

Now we are able to prove \eqref{eq-zpol-wheel-sect1} of Theorem \ref{zcoef}.

\begin{proof}[Proof of \eqref{eq-zpol-wheel-sect1}.]
It   suffices to show that
\begin{align}
 \sum_{n=2}^{\infty}{ \left ( \sum _ {k = 0}^n\binom {n} {k}^2  t^k\right) u^n}=&
 \frac{1}{\sqrt{(1-(t+1) u)^2-4 t u^2}}-1-(1+t)u \label{binsq},\\[5pt]
 \sum _{n=2}^{\infty}\left(\sum _{k=0}^n\frac {2}{n} \binom {n} {k +1}\binom {n} {k - 1} \right)u^n=&\frac{2 u (1-(t+1) u) (t (u+1)+1)}{1-(t+1) u-2 t u^2+\sqrt{(1-(t+1) u)^2-4 t u^2}}\notag \\
 &-(1+t)u.\label{z_wheel_eq2}
\end{align}
Note that \eqref{binsq}  is equivalent to
\begin{align}\label{binsq2}
\sum_{n=0}^{\infty}{ \left ( \sum _ {k = 0}^n\binom {n} {k}^2  t^k\right) u^n}=\frac{1}{\sqrt{(1-(t+1) u)^2-4 t u^2}},
\end{align}
 which can be proved  by using the generating function of Legendre polynomials. But here we give a proof with the aid of  Mathematica.
%

\begin{mma}
\In  a[n\_, k\_] :=|Binomial[n,k]|^2;\\
\end{mma}
\begin{mma}
\In |ReleaseHold|[|First|[|Zb|[|FunctionExpand|[a[n, k]] t^k, {k, 0, n}, n] /. |SUM| \to  a]];  \\
\end{mma}
\begin{mma}
\In|Simplify|[|FunctionExpand|[\%], |Assumptions| \to n  \in \mathbb{Z}]; \\
\end{mma}
\begin{mma}
\In  rec = |Collect|[\%, a[\_], |Factor|] \\
\Out (1+n) (-1+t)^2 a[n]+(2+n) a[2+n]==(3+2n) (1+t) a[1+n] \\
\In |Table|[a[n] == |Sum|[a[n, k] t^k, {k, 0, n}], {n, 0, 1}]\\
\Out \{a[0] == 1, a[1] == 1+t\};\\
\In de=|RE2DE|[\{rec,a[0]==1,a[1]==1+t\}, a[n], f[u]]\\
\Out   \{ ((t-1)^2 u^2-2 (t+1) u+1 ) f'[u]+f[u]  ((t-1)^2 u-t-1 )==0,f[0]==1\}  \\
\In |Simplify|[|de|\text{/.}\, f\to |T1|]\\
\Out \{|True|,|True|\} \\
\end{mma}
It is easy to see that  \eqref{z_wheel_eq2} is equivalent to the following equation:
\begin{align*}
\sum _{n=0}^{\infty}\left(\sum _{k=0}^n   \frac {2}{n+2} \right.&\left. \binom {n+2} {k +2}\binom {n+2} {k }  t^k\right)u^n\\[5pt]
&=\frac{2  (1-(t+1) u) (t (u+1)+1)}{u (1-(t+1) u-2 t u^2+\sqrt{(1-(t+1) u)^2-4 t u^2})}
-\frac{(1+t)u}{u^2}.
\end{align*}
Then it can be proved by the following  Mathematica codes.
%
%
%
%
%

\begin{mma}
\In |T3|[u\_]:=\frac{2t}{1-2 (t+1) u+(t^2+1) u^2+(1-(t+1) u) \sqrt{1-(t+1) u)^2-4 t u^2}};\\
\end{mma}
\begin{mma}
\In |Simplify|[|T3|[u] ==\frac{|T2|[u] }{u^2} -\frac{ (1 + t) u}{u^2}];\\
\Out |True|\\
\end{mma}
\begin{mma}
\In  a[n\_, k\_] :=\frac{2}{n+2}|Binomial|[n +2,k+2] |Binomial|[n +2,k ];\\
\end{mma}
\begin{mma}
\In |ReleaseHold|[|First|[|Zb|[|FunctionExpand|[a[n, k]] t^k, {k, 0, n}, n] /. |SUM| \to a]];  \\
\end{mma}
\begin{mma}
\In|Simplify|[|FunctionExpand|[\%], |Assumptions| \to n  \in \mathbb{Z}]; \\
\end{mma}
\begin{mma}
\In  |rec| = |Collect|[\%, a[\_], |Factor|] \\
\Out (2 + n) (3 + n) (-1 + t)^2 a[n] + (2 + n) (6 + n) a[2 + n] == (3+n) (7 + 2 n) (1 + t) a[1 + n] \\
\In |Table|[a[n] == |Sum|[a[n, k] t^k, {k, 0, n}], {n, 0, 1}]\\
\Out \{a[0] == 1, a[1] == 2+2t\};\\
\In |de|=|RE2DE|[\{|rec|,a[0]==1,a[1]==2+2t\}, a[n], f[u]]\\
\Out  \{ 2 (-5 + 3 u + 3 t^2 u - t (5 + 6 u)) f[u] + (5 - 11 u - 11 t u + 6 u^2 - 12 t u^2 + 6 t^2 u^2) f^{'}[u]+(u - 2 u^2 - 2 t u^2 + u^3 - 2 t u^3 + t^2 u^3) f^{''}[u] == 0, f[0] == 1,f'[0]== 2(1+t) \} \\
\In |Simplify|[|de|/.f\to |T3|]\\
\Out \{|True|,|True|,|True|\} \\
\end{mma}
This completes the proof.
\end{proof}

\subsection{Whirl matroids}
The main objective of this subsection is to determine the $Z$-polynomials of whirl  matriods.
Let
\begin{align}\label{eqn-whirlz-gf}
Z^W(t,u) := \sum_{n=1}^\infty Z_{W^n}(t) u^{n},
\end{align}
where matroid   $W^1$ and matroid  $W^2$ are defined as in the subsection \ref{subsect-whirl}.

We have the following result.
\begin{thm}\label{thm-zpol-whirl}
We have
\begin{align}\label{eq-gf-zpol-whirl}
  Z^W(t,u)=\frac{1}{\sqrt{(1-(t+1) u)^2-4 t u^2}}-1.
\end{align}
\end{thm}

The proof of \eqref{eq-gf-zpol-whirl} is very similar to that of \eqref{eq-gf-zpol-wheel}. We will prove
\eqref{eq-gf-zpol-whirl} along the lines of proving
\eqref{eq-gf-klpol-whirl}, just as we prove
 \eqref{eq-gf-zpol-wheel} along the lines of proving \eqref{eq-gf-klpol-wheel}.
By \eqref{zpol-reform}, we have
\begin{align*}
Z^W(t,u)=\sum_{n=1}^{\infty}
\left(\sum_{F\in L(W^n) }{{\rk (W^n)_F}(t) P_{(W^n)^F}(t)}  \right)u^n.
\end{align*}
Recall that the proof of \eqref{eq-gf-klpol-whirl} is based on the computation of the following summation
\begin{align}
\sum_{n=1}^{\infty}\left(\sum_{F\in L(W^n) }{\chi_{(W^n)_F}(t) P_{(W^n)^F}(t)}  \right)u^n.
\end{align}
For that purpose we divide $L(W^n)$ into two  disjoint parts  $L_1(W^n)$ and $L_2(W^n)$, and introduce $\Psi^W_1(u)$ and $\Psi^W_2(u)$ as defined in \eqref{eq-whirl-l1} and \eqref{eq-whirl-l2}. To determine $\Psi^W_2(u)$, we also divide $L_2(W^n)$ into three  disjoint parts  $L_2^{\langle 1\rangle}(W^n), L_2^{\langle 2\rangle}(W^n)$ and $L_2^{\langle 3\rangle}(W^n)$, and introduce $\Psi^W_{2,1}(u),\Psi^W_{2,2}(u)$ and $\Psi^W_{2,3}(u)$ as defined in \eqref{eq-whirl-gf-l2-123}.

Again by the arguments immediately before Subsection \ref{subsection-zpol-fan}, we may introduce parallel functions to compute $Z^W(t,u)$. Let
 \begin{align*}
Z^W_1(u)&=\sum_{n=1}^{\infty}\left( \sum_{F \in L_1(W^n) }{t^{\rk\,(W^n)_F}  P_{(W^n)^F}(t)} \right) u^n,\\[5pt]
Z^W_2(u)&=\sum_{n=1}^{\infty}\left(\sum_{F \in L_2(W^n) }{t^{\rk\,(W^n)_F}  P_{(W^n)^F}(t)}\right)u^n.
\end{align*}
Then
$$Z^W(t,u)=Z^W_1(u)+Z^W_2(u).$$
Accordingly, for $i=1,2,3$ let
\begin{align}\label{eq-z-whirl-gf-l2-123}
Z^W_{2,i}(u)=\sum_{n=1}^{\infty}\left( \sum_{F \in L_2^{\langle i\rangle}(W^n)  }{t^{\rk\,(W^n)_F}  P_{(W^n)^F}(t)} \right) u^n.
\end{align}
Then we have
$$Z^W_2(u)=Z^W_{2,1}(u)+Z^W_{2,2}(u)+Z^W_{2,3}(u).$$

Firstly, parallel to Lemma \ref{whirl1}, we obtain the following result.
\begin{lem}\label{z-whirl1}
We have
\begin{align*}
Z^W_1(u)&=\frac{1}{1-t u} \times \frac{u}{1-tu}.
\end{align*}
\end{lem}
\begin{proof}
Recall that  there are $n$ flats of rank $n-1$ in  $L_1(W^n)$.
Therefore, we have
\begin{align*}
Z^W_1(u)&=\sum_{n=1}^{\infty}{n t^{n-1} u^n}=\frac{1}{1-t u} \times \frac{u}{1-tu}.
\end{align*}
This completes the proof.
\end{proof}
Secondly, parallel to Lemmas \ref{whirl21}, \ref{whirl22} and \ref{whirl23}, we obtain the following results. The proofs are very similar to Lemmas \ref{z-w1},\ref{z-w2} and \ref{z-w3}, which will be omitted here.

\begin{lem}\label{z-whirl21}
We have
\begin{align*}
Z^W_{2,1}(u)&=\frac{tu}{1-t u}.
\end{align*}
\end{lem}

\begin{lem}\label{z-whirl22}
We have
\begin{align*}
Z^W_{2,2}(u)&=\frac{1}{1-tu} \left(\Phi^W\left(t,\frac{u}{1-tu}\right)-\frac{u}{1-tu}\right),
\end{align*}
where $\Phi^W(t,u)$ is the generating function of the \KL polynomials for whirl matroids.
\end{lem}

\begin{lem}\label{z-whirl23}
We have
\begin{align*}
Z^W_{2,3}(u)&={u \frac{\partial \tilde{\Psi} _{eo}(u)}{\partial u}} \frac{1}{{1-\tilde{\Psi} _{eo}(u)}},
\end{align*}
where $\tilde{\Psi} _{eo}(u)$ is defined as in Lemma \ref{lem-gf-m}.
\end{lem}

Combining  Lemmas \ref{z-whirl1},\ref{z-whirl23},\ref{z-whirl22} and \ref{z-whirl23}, we get the following result.

\begin{lem}\label{eq-another-zpol-whirl}
We have
\begin{align*}
Z^W(t,u)&=\frac{tu}{1-t u}+\frac{\Phi^W\left(t,\frac{u}{1-t u}\right)}{1-t u}+{u \frac{\partial \tilde{\Psi} _{eo}(u)}{\partial u}} \frac{1}{{1-\tilde{\Psi} _{eo}(u)}}.
\end{align*}
\end{lem}

Now we are able to prove Theorem \ref{thm-zpol-whirl}.

\begin{proof}[Proof of Theorem \ref{thm-zpol-whirl}.]
In Subsection \ref{subsect-whirl} we already determined the value of $\Phi^W(t,u)$. Combining Lemmas \ref{lem-gf-m} and \ref{eq-another-zpol-whirl}, it remains to verify the equality based on \eqref{eq-gf-klpol-whirl}, which can be completed with the aid of Mathematica.

%
%
\begin{mma}
\In Z^W[t\_,u\_]:=\frac{t u}{1-t u}+\frac{1}{1-t u}\Phi^W\left[\frac{u}{1-t u}\right]+u\frac{\partial \tilde{\Psi }^{|eo|}(u)}{\partial u}
\frac{1}{1-\tilde{\Psi}^{|eo|}(u)};\\
\end{mma}
\begin{mma}
\In |Simplify|[Z_W(t,u)==\frac{1}{\sqrt{(1-(t+1) u)^2-4 t u^2}}-1,|Assumptions|\to 1-t u>0];\\
\Out  |True|;\\
\end{mma}
This completes the proof.
\end{proof}

Finally, we are in the position to prove  \eqref{eq-zpol-whirl-sect1} of Theorem \ref{zcoef}.

\begin{proof}[Proof of \eqref{eq-zpol-whirl-sect1}.]
The equivalence between  \eqref{eq-zpol-whirl-sect1}  and  \eqref{eq-gf-zpol-whirl} is clear in view of \eqref{binsq2}.
\end{proof}

\subsection{Real zeros of \texorpdfstring{$Z$}{Z}-polynomials}

The main objective of this subsection is to prove Theorem \ref{zroot}, which states that the $Z$-polynomials of fan graphs, wheel graphs and whirl matroids are all real-rooted.

\begin{proof}[Proof of Theorem \ref{zroot}]
By Theorem \ref{zcoef} we see that the $Z$-polynomials of fan graphs are just the classical Narayana polynomials, and the polynomials of whirl matroids are just the Narayana polynomials of type $B$. It is well known that the Narayana polynomials are real-rooted.

It remains to prove the real-rootedness of the $Z$-polynomials of wheel graphs. For any $n\geq 3$, we may rewrite \eqref{eq-zpol-wheel-sect1} as
\begin{align*}
Z_{W_n}(t)&=\sum_{k = 0}^n ((1 + k) n^2 + (1 - 2 k - k^2) n + 2 k^2)) \frac{(n-1)!}{(k+1)!(n+1-k)!}  \binom {n} {k}t^k.
\end{align*}
It is easy to see that
\begin{align*}
 \sum_{k = 0}^n \frac{(n-1)!}{(k+1)!(n+1-k)!}  \binom {n} {k}t^k
\end{align*}
has only real roots for $n\geq 3$ since $\{\frac{(n-1)!}{(k+1)!(n+1-k)!}\}_{k=0}^{\infty}$ is a multiple sequence by a slight modification of Lemma \ref{lem-ms-wheel}. To prove the real-rootedness of
$Z_{W_n}(t)$, it suffices to show that
$$\{(1+k)n^2-(k^2-2k+1)n+2k^2)\}_{k=0}^{n}$$ is an $n$-sequence, which is clearly a positive sequence for $\geq 3$.
By Theorem \ref{n-char}, it suffices to show that the  polynomial
\begin{align*}
h_n(t)=\sum_{k=0}^{n}{\left((1 + k) n^2-(k^2-2k+1) n + 2 k^2)\right)}\binom{n}{k} t^k
\end{align*}
has only negative zeros.

Along the lines of the proof of Lemma \ref{wheel-n-seq}, we obtain that
\begin{align*}
 h_n(t)&=\sum_{k=0}^{n}{\left((2 - n)\cdot k(k - 1)+(n-1)(n-2)\cdot  k + n(n+1) \right)\binom{n}{k}t^k}\\[5pt]
 &=-n(n-1)(n-2)t^2(1+t)^{n-2}+ n(n-1)(n-2)t(1+t)^{n-1}+ n(n+1)(1+t)^n\\[5pt]
 &=\left(-n (n - 1) (n - 2) t^2  + n (n - 1) (n - 2) t (1 + t) +
 n (n + 1) (1 + t)^2\right)(1+t)^n\\[5pt]
 &=n \left((n+1)t^2+(n^2-n+4)t+n+1\right)(1+t)^n.
 \end{align*}
It is obvious that the quadratic polynomial  $$(n+1)t^2+(n^2-n+4)t+n+1$$ has only positive coefficients and its   discriminant
$$(n^2-n+4)^2-4 (n+1)^2=(n-1)(n-2)(n^2+n+6)>0.$$
Thus $(n+1)t^2+(n^2-n+4)t+n+1$ has two  negative roots and moreover $h_n(t)$ has only negative zeros, as desired. This completes the proof.
\end{proof}

\noindent \textit{Acknowledgements.}
This work was supported by the National Science Foundation of China. 

\end{document}